\documentclass{amsart}
\usepackage{amssymb}

\usepackage{amscd}
\usepackage{thmdefs}

\input{tcilatex}
\theoremstyle{definition}
\theoremstyle{remark}
\numberwithin{equation}{section}

\input tcilatex

\begin{document}
\title[Embeddings of weighted spaces and CKN inequalities]{Embeddings of weighted Sobolev spaces and generalized
Caffarelli-Kohn-Nirenberg inequalities}
\author{Patrick J. Rabier}
\address{Department of Mathematics, University of Pittsburgh, Pittsburgh, PA 15260}
\email{rabier@imap.pitt.edu}
\thanks{I am grateful to Prof. V. Maz'ya for valuable information about the early
historical development of the subject and other connections with the
existing literature. }
\maketitle

\begin{center}
(to appear in J. Analyse Math\'{e}matique)
\end{center}

\begin{abstract}
We characterize all the real numbers $a,b,c$ and $1\leq p,q,r<\infty $ such
that the weighted Sobolev space $W_{\{a,b\}}^{(q,p)}(\Bbb{R}^{N}\backslash
\{0\}):=\{u\in L_{loc}^{1}(\Bbb{R}^{N}\backslash \{0\}):|x|^{\frac{a}{q}%
}u\in L^{q}(\Bbb{R}^{N}),|x|^{\frac{b}{p}}\nabla u\in (L^{p}(\Bbb{R}%
^{N}))^{N}\}$ is continuously \thinspace embedded \thinspace into \linebreak 
$L^{r}(\Bbb{R}^{N};|x|^{c}dx)$\noindent $:=\{u\in L_{loc}^{1}(\Bbb{R}%
^{N}\backslash \{0\}):|x|^{\frac{c}{r}}u\in L^{r}(\Bbb{R}^{N})\},$ with norm 
$||\cdot ||_{c,r}.$

Except when $N\geq 2$ and $a=c=b-p=-N,$ it turns out that this embedding is
equivalent to the multiplicative inequality $||u||_{c,r}\leq C||\nabla
u||_{b,p}^{\theta }||u||_{a,q}^{1-\theta }$ for some suitable $\theta \in
[0,1],$ often but not always unique. If $a,b,c>-N,$ then $C_{0}^{\infty }(%
\Bbb{R}^{N})\subset W_{\{a,b\}}^{(q,p)}(\Bbb{R}^{N}\backslash \{0\})\cap
L^{r}(\Bbb{R}^{N};|x|^{c}dx)$ and such inequalities for $u\in C_{0}^{\infty %
}(\Bbb{R}^{N})$ are the well-known Caffarelli-Kohn-Nirenberg inequalities,
but their generalization to $W_{\{a,b\}}^{(q,p)}(\Bbb{R}^{N}\backslash \{0\})
$ cannot be proved by a denseness argument. Without assuming $a,b,c>-N,$ the
inequalities are essentially new even when $u\in C_{0}^{\infty }(\Bbb{R}%
^{N}\backslash \{0\}),$ although a few special cases are known, most notably
the Hardy-type inequalities when $p=q.$

In a different direction, the embedding theorem easily yields a
generalization when the weights $|x|^{a},|x|^{b}$ and $|x|^{c}$ are replaced
by more general weights $w_{a},w_{b}$ and $w_{c},$ respectively, having
multiple power-like singularities at finite distance and at infinity.
\end{abstract}

\section{Introduction\label{intro}}

If $d\in \Bbb{R}$ and $1\leq s<\infty ,$ let $||\cdot ||_{d,s}$ denote the
norm of the space $L^{s}(\Bbb{R}^{N};|x|^{d}dx),$ where the $|x|^{d}dx$
-measure of $\{0\}$ is defined to be $0$ (which must be specified if $d\leq
-N$). With this definition, $u\in L^{s}(\Bbb{R}^{N};|x|^{d}dx)$ if and only
if $|x|^{\frac{d}{s}}u\in L^{s}(\Bbb{R}^{N})$ and $||u||_{d,s}=||\,|x|^{%
\frac{d}{s}}u||_{s},$ where $||\cdot ||_{s}:=||\cdot ||_{0,s}.$ Throughout
the paper, $\Bbb{R}_{*}^{N}:=\Bbb{R}^{N}\backslash \{0\}.$

Given $a,b\in \Bbb{R}$ and $1\leq p,q<\infty ,$ consider the weighted
Sobolev space 
\begin{multline}
W_{\{a,b\}}^{1,(q,p)}(\Bbb{R}_{*}^{N}):=  \label{-1} \\
\{u\in L_{loc}^{1}(\Bbb{R}_{*}^{N}):u\in L^{q}(\Bbb{R}^{N};|x|^{a}dx),\quad
\nabla u\in (L^{p}(\Bbb{R}^{N};|x|^{b}dx))^{N}\},
\end{multline}
equipped with the norm 
\begin{equation}
||\,u||_{a,q}+||\,\nabla u||_{b,p}.  \label{0}
\end{equation}
Since $W_{\{a,b\}}^{1,(q,p)}(\Bbb{R}_{*}^{N})$ may contain functions which
are not locally integrable near $0$ and hence not distributions on $\Bbb{R}
^{N},$ it is generally larger than the space $W_{\{a,b\}}^{1,(q,p)}(\Bbb{R}%
^{N})$ (self-explanatory notation) which, incidentally, is not always
complete.

In this paper, we characterize all the real numbers $a,b,c$ and $1\leq
p,q,r<\infty $ such that 
\begin{equation*}
W_{\{a,b\}}^{1,(q,p)}(\Bbb{R}_{*}^{N})\hookrightarrow L^{r}(\Bbb{R}%
^{N};|x|^{c}dx),
\end{equation*}
where ``$\hookrightarrow $'' denotes continuous embedding. This provides
sufficient conditions for $W_{\{a,b\}}^{1,(q,p)}(\Bbb{R}^{N})\hookrightarrow
L^{r}(\Bbb{R}^{N};|x|^{c}dx),$ but their necessity is not investigated.

In spite of the large literature devoted to embeddings of weighted Sobolev
spaces, there seems to be little that addresses and resolves the exact same
question in special cases$.$ While most results allow for weights satisfying
general properties, they also incorporate a number of restrictive hypotheses
which are rarely necessary. Only a few are applicable to the whole -or
punctured- space and even fewer accommodate weights which, like all
nontrivial power weights, exhibit singularities at $0$ and infinity
simultaneously. This is especially true when more than one weight (here, $%
a\neq b$) or more than one order of integration (i.e., $p\neq q$) is
involved in the source space. In addition, the weighted spaces are often
defined to be the unknown closure of some subspace of smooth (enough)
functions, as indeed the denseness issue is a notorious difficulty (\cite
{Zh98}). In particular, this is the definition chosen in \cite{Ma85} (see
also the more recent and expanded book \cite{Ma11}), except in the
unweighted case.

Before continuing this discussion, we shall state the embedding theorem. In
addition to the standard notation 
\begin{equation*}
p^{*}=\infty \text{ if }p\geq N\text{ and }p^{*}=\frac{Np}{N-p}\text{ if }
1\leq p<N,
\end{equation*}
we denote by $c^{0}$ and $c^{1}$ the two points 
\begin{equation}
c^{0}:=\frac{r(a+N)}{q}-N\qquad \text{and\qquad }c^{1}:=\frac{r(b-p+N)}{p}-N,
\label{1}
\end{equation}
where it is understood that $a,b,p,q$ and $r$ are given. The points $c^{0}$
and $c^{1}$ are distinct if and only if $\frac{a+N}{q}\neq \frac{b-p+N}{p}.$
If so and if $c$ is in the closed interval with endpoints $c^{0}$ and $%
c^{1}, $ we set 
\begin{equation}
\theta _{c}:=\frac{c-c^{0}}{c^{1}-c^{0}},  \label{2}
\end{equation}
so that $\theta _{c}\in [0,1]$ and that 
\begin{equation}
c=\theta _{c}c^{1}+(1-\theta _{c})c^{0}.  \label{3}
\end{equation}
Observe that $\theta _{c^{0}}=0$ and $\theta _{c^{1}}=1$ and that, by (\ref
{1}), (\ref{2}) and (\ref{3}), 
\begin{equation}
\frac{c+N}{r}=\theta _{c}\frac{b-p+N}{p}+(1-\theta _{c})\frac{a+N}{q}.
\label{4}
\end{equation}

\begin{theorem}
\smallskip \label{th0} Let $a,b,c\in \Bbb{R}$ and $1\leq p,q,r<\infty $ be
given ($1\leq p<\infty $ and $0<q,r<\infty $ if $N=1$). Then, $%
W_{\{a,b\}}^{1,(q,p)}(\Bbb{R}_{*}^{N})\hookrightarrow L^{r}(\Bbb{R}%
^{N};|x|^{c}dx)$ (and hence $W_{\{a,b\}}^{1,(q,p)}(\Bbb{R}%
_{*}^{N})\hookrightarrow W_{\{c,b\}}^{1,(r,p)}(\Bbb{R}_{*}^{N})$) if and
only if $r\leq \max \{p^{*},q\}$ and one of the following conditions%
\footnote{%
The overlap between conditions (iii), (iv) and (v) makes for a simpler and
clearer statement.} holds:\newline
(i) $a$ and $b-p$ are on the same side of $-N$ (including $-N$), $\frac{a+N}{%
q}\neq \frac{b-p+N}{p},$ $c$ is in the open interval with endpoints $c^{0}$
and $c^{1}$ and $\theta _{c}\left( \frac{1}{p}-\frac{1}{N}-\frac{1}{q}%
\right) \leq \frac{1}{r}-\frac{1}{q}.$\newline
(ii) $a$ and $b-p$ are strictly on opposite sides of $-N$ (hence $\frac{a+N}{%
q}\neq \frac{b-p+N}{p}$), $c$ is in the open interval with endpoints $c^{0}$
and $-N$ and $\theta _{c}\left( \frac{1}{p}-\frac{1}{N}-\frac{1}{q}\right)
\leq \frac{1}{r}-\frac{1}{q}.$\newline
(iii) $r=q$ and $c=a$ ($=c^{0}$).\newline
(iv) $p\leq r\leq p^{*},a\leq -N$ and $b-p<-N$ or $a\geq -N$ and $%
b-p>-N,c=c^{1}.$\newline
(v) ($\max \{p^{*},q\}\geq $) $r\geq \min \{p,q\},\frac{a+N}{q}=\frac{b-p+N}{%
p}\neq 0$ and $c=c^{1}$ ($=c^{0}$).\newline
(vi) $a=-N,b=p-N,q<r\leq p^{*}$ and $c=c^{1}$ ($=c^{0}=-N$).
\end{theorem}

\smallskip Since $r$ is finite, $r=p^{*}$ is impossible when $p\geq N.$ The
set of admissible values of $c$ is an interval (possibly $\emptyset ,$ see
Remark \ref{rm1}), of which $c^{0},c^{1}$ and $-N$ may or may not be
endpoints, but never interior points. When $c^{0}$ or $c^{1}$ are endpoints,
their admissibility is decided by parts (iii) to (vi). Endpoints other than $%
c^{0},$ $c^{1}$ or $-N$ are always admissible, but $-N$ is never admissible
when $a\neq -N.$ If $a=-N,$ then $-N$ is admissible only in the trivial case
(iii) and the exceptional case (vi).

\smallskip Apparently, aside from the trivial part (iii), only parts (v) and
(vi) of Theorem \ref{th0} when $q=p$ (hence $a=b-p$) are known with
nontrivial weights. See Opic and Kufner \cite[p. 291]{OpKu90}, where the
result is credited to Opic and Gurka \cite{OpGu89}. Curiously, if $b-p\neq
-N $ and $a_{q}:=\frac{q(b-p+N)}{p}-N,$ part (v) shows that the space $%
W_{\{a_{q},b\}}^{1,(q,p)}(\Bbb{R}_{*}^{N})$ is independent of $q\in
[p,p^{*}],q<\infty ,$ with equivalent norms as $q$ is varied. When $N=1,$
part (iv) can -and will- be deduced from an inequality of Bradley \cite{Br78}%
. Related, but different, work is discussed further below.

In the unweighted case $a=b=c=0$ and if $p=q$ and $N\geq 2$ (a minor point),
Theorem \ref{th0} gives again $W^{1,p}(\Bbb{R}_{*}^{N})=W^{1,p}(\Bbb{R}
^{N})\hookrightarrow L^{r}(\Bbb{R}^{N})$ if and only if ($r<\infty $ and) $%
p\leq r\leq p^{*}$ (Subsection \ref{ex1}). If $p\neq q$ (and still $a=b=c=0$%
), Theorem \ref{th0} is akin to embedding theorems in \cite{BeIlNi78}, \cite
{BeIlNi79}.

\begin{remark}
\label{rm1}If $r\leq \min \{p^{*},q\},$ then $\theta _{c}\left( \frac{1}{p}-%
\frac{1}{N}-\frac{1}{q}\right) \leq \frac{1}{r}-\frac{1}{q}$ for every $c$
between $c^{0}$ and $c^{1}.$ In contrast, all the conditions of Theorem \ref
{th0} fail (i.e., no embedding holds for any $c$) if $p<N$ and $r>\max
\{p^{*},q\}$ or if either (i) $p<N,r=p^{*}>q,b-p=-N\neq a$ or (ii) $q<r\leq
p^{*},a$ and $b-p$ are strictly on opposite sides of $-N$ (hence $\theta
_{-N}$ is defined) and $\theta _{-N}\left( \frac{1}{p}-\frac{1}{N}-\frac{1}{q%
}\right) \geq \frac{1}{r}-\frac{1}{q}.$
\end{remark}

\smallskip When $\frac{a+N}{q}\neq \frac{b-p+N}{p},$ a simple rescaling
shows (Corollary \ref{cor2}) that the embedding $W_{\{a,b\}}^{1,(q,p)}(\Bbb{R%
}_{*}^{N})\hookrightarrow L^{r}(\Bbb{R}^{N};|x|^{c}dx)$ is equivalent to the
multiplicative inequality 
\begin{equation}
||u||_{c,r}\leq C||\nabla u||_{b,p}^{\theta _{c}}||u||_{a,q}^{1-\theta _{c}},
\label{5}
\end{equation}
rather than just $||u||_{c,r}\leq C\left( ||u||_{a,q}+||\nabla
u||_{b,p}\right) .$

When $a,b,c>-N$ and $u\in C_{0}^{\infty }(\Bbb{R}^{N}),$ (\ref{5}) is one of
the well-known Caffarelli-Kohn-Nirenberg (CKN for short) inequalities in 
\cite{CaKoNi84}. Therefore, parts (i) and (ii) of Theorem \ref{th0} give
necessary and sufficient conditions for the validity of the CKN inequality (%
\ref{5}) when $\frac{a+N}{q}\neq \frac{b-p+N}{p},$ but without the
restriction $a,b,c>-N$ and for $u\in W_{\{a,b\}}^{1,(q,p)}(\Bbb{R}_{*}^{N}).$
Note that $C_{0}^{\infty }(\Bbb{R}^{N})\subset W_{\{a,b\}}^{1,(q,p)}(\Bbb{R}%
_{*}^{N})$ when $a,b>-N,$ so that even in this case, (\ref{5}) is a genuine
generalization. As already pointed out, it does not follow by a denseness
argument without many extra conditions ($\Bbb{R}_{*}^{N}$ replaced by $\Bbb{R%
}^{N},p=q,$ $a=b$ and $|x|^{a}$ an $A_{p}$ weight, i.e. $-N<a<(p-1)N;$ see 
\cite[Theorem 1.27]{HeKiMa93} or \cite{NaToYa04}). The denseness of $%
C_{0}^{\infty }(\Bbb{R}^{N})$ is obviously meaningless when $a\leq -N$ or $%
b\leq -N$ while that of $C_{0}^{\infty }(\Bbb{R}_{*}^{N}),$ always contained
in $W_{\{a,b\}}^{1,(q,p)}(\Bbb{R}_{*}^{N}),$ is generally false (see
Subsection \ref{ex3}) and hence definitely not a viable approach.

Inequalities of CKN type have been discussed earlier, beginning with the
1961 work of Il'in \cite[Theorem 1.4]{Il61}, who proved (with $c^{1}$ given
by (\ref{1})) $||u||_{c^{1},r,G}\leq C||\nabla u||_{b,p,\Omega }$ when $%
\Omega $ is a fairly general open subset of $\Bbb{R}^{N},$ $G$ is a bounded
measurable subset of $\Omega $ and $u$ is $C^{1}.$ There are further
limitations about $b,p$ and $r,$ but the result has various generalizations
when higher order derivatives are involved, or when $G$ is a bounded subset
of a section of $\Omega $ by a lower-dimensional hyperplane. Results of a
somewhat similar nature are proved in \cite[Section 2.1.6]{Ma85}, \cite{Ma11}
when $\Omega =\Bbb{R}^{N}$ and $u\in C_{0}^{\infty }(\Bbb{R}^{N}).$

When $\Omega $ is an open subset of $\Bbb{R}^{N},\mu $ and $\nu $ are
nonnegative Borel measures, $\Phi \geq 0$ is continuous and positively
homogeneous of degree $1$ in its second argument and $\frac{1}{r}\leq \frac{%
\theta }{p}+\frac{1-\theta }{q},$ Maz'ya \cite[Theorem 9]{Ma73} (reproduced
in \cite[p.127]{Ma85} and \cite{Ma11}) gives interesting\emph{\ necessary
and sufficient} conditions for the inequality 
\begin{equation}
||u||_{L^{r}(\Omega ;\mu )}\leq C\left( \int_{\Omega }\Phi (x,\nabla
u)^{p}dx\right) ^{\frac{\theta }{p}}||u||_{L^{q}(\Omega ;\nu )}^{1-\theta },
\label{6}
\end{equation}
to hold for $u\in C_{0}^{\infty }(\Omega ).$ When $\Omega =\Bbb{R}%
_{*}^{N},\mu (E)=\int_{E}|x|^{c}dx,\nu (E)=\int_{E}|x|^{a}dx$ and $\Phi
(x,y)=|x|^{\frac{b}{p}}|y|,$ the setting of Theorem \ref{th0} is recovered.

Maz'ya's conditions for (\ref{6}) are expressed in terms of the $(p,\Phi )$%
-capacity of ``admissible'' sets and their $\mu $ and $\nu $ measures. As
early as 1960, he noted in \cite{Ma60} that such conditions could be used to
prove the equivalence between various inequalities (e.g., Sobolev and Nash).
This kind of equivalence has since been revisited by a number of authors.
For example, when $a=c,$ it follows from Bakry \textit{et al}. \cite
{BaCoLeSa95} that if the inequality $||u||_{a,r}\leq C||\nabla
u||_{b,p}^{\theta }||u||_{a,q}^{1-\theta }$ holds when $q=q_{0},r=r_{0},%
\theta =\theta _{0}$ and (say) $u$ is a Lipschitz continuous function with
compact support, then the same inequality continues to hold for a family of
other values of $q,r$ and $\theta .$ Once again, denseness issues are an
obstacle to extending this property to the spaces $W_{\{a,b\}}^{1,(q,p)}(%
\Bbb{R}_{*}^{N})$ unless $a=b=c=0$ (unweighted case).

The connection of this work with the CKN inequalities can be found in some
of the preliminary results in \cite{CaKoNi84} which, possibly in generalized
form, are also useful for the proof of Theorem \ref{th0}. However, without
the compactness of the supports and other key assumptions, a mere tweaking
of the arguments of \cite{CaKoNi84} is not possible.

In the next section, we show that (\ref{5}) is equivalent to an embedding
inequality and that the hypotheses of Theorem \ref{th0} are necessary. The
necessity of $r\leq \max \{p^{*},q\}$ and of $\theta _{c}\left( \frac{1}{p}-%
\frac{1}{N}-\frac{1}{q}\right) \leq \frac{1}{r}-\frac{1}{q}$ in parts (i)
and (ii) of Theorem \ref{th0} follows very simply from (\ref{5}) and a
remark in \cite{CaKoNi84} used here in a more general framework (Theorem \ref
{th3} (i)). A variant of it proves the necessity of $r\leq \max \{p^{*},q\}$
in the remaining cases (Theorem \ref{th3} (ii)).

The verification of the sufficiency is demanding. The general idea is first
to prove Theorem \ref{th0} for radially symmetric functions. Once this is
done, there are two different ways to proceed. The first one is to reduce
the problem to the symmetric case by a \emph{suitable} radial
symmetrization. This works when $1\leq r\leq \min \{p,q\}.$ The second
option is to prove an independent embedding theorem for a direct complement
of the subspace of radially symmetric functions. This can be done, based on
ideas in \cite{CaKoNi84}, under assumptions about $p,q$ and $r$ that rule
out $r<\min \{p,q\}.$ This is why it is crucial that this case can be
settled by other arguments.

The proof of the embedding theorem for radially symmetric functions and,
next, by radial symmetrization, requires some preliminaries. It is more
natural to work with the larger spaces (the domain $\Bbb{R}_{*}^{N}$ is not
mentioned for simplicity) 
\begin{multline}
\widetilde{W}_{\{a,b\}}^{1,(q,p)}:=  \label{7} \\
\{u\in L_{loc}^{1}(\Bbb{R}_{*}^{N}):u\in L^{q}(\Bbb{R}^{N};|x|^{a}dx),\quad
\partial _{\rho }u\in L^{p}(\Bbb{R}^{N};|x|^{b}dx)\},
\end{multline}
equipped with the norm 
\begin{equation}
||u||_{\{a,b\},(q,p)}:=||u||_{a,q}+||\,\partial _{\rho }u||_{b,p},  \label{8}
\end{equation}
where $\partial _{\rho }u:=\nabla u\cdot \frac{x}{|x|}$ is the radial
derivative of $u.$ Since $|x|^{-1}x$ is a smooth field on $\Bbb{R}_{*}^{N},$
this definition makes sense for every distribution $u$ on $\Bbb{R}_{*}^{N}.$

When $0<q<1,$ the definitions (\ref{-1}) and (\ref{7}) can still be used,
but (\ref{0}) and (\ref{8}) are only quasi-norms. The equivalence between
continuity and boundedness for linear operators remains true in quasi-normed
spaces. For more details about such spaces, see \cite{BeLi00} or \cite{Ro72}.

The spaces $W_{\{a,b\}}^{1,(q,p)}(\Bbb{R}_{*}^{N})$ and $\widetilde{W}
_{\{a,b\}}^{1,(q,p)}$ contain the same radially symmetric functions and the
induced (quasi) norms are the same, because $\nabla u=(\partial _{\rho }u)%
\frac{x}{|x|}$ when $u$ is radially symmetric. Thus, when referring to
radially symmetric functions, the ambient space $W_{\{a,b\}}^{1,(q,p)}(\Bbb{R%
}_{*}^{N})$ or $\widetilde{W}_{\{a,b\}}^{1,(q,p)}$ is unimportant.

In the next section, the basic features of a related space $\widetilde{W}
_{loc}^{1,p}(\Bbb{R}_{*}^{N})$ (abbreviated $\widetilde{W}_{loc}^{1,p}$) are
discussed, along with some of their implications regarding $\widetilde{W}
_{\{a,b\}}^{1,(q,p)}.$ This material is directly relevant to the proof of
the main results of Sections \ref{radial} and \ref{embedding1}.

Necessary and sufficient conditions for the continuity of the embedding of
the subspace of radially symmetric functions when $q,r>0$ and $p\geq 1$ are
given in Theorem \ref{th16}. Of course, this is a (barely) disguised form of
Theorem \ref{th0} when $N=1.$ Compared with the treatment of the same
problem in \cite{CaKoNi84}, convenient tools (e.g., radial integration by
parts) cannot be used and some estimates (e.g., of $|u(0)|$) make no longer
sense. For that reason, our approach is technically completely different.

The proof of Theorem \ref{th0} for arbitrary $N$ begins in Section \ref
{embedding1}, where the case $1\leq r\leq \min \{p,q\}$ is considered. As
mentioned before, this is done by radial symmetrization, though not in the
obvious way (Lemma \ref{lm17}). The result (Theorem \ref{th18}) is more
general and sharper than the corresponding part of Theorem \ref{th0} since
it establishes the continuous embedding of the larger space $\widetilde{W}
_{\{a,b\}}^{1,(q,p)},$ with a weaker norm, into $L^{r}(\Bbb{R}
^{N};|x|^{c}dx) $ under the conditions already necessary for the embedding
of $W_{\{a,b\}}^{1,(q,p)}(\Bbb{R}_{*}^{N}).$ Thus, the embedding is obtained
without assuming the integrability of the first derivatives, except for just
the radial one.

The case when $r>\min \{p,q\}$ is split into the three parts: $p<r\leq q$
(Theorem \ref{th22}), $r>q$ and $r\geq p$ (Theorem \ref{th26}) and $q<r<p$
(Theorem \ref{th27}). If $p=q,$ Sections \ref{embedding2} and \ref
{embedding4} can be skipped with no prejudice. A preliminary embedding lemma
for functions with null radial symmetrization, essentially due to
Caffarelli, Kohn and Nirenberg, is proved in Section \ref{CKNlemma} (Lemma 
\ref{lm19}), then rephrased in a more convenient way (Corollary \ref{cor20}%
). The technical steps are simple, but cannot be repeated with the larger
space $\widetilde{W}_{\{a,b\}}^{1,(q,p)}.$ The proofs of Theorem \ref{th22}
(when $p<r\leq q$) and Theorem \ref{th27} (when $1\leq q<p<r$) also heavily
rely on Theorem \ref{th18} (when $1\leq r\leq \min \{p,q\},$ but with other
parameters).

\smallskip The relationship between Theorem \ref{th0} and the CKN
inequalities does not stop with (\ref{5}) when $\frac{a+N}{q}\neq \frac{b-p+N%
}{p}:$ In Section \ref{CKNinequalities}, we show that the embedding $%
W_{\{a,b\}}^{1,(q,p)}(\Bbb{R}_{*}^{N})\hookrightarrow L^{r}(\Bbb{R}%
^{N};|x|^{c}dx)$ continues to be equivalent to a multiplicative inequality $%
||u||_{c,r}\leq C||\nabla u||_{b,p}^{\theta }||u||_{a,q}^{1-\theta }$ for
some suitable $\theta \in [0,1]$ when $\frac{a+N}{q}=\frac{b-p+N}{p}$
(Theorem \ref{th29}), except when $N\geq 2$ and $a=c=b-p=-N$ (Theorem \ref
{th30}). Of course, $\theta $ is no longer $\theta _{c}$ in (\ref{2}), which
is not defined, and it may not always be unique (Remark \ref{rm5}) When $%
\theta =1,$ this is an $N$-dimensional weighted Hardy inequality more
general than those in the current literature (\cite{GaGuWh85}, \cite{OpKu90}%
). The case when $u\in C_{0}^{\infty }(\Bbb{R}_{*}^{N}),p=q=r=2,c=\frac{a+b}{%
2}-1$ and $\theta =\frac{1}{2}$ was recently investigated by Catrina and
Costa \cite{CaCo09}.

In Section \ref{examples}, three special cases are discussed and the
(simple) generalization when $|x|^{a},|x|^{b}$ and $|x|^{c}$ are replaced by
weights $w_{a},w_{b}$ and $w_{c}$ having multiple power-like singularities
is briefly sketched.

\subsection{Notation\label{notation}}

Throughout the paper, $C>0$ denotes a constant which, as is customary, may
have different values in different places. If $k\geq 1$ is a real number, $%
k^{\prime }\leq \infty $ will always denote the H\"{o}lder conjugate of $k.$
Also, $\zeta \in C_{0}^{\infty }(\Bbb{R}^{N})$ is chosen once and for all
such that $0\leq \zeta \leq 1$ is radially symmetric, $\zeta (x)=1$ if $%
|x|\leq \frac{1}{2}$ and $\zeta (x)=0$ if $|x|\geq 1.$ Naturally, we shall
also use the notation introduced more formally earlier on. Up to and
including Section \ref{radial}, we shall frequently refer to the Kelvin
transform, defined in the following remark.

\begin{remark}
\label{rm2}The Kelvin transform $x\mapsto $ $|x|^{-2}x$ on $\Bbb{R}_{*}^{N}$
is an isometry from $\widetilde{W}_{\{a,b\}}^{1,(q,p)}$ onto $\widetilde{W}%
_{\{-2N-a,2p-2N-b\}}^{1,(q,p)}$ and from $L^{r}(\Bbb{R}^{N};|x|^{c}dx)$ onto 
$L^{r}(\Bbb{R}^{N};|x|^{-2N-c}dx)$ for all values of the parameters. As a
result, in many proofs that split into two complementary cases, it will be
enough to discuss only one of them, because the other follows from this
isometry.
\end{remark}

\section{Necessary conditions for continuous embedding\label{necessary}}

In this section, we prove that the conditions given in Theorem \ref{th0} are
necessary.

\begin{theorem}
\label{th1}Let $a,b,c\in \Bbb{R}$ and $1\leq p<\infty ,0<q,r<\infty $ be
given. Then, $W_{\{a,b\}}^{1,(q,p)}(\Bbb{R}_{*}^{N})$ (hence \textit{a
fortiori} $\widetilde{W}_{\{a,b\}}^{1,(q,p)}$) is not contained $L^{r}(\Bbb{R%
}^{N};|x|^{c}dx)$ if:\newline
(i) $c$ does not belong to the closed interval with endpoints $c^{0}$ and $%
c^{1}.$\newline
(ii) $b-p\leq -N<a$ or $b-p\geq -N>a$ and $c$ does not belong to the
interval with endpoints $c^{0}$ (included) and $-N$ (not included).\newline
Furthermore, $W_{\{a,b\}}^{1,(q,p)}(\Bbb{R}_{*}^{N})$ (hence \textit{a
fortiori} $\widetilde{W}_{\{a,b\}}^{1,(q,p)}$) is not continuously\footnote{%
In principle at least, that does not rule out $W_{\{a,b\}}^{1,(q,p)}(\Bbb{R}
_{*}^{N})\subset L^{r}(\Bbb{R}^{N};|x|^{c}dx).$} embedded into $L^{r}(\Bbb{R}%
^{N};|x|^{c}dx)$ if: \newline
(iii) $\frac{a+N}{q}\neq \frac{b-p+N}{p},$ $c=c^{0}$ and $r\neq q$ (if $r=q,$
then $c^{0}=a$ and the embedding is trivial).\newline
(iv) $\frac{a+N}{q}\neq \frac{b-p+N}{p},$ $c=c^{1}$ and $r<p.$\newline
(v) $\frac{a+N}{q}=\frac{b-p+N}{p},r<\min \{p,q\}$ and $c=c^{0}$ ($=c^{1}$).%
\newline
(vi) $a=-N,b=p-N,r<q$ and $c=c^{0}$ ($=c^{1}=-N$).
\end{theorem}

\begin{proof}
(i) If $c<\min \left\{ c^{0},c^{1}\right\} ,$ let $u(x):=|x|^{-\frac{c+N}{r}
}\zeta (x)$ with $\zeta $ as in subsection \ref{notation}. Then, $u\notin
L^{r}(\Bbb{R}^{N};|x|^{c}dx)$ since $|x|^{c}|u(x)|^{r}=|x|^{-N}$ on a
neighborhood of $0,$ but $u\in W_{\{a,b\}}^{1,(q,p)}(\Bbb{R}_{*}^{N})$ since 
$\min \left\{ a-\frac{q(c+N)}{r},b-p-\frac{p(c+N)}{r}\right\} >-N$ and $%
\nabla \zeta $ has compact support and vanishes on a neighborhood of $0.$

If $c>\max \{c^{0},c^{1}\},$ let $u(x):=|x|^{-\frac{c+N}{r}}(1-\zeta (x))$
and argue as above, with obvious modifications.

(ii) By Kelvin transform (Remark \ref{rm2}), it suffices to consider $%
b-p\leq -N<a.$ Note that $c^{1}\leq -N<c^{0}$ and let $c\notin \left(
-N,c^{0}\right] .$ By (i), $W_{\{a,b\}}^{1,(q,p)}(\Bbb{R}_{*}^{N})\nsubseteq
L^{q}(\Bbb{R}^{N};|x|^{c}dx)$ if $c>c^{0}.$ If now $c\leq -N,$ then $\zeta
\notin L^{r}(\Bbb{R}^{N};|x|^{c}dx)$ since $\zeta =1$ on a neighborhood of $%
0,$ but $\zeta \in W_{\{a,b\}}^{1,(q,p)}(\Bbb{R}_{*}^{N})$ because $a>-N$
and $\nabla \zeta $ has compact support and vanishes on a neighborhood of $%
0. $

(iii) By contradiction, if $W_{\{a,b\}}^{1,(q,p)}(\Bbb{R}_{*}^{N})$ $%
\hookrightarrow L^{r}(\Bbb{R}^{N};|x|^{c^{0}}dx),$ then $||u||_{c^{0},r}\leq 
$\linebreak $C(||u||_{a,q}+||\nabla u||_{b,p})$ for every $u\in
W_{\{a,b\}}^{1,(q,p)}(\Bbb{R}_{*}^{N}).$ By rescaling and since $\frac{%
c^{0}+N}{r}=\frac{a+N}{q},$ it follows that $||u||_{c^{0},r}\leq
C(||u||_{a,q}+\lambda ^{\frac{a+N}{q}-\frac{b-p+N}{p}}||\nabla u||_{b,p})$
for the same constant $C$ independent of $\lambda >0.$ Since $\frac{a+N}{q}
\neq \frac{b-p+N}{p},$ this yields $||u||_{c^{0},r}\leq C||u||_{a,q}.$ In
particular, if $u(x):=|x|^{-\frac{c^{0}+N-1}{r}}g(|x|)=|x|^{\frac{1}{r}-%
\frac{a+N}{q}}g(|x|)$ with $g\in C_{0}^{\infty }(0,\infty ),$ it follows
that $||g||_{r}\leq C||g||_{\frac{q}{r}-1,q}$ when $g\in C_{0}^{\infty
}(0,\infty ),g\geq 0,$ or $g$ is the a.e. limit of a nondecreasing sequence
of such functions. Thus, a counterexample is obtained by choosing $g:=\chi
_{(n,n+1)}$ if $r>q$ and $g:=t^{\frac{1}{n}-\frac{1}{r}}\chi _{(0,1)}$ if $%
r<q$ and by letting $n$ tend to $\infty .$

(iv) The scaling used in (iii) now shows that if $W_{\{a,b\}}^{1,(q,p)}(\Bbb{%
R}_{*}^{N})$ $\hookrightarrow L^{r}(\Bbb{R}^{N};|x|^{c^{1}}dx),$ then $%
||u||_{c^{1},r}\leq C||\nabla u||_{b,p}$ for some constant $C>0.$ The proof
that $C$ does not exist is slightly different when $a\neq -N$ and when $%
a=-N. $

\textit{Case (iv-{\scriptsize 1}):} $a\neq -N.$

By Kelvin transform, we may assume $a<-N$ with no loss of generality. It
suffices to prove that, given $C>0,$ 
\begin{equation}
||f||_{c^{1}+N-1,r}\leq C||f^{\prime }||_{b+N-1,p},  \label{9}
\end{equation}
cannot hold for every $f\in W_{loc}^{1,p}(0,\infty )$ with $f\geq 0,f=0$ on
a neighborhood of $0$ and $f=M$ (constant) on a neighborhood of $\infty $
(if so, $u(x)=f(|x|)$ is in $W_{\{a,b\}}^{1,(q,p)}(\Bbb{R}_{*}^{N})$
irrespective of $b\in \Bbb{R}$ and of $p\geq 1,q>0$).

It is well-known that if $1\leq r<p$ and $C>0,$ the weighted Hardy
inequality $\left( \int_{0}^{\infty }t^{\frac{r(b-p+N)}{p}-1}\left(
\int_{0}^{t}g(\tau )d\tau \right) ^{r}dt\right) ^{\frac{1}{r}}\leq C\left(
\int_{0}^{\infty }t^{b+N-1}g(t)^{p}dt\right) ^{\frac{1}{p}}$ does not hold
for every measurable $g\geq 0$ on $(0,\infty ),$ because power weights never
satisfy the necessary compatibility condition when $r<p$ (%
\cite[Theorem 1, p. 47]{Ma85}). This is also true, but more delicate, when $%
0<r<1$ (\cite{Si91}, \cite{SiSt96}). Thus, if $0<r<p,$ there is a sequence $%
g_{n}\geq 0$ such that $\int_{0}^{\infty }t^{b+N-1}g_{n}(t)^{p}dt<\infty $
and that 
\begin{equation*}
\left( \int_{0}^{\infty }t^{\frac{r(b-p+N)}{p}-1}\left(
\int_{0}^{t}g_{n}(\tau )d\tau \right) ^{r}dt\right) ^{\frac{1}{r}}>n\left(
\int_{0}^{\infty }t^{b+N-1}g_{n}^{p}(t)dt\right) ^{\frac{1}{p}}.
\end{equation*}
If $b-p\geq -N,$ the left-hand side is even $\infty $ when $g_{n}\neq 0,$ so
it may be assumed that $b-p<-N$ whenever convenient (which happens to be the
case when $p=1$). The simple proof by Sinnamon and Stepanov (%
\cite[Theorem 2.4]{SiSt96} if $p>1,$ \cite[Theorem 3.3]{SiSt96} if $p=1$)
reveals at once that $g_{n}$ may be chosen in $L^{p}(0,\infty )$ and with
compact support. Then, $f_{n}(t):=\int_{0}^{t}g_{n}(\tau )d\tau \geq 0$
vanishes on a neighborhood of $0$ and is eventually constant. Since $\frac{%
r(b-p+N)}{p}-1=c^{1}+N-1,$ this provides a counterexample to (\ref{9}).

\textit{Case (iv-{\scriptsize 2}):} $a=-N.$

Then, $b-p\neq -N$ since $\frac{a+N}{q}\neq \frac{b-p+N}{p}.$ By the usual
Kelvin transform argument -which does not affect $a=-N$- we may assume $%
b-p<-N.$ It suffices to show that (\ref{9}) cannot hold for every $f\in
W_{loc}^{1,p}(0,\infty )$ with $f\geq 0,$ $f=0$ on a neighborhood of $0$ and 
$f(t)=Mt^{-\varepsilon }$ for some constants $M,\varepsilon >0$ and large $t$
(if so, $u(x)=f(|x|)$ is in $W_{\{-N,b\}}^{1,(q,p)}(\Bbb{R}_{*}^{N})$ since $%
b-p<-N$).

With $f_{n}$ and $g_{n}=f_{n}^{\prime }$ as in Case (iv-{\scriptsize 1})
above, set $h_{n}(t):=f_{n}(t)$ if $0<t<1$ and $h_{n}(t):=t^{-\varepsilon
_{n}}f_{n}(t)$ if $t\geq 1,$ where $\varepsilon _{n}>0$ will be chosen
shortly. Note that $h_{n}=0$ on a neighborhood of $0$ and $%
h_{n}(t)=M_{n}t^{-\varepsilon _{n}}$ for $t>0$ large enough since $%
f_{n}(t)=M_{n}$ is constant for large $t.$ Since $f_{n}$ provides a
counterexample to (\ref{9}) and $h_{n}=f_{n}$ on $(0,1),$ $h_{n}$ will also
be a counterexample if, when $n$ is fixed, $\varepsilon _{n}>0$ can be
chosen so that $\int_{1}^{\infty }t^{c^{1}+N-1}h_{n}(t)^{r}dt$ is
arbitrarily close to $\int_{1}^{\infty }t^{c^{1}+N-1}f_{n}(t)^{r}dt$ and $%
\int_{1}^{\infty }t^{b+N-1}|h_{n}^{\prime }(t)|^{p}dt$ is arbitrarily close
to $\int_{1}^{\infty }t^{b+N-1}|f_{n}^{\prime }(t)|^{p}dt.$

By the monotone convergence of $\int_{1}^{\infty }t^{c^{1}+N-1-r\varepsilon
}f_{n}(t)^{r}dt$ as $\varepsilon \searrow 0,$ the former property holds. For
the latter, it suffices to use (1) $\lim_{\varepsilon \rightarrow
0}\int_{1}^{\infty }t^{b+N-1-p\varepsilon }g_{n}(t)^{p}dt=\int_{1}^{\infty
}t^{b+N-1}g_{n}(t)^{p}dt,$ also proved by a monotone convergence argument,
and\linebreak (2) $\lim_{\varepsilon \rightarrow 0}\varepsilon
^{p}\int_{1}^{\infty }t^{-p\varepsilon +b-p+N-1}f_{n}(t)^{p}dt=0,$ which
follows from the boundedness of $f_{n}$ and from $b-p<-N.$

(v) The main difference with the proof of parts (iii) and (iv) is that the
scaling argument used there is inoperative because all the powers of $%
\lambda $ cancel out. Let $\eta $ denote the common value 
\begin{equation}
\eta :=\frac{a+N}{q}=\frac{b-p+N}{p}=\frac{c+N}{r}.  \label{10}
\end{equation}
If $||u||_{c,r}\leq C(||u||_{a,q}+||\nabla u||_{b,p})$ for every $u\in
W_{\{a,b\}}^{1,(q,p)}(\Bbb{R}_{*}^{N}),$ the choice $u(x):=f(|x|)$ with $%
f\in C_{0}^{\infty }(0,\infty )$ yields $||f||_{c+N-1,r}\leq C(||f^{\prime
}||_{b+N-1,p}+||f||_{a+N-1,q}).$ If now $g\in C_{0}^{\infty }(\Bbb{R}),$
then $f(t)=t^{-\eta }g(\ln t)$ with $\eta $ from (\ref{10}) is in $%
C_{0}^{\infty }(0,\infty ).$ By the change of variable $s:=\ln t,$ we obtain
the unweighted inequality $||g||_{r}\leq C(||g^{\prime
}||_{p}+||g||_{q}+||g||_{p})$ for every $g\in C_{0}^{\infty }(\Bbb{R}).$
With $g\neq 0$ chosen once and for all and $g(t)$ replaced by $g\left(
\lambda t\right) ,\lambda >0,$ it follows that $I_{1}\leq C(\lambda ^{\frac{1%
}{p^{\prime }}+\frac{1}{r}}I_{2}++\lambda ^{\frac{1}{r}-\frac{1}{q}%
}I_{3}+\lambda ^{\frac{1}{r}-\frac{1}{p}}I_{4})$ with $I_{1},...,I_{4}>0$
independent of $\lambda .$ Since $r<\min \{p,q\},$ the right-hand side tends
to $0$ with $\lambda ,$ which is absurd.

(vi) Argue as in (v) above, just noticing that now $\eta =0$ in (\ref{10}),
which produces the simpler $||g||_{r}\leq C(||g^{\prime }||_{p}+||g||_{q})$
when $g\in C_{0}^{\infty }(\Bbb{R}).$ Then, $I_{1}\leq C(\lambda ^{\frac{1}{
p^{\prime }}+\frac{1}{r}}I_{2}+\lambda ^{\frac{1}{r}-\frac{1}{q}}I_{3})$ for 
$\lambda >0$ by rescaling, which is absurd if $r<q.$
\end{proof}

As a corollary, we obtain that the embedding is often characterized by a
multiplicative rather than additive norm inequality (see also Section \ref
{CKNinequalities}).

\begin{corollary}
\label{cor2}Let $a,b,c\in \Bbb{R}$ and $1\leq p<\infty ,0<q,r<\infty $ be
such that $\frac{a+N}{q}\neq \frac{b-p+N}{p}.$ Then, $W_{\{a,b\}}^{1,(q,p)}(%
\Bbb{R}_{*}^{N})$ is continuously embedded into $L^{r}(\Bbb{R}%
^{N};|x|^{c}dx) $ if and only if $c$ is in the closed interval with
endpoints $c^{0}$ and $c^{1}$ and there is $C>0$ such that 
\begin{equation}
||u||_{c,r}\leq C||\nabla u||_{b,p}^{\theta _{c}}||u||_{a,q}^{1-\theta
_{c}},\qquad \forall u\in W_{\{a,b\}}^{1,(q,p)}(\Bbb{R}_{*}^{N}),  \label{11}
\end{equation}
where $\theta _{c}$ is given by (\ref{2}). The same property is true upon
replacing $W_{\{a,b\}}^{1,(q,p)}(\Bbb{R}_{*}^{N})$ by $\widetilde{W}%
_{\{a,b\}}^{1,(q,p)}$and (\ref{11}) by 
\begin{equation}
||u||_{c,r}\leq C||\partial _{\rho }u||_{b,p}^{\theta
_{c}}||u||_{a,q}^{1-\theta _{c}},\qquad \forall u\in \widetilde{W}%
_{\{a,b\}}^{1,(q,p)}.  \label{12}
\end{equation}
\end{corollary}

\begin{proof}
The sufficiency follows from the arithmetic-geometric inequality. We prove
the necessity for $\widetilde{W}_{\{a,b\}}^{1,(q,p)}.$ Similar arguments
work in the case of $W_{\{a,b\}}^{1,(q,p)}(\Bbb{R}_{*}^{N}).$

Suppose then that $\widetilde{W}_{\{a,b\}}^{1,(q,p)}\hookrightarrow L^{r}(%
\Bbb{R}^{N};|x|^{c}dx).$ By part (i) of Theorem \ref{th1}, $c$ is in the
closed interval with (distinct) endpoints $c^{0}$ and $c^{1}.$ Furthermore, $%
||u||_{c,r}\leq C(||u||_{a,q}+||\partial _{\rho }u||_{b,p})$ for every $u\in 
\widetilde{W}_{\{a,b\}}^{1,(q,p)}.$ In this inequality, replace $u(x)$ by $%
u(\lambda x)$ with $\lambda >0$ to get 
\begin{multline}
||u||_{c,r}\leq C\lambda ^{\frac{c+N}{r}-\frac{a+N}{q}}||u||_{a,q}+C\lambda
^{\frac{c+N}{r}-\frac{b-p+N}{p}}||\partial _{\rho }u||_{b,p}=  \label{13} \\
C\lambda ^{\theta _{c}\frac{c^{1}-c^{0}}{r}}||u||_{a,q}+C\lambda ^{(1-\theta
_{c})\frac{c^{0}-c^{1}}{r}}||\partial _{\rho }u||_{b,p}.
\end{multline}

If $c=c^{0}$ ($c=c^{1}$), then $\theta _{c}=0$ ($\theta _{c}=1$), so that $%
||u||_{c,r}\leq C||u||_{a,q}$ ($||u||_{c,r}\leq C||\partial _{\rho
}u||_{b,p} $), i.e., (\ref{12}) holds, by letting $\lambda $ tend to $0$ or
to $\infty . $ Otherwise, (\ref{12}) follows by minimizing the right-hand
side of (\ref{13}) for $\lambda >0.$ This changes $C,$ which however remains
independent of $u$ even though the minimizer is of course $u$-dependent. (If 
$\theta _{c}\neq 0,$ (\ref{13}) shows that $u=0$ if $\partial _{\rho }u=0,$
so that it is not restrictive to assume $||u||_{a,q}>0$ and $||\partial
_{\rho }u||_{b,p}>0$ in the minimization step.)
\end{proof}

The next theorem gives a different necessary condition for the continuity of
the embedding $W_{\{a,b\}}^{1,(q,p)}(\Bbb{R}_{*}^{N})\hookrightarrow L^{r}(%
\Bbb{R}^{N};|x|^{c}dx).$

\begin{theorem}
\label{th3}Let $a,b,c\in \Bbb{R}$ and $1\leq p<\infty ,0<q,r<\infty $ be
given. \newline
(i) If $\frac{a+N}{q}\neq \frac{b-p+N}{p}$ and $W_{\{a,b\}}^{1,(q,p)}(\Bbb{R}%
_{*}^{N})\hookrightarrow L^{r}(\Bbb{R}^{N};|x|^{c}dx),$ then $\theta _{c}\in
[0,1]$ and 
\begin{equation}
\theta _{c}\left( \frac{1}{p}-\frac{1}{N}-\frac{1}{q}\right) \leq \frac{1}{r}%
-\frac{1}{q}.  \label{14}
\end{equation}
In particular, $r\leq \max \{p^{*},q\}.$\newline
(ii) If $\frac{a+N}{q}=\frac{b-p+N}{p}$ and $c=c^{0}$ ($=c^{1}$) and if $%
W_{\{a,b\}}^{1,(q,p)}(\Bbb{R}_{*}^{N})\hookrightarrow L^{r}(\Bbb{R}%
^{N};|x|^{c}dx),$ then $r\leq \max \{p^{*},q\}.$
\end{theorem}

\begin{proof}
(i) Part (i) of Theorem \ref{th1} shows that $\theta _{c}\in [0,1].$ The
next argument is taken from \cite{CaKoNi84}, with a minor adjustment to fit
the setting of this paper. Let $\varphi \in C_{0}^{\infty }(\Bbb{R}
^{N}),\varphi \neq 0,$ be chosen once and for all. If $x_{0}\in \Bbb{R}^{N}$
and $R:=|x_{0}|$ is large enough, then $\varphi (\cdot +x_{0})\in
C_{0}^{\infty }(\Bbb{R}_{*}^{N})\subset W_{\{a,b\}}^{1,(q,p)}(\Bbb{R}%
_{*}^{N})$ irrespective of $a,b,p$ and $q.$ By using (\ref{11}) with $%
u=\varphi (\cdot +x_{0})$ and by letting $R\rightarrow \infty ,$ we get
(because $\limfunc{Supp}\varphi $ is compact) $R^{\frac{c}{r}}||\varphi
||_{r}\leq CR^{\frac{b\theta _{c}}{p}+\frac{a(1-\theta _{c})}{q}}||\nabla
\varphi ||_{p}^{\theta _{c}}||\varphi ||_{q}^{1-\theta _{c}}$ for large $R$
after changing $C,$ whence $\frac{c}{r}\leq \frac{b\theta _{c}}{p}+\frac{%
a(1-\theta _{c})}{q}.$ Then, (\ref{14}) follows by adding $\frac{N}{r}$ and
using (\ref{4}).

If $p<N$ and $r>\max \{p^{*},q\},$ then (\ref{14}) cannot hold since it
fails when $\theta _{c}=0$ and when $\theta _{c}=1.$ Thus, $r\leq \max
\{p^{*},q\}$ is necessary.

(ii) Use the same method as in (i), but with the additive inequality $%
||\varphi ||_{c^{0},r}\leq C(||\varphi ||_{a,q}+||\nabla \varphi ||_{b,p}).$
This yields $R^{\frac{c^{0}}{r}}||\varphi ||_{r}\leq C(R^{\frac{a}{q}
}||\varphi ||_{q}+R^{\frac{b}{p}}||\nabla \varphi ||_{p})$ for large $R>0.$
By (\ref{1}), $\frac{c^{0}}{r}=\frac{a}{q}+\frac{N}{q}-\frac{N}{r}$ and
(since $c^{0}=c^{1}$) $\frac{b}{p}=\frac{a}{q}+\frac{N}{q}+1-\frac{N}{p},$
whence $R^{\frac{N}{q}-\frac{N}{r}}||\varphi ||_{r}\leq C(||\varphi
||_{q}+R^{\frac{N}{q}+1-\frac{N}{p}}||\nabla \varphi ||_{p}).$ If $r>q,$
this implies $\frac{N}{q}-\frac{N}{r}\leq \frac{N}{q}+1-\frac{N}{p},$ i.e., $%
r\leq p^{*},$ so that $r\leq \max \{p^{*},q\}$ in all cases.
\end{proof}

The above proof may give the wrong impression that (\ref{14}) arises only as
a result of integrability at infinity. That this is not the case can be seen
by noticing that the choice $\varphi (x|x|^{-2}+x_{0})$ instead of $\varphi
(x+x_{0})$ also yields (\ref{14}), while the support of $\varphi
(x|x|^{-2}+x_{0})$ shrinks towards $0$ as $|x_{0}|\rightarrow \infty .$

The verification that Theorem \ref{th1} and Theorem \ref{th3} together imply
that the hypotheses made in Theorem \ref{th0} are necessary is routine and
left to the reader.

\section{The spaces $\widetilde{W}_{loc}^{1,p}$ and related concepts\label%
{spaces}}

In this section, we develop the background material needed for the proofs of
the main results of the next two sections. Let $\omega _{N}$ denote the
volume of the unit ball of $\Bbb{R}^{N}.$ If $u\in L_{loc}^{p}(\Bbb{R}%
_{*}^{N})$ with $p\geq 1,$ define the spherical mean of $u$ 
\begin{equation}
f_{u}(t):=(N\omega _{N})^{-1}\int_{\Bbb{S}^{N-1}}u(t\sigma )d\sigma .
\label{15}
\end{equation}
By Fubini's theorem in spherical coordinates, $f_{u}(t)$ is defined for a.e. 
$t>0$ and $f_{u}\in L_{loc}^{p}(0,\infty ).$ If $u\in \widetilde{W}%
_{loc}^{1,p},$ where 
\begin{equation*}
\widetilde{W}_{loc}^{1,p}:=\{u\in L_{loc}^{p}(\Bbb{R}_{*}^{N}):\partial
_{\rho }u\in L_{loc}^{p}(\Bbb{R}_{*}^{N})\}
\end{equation*}
and $\partial _{\rho }u:=\nabla u\cdot \frac{x}{|x|},$ more is true:

\begin{lemma}
\smallskip \label{lm4}If $1\leq p<\infty $ and $u\in \widetilde{W}%
_{loc}^{1,p},$ then $f_{u}\in W_{loc}^{1,p}(0,\infty ).$ Furthermore, 
\begin{equation}
f_{u}^{\prime }(t)=(N\omega _{N})^{-1}\int_{\Bbb{S}^{N-1}}\partial _{\rho
}u(t\sigma )d\sigma .  \label{16}
\end{equation}
Conversely, if $f\in W_{loc}^{1,p}(0,\infty )$ and $u(x):=f(|x|),$ then $%
u\in \widetilde{W}_{loc}^{1,p}$ and $f_{u}=f,\partial _{\rho }u(x)=f^{\prime
}(|x|).$
\end{lemma}

\begin{proof}
Let $u\in \widetilde{W}_{loc}^{1,p}.$ If $\varphi \in C_{0}^{\infty
}(0,\infty ),$ set $\psi (x):=\varphi (|x|),$ so that $\psi \in
C_{0}^{\infty }(\Bbb{R}_{*}^{N})$ and $\partial _{\rho }\psi (x)=\varphi
^{\prime }(|x|).$ It follows that $\langle f_{u}^{\prime },\varphi \rangle
=-(N\omega _{N})^{-1}\left\langle u,|x|^{1-N}\partial _{\rho }\psi
\right\rangle =(N\omega _{N})^{-1}\left\langle |x|^{1-N}\partial _{\rho
}u,\psi \right\rangle $ (use $\nabla \cdot \left( |x|^{-N}x\right) =0$).
Since $\partial _{\rho }u\in L_{loc}^{p}(\Bbb{R}_{*}^{N}),$ this shows that $%
\langle f_{u}^{\prime },\varphi \rangle =\langle f_{\partial _{\rho
}u},\varphi \rangle ,$ that is, $f_{u}^{\prime }=f_{\partial _{\rho }u}\in
L_{loc}^{p}(0,\infty ).$ Thus, $f_{u}\in W_{loc}^{1,p}(0,\infty )$ and (\ref
{16}) holds.

Conversely, suppose that $f\in W_{loc}^{1,p}(0,\infty )$ and set $%
u(x):=f(|x|).$ Then, $u\in L_{loc}^{p}(\Bbb{R}_{*}^{N})$ (it is continuous)
and, by \cite[Theorem 4.3]{MaMi72}, $\nabla u(x)=f^{\prime }(|x|)\frac{x}{|x|%
}$ because $f$ is locally absolutely continuous. Thus, $u\in W_{loc}^{1,p}(%
\Bbb{R}_{*}^{N})\subset \widetilde{W}_{loc}^{1,p}$ and $f^{\prime
}(|x|)=\nabla u(x)\cdot \frac{x}{|x|}$ $=\partial _{\rho }u(x).$ That $%
f_{u}=f$ is obvious.
\end{proof}

If $\Omega $ is an open subset of $\Bbb{R}^{N}$ and $u\in W^{1,1}(\Omega ),$
it is well-known that $|u|\in W^{1,1}(\Omega )$ with $\nabla |u|=(\limfunc{
sgn}u)\nabla u$ (see for instance \cite[p. 48]{Zi89} or \cite[Theorem 2.2]
{MaMi72} for more general statements), where $\limfunc{sgn}u$ is defined to
be $0$ at points where $u=0.$ This is proved by showing that if $u\in
L^{1}(\Omega )$ and $\partial _{i}u\in L^{1}(\Omega )$ for some index $1\leq
i\leq N,$ then $\partial _{i}|u|\in L^{1}(\Omega )$ and $\partial _{i}|u|=(%
\limfunc{sgn}u)\partial _{i}u,$ because the assumptions suffice to ensure
the local absolute continuity of $u$ on almost every line segment in $\Omega 
$ parallel to the $x_{i}$-axis. Since a radial derivative is just a
directional derivative after passing to spherical coordinates, the same
arguments show that if $u\in \widetilde{W}_{loc}^{1,1},$ then $|u|\in 
\widetilde{W}_{loc}^{1,1}$ and $\partial _{\rho }|u|=(\limfunc{sgn}
u)\partial _{\rho }u.$ (That the derivative of $u(\cdot ,\sigma )$ is $%
\partial _{\rho }u(\cdot ,\sigma )$ can be justified by a variant of the
proof of Lemma \ref{lm4}.)

Another well-known result, usually proved by localization and mollification,
is that if $u\in W^{1,p}(\Omega )$ and $u\geq 0,$ then $u^{p}\in
W^{1,1}(\Omega )$ and $\partial _{i}(u^{p})=pu^{p-1}\partial _{i}u.$ Not
surprisingly, the proof actually requires only $u$ and $\partial _{i}u$ to
be in $L^{p}(\Omega ),$ so that completely similar arguments show that if $%
u\in \widetilde{W}_{loc}^{1,p}$ and $u\geq 0,$ then $u^{p}\in \widetilde{W}
_{loc}^{1,1}$ and $\partial _{\rho }u^{p}=pu^{p-1}\partial _{\rho }u.$ By
combining the above, we find:

\begin{lemma}
\label{lm5}If $1\leq p<\infty $ and $u\in \widetilde{W}_{loc}^{1,p},$ then $%
f_{u}\in W_{loc}^{1,p}(0,\infty )$ and $f_{u}^{\prime }$ is given by (\ref
{16}). Furthermore, $|u|^{p}\in \widetilde{W}_{loc}^{1,1}$and $\partial
_{\rho }(|u|^{p})=p|u|^{p-1}(\limfunc{sgn}u)\partial _{\rho }u,$ where $%
\limfunc{sgn}u:=0$ on $u^{-1}(0).$
\end{lemma}

Since $f_{|u|}$ is continuous on $(0,\infty )$ when $u\in \widetilde{W}%
_{loc}^{1,1},$ the following two subsets are well defined: 
\begin{equation}
\widetilde{W}_{loc,-}^{1,1}:=\{u\in \widetilde{W}_{loc}^{1,1}:\underline{%
\lim }_{t\rightarrow \infty }f_{|u|}(t)=0\},  \label{17}
\end{equation}
\begin{equation}
\widetilde{W}_{loc,+}^{1,1}:=\{u\in \widetilde{W}_{loc}^{1,1}:\underline{%
\lim }_{t\rightarrow 0^{+}}f_{|u|}(t)=0\}.  \label{18}
\end{equation}
The sets $\widetilde{W}_{loc,\pm }^{1,1}$ are not closed under addition and
so are not vector spaces. They are exchanged into one another by Kelvin
transform. Various other properties are collected in the next lemma.

\begin{lemma}
\label{lm6}The following properties hold:\newline
(i) If $u\in \widetilde{W}_{loc,-}^{1,1}$ ($\widetilde{W}_{loc,+}^{1,1}$),
then $|u|\in \widetilde{W}_{loc,-}^{1,1}$ ($\widetilde{W}_{loc,+}^{1,1}$).%
\newline
(ii) $u\in \widetilde{W}_{loc,-}^{1,1}$ ($\widetilde{W}_{loc,+}^{1,1}$) $%
\Rightarrow u_{S}:=f_{u}\circ |\cdot |\in \widetilde{W}_{loc,-}^{1,1}$ ($%
\widetilde{W}_{loc,+}^{1,1}$) and $\partial _{\rho }u_{S}(x)=f_{u}^{\prime
}(|x|).$\newline
(iii) If $u\in \widetilde{W}_{loc}^{1,1}$ and $|x|^{a}|u|^{q}\in L^{1}(\Bbb{R%
}^{N})$ for some $a\in \Bbb{R}$ and some $q\geq 1,$ then $u\in \widetilde{W}%
_{loc,-}^{1,1}$ ($\widetilde{W}_{loc,+}^{1,1}$) if $a\geq -N$ ($a\leq -N$).
In particular (see (\ref{7})), $\widetilde{W}_{\{a,b\}}^{1,(q,p)}\subset 
\widetilde{W}_{loc,-}^{1,1}$ ($\widetilde{W}_{loc,+}^{1,1}$) if $a\geq -N$ ($%
a\leq -N$). \newline
(iv) If $u\in \widetilde{W}_{loc}^{1,1}$ is radially symmetric and $%
|x|^{a}|u|^{q}\in L^{1}(\Bbb{R}^{N})$ for some $a\in \Bbb{R}$ and $q>0,$
then $u\in \widetilde{W}_{loc,-}^{1,1}$ ($\widetilde{W}_{loc,+}^{1,1}$) if $%
a\geq -N$ ($a\leq -N$) In particular, if $u\in \widetilde{W}%
_{\{a,b\}}^{1,(q,p)}$ is radially symmetric, then $u\in \widetilde{W}%
_{loc,-}^{1,1}$ ($\widetilde{W}_{loc,+}^{1,1}$) if $a\geq -N$ ($a\leq -N$).
\end{lemma}

\begin{proof}
(i) Use Lemma \ref{lm5} and the definitions (\ref{17}) and (\ref{18}).

(ii) That $u_{S}:=f_{u}\circ |\cdot |\in \widetilde{W}_{loc}^{1,1}$ and $%
\partial _{\rho }u_{S}(x)=f_{u}^{\prime }(|x|)$ follows from Lemma \ref{lm4}%
. Next, the remark that $f_{|u_{S}|}=|f_{u}|\leq f_{|u|}$ shows that if also 
$\underline{\lim }_{t\rightarrow \infty }f_{|u|}(t)=0$ (or $\underline{\lim }%
_{t\rightarrow 0^{+}}f_{|u|}(t)=0$), then $\underline{\lim }_{t\rightarrow
\infty }f_{|u_{S}|}(t)=0$ (or $\underline{\lim }_{t\rightarrow \infty
}f_{|u_{S}|}(t)(t)=0$).

(iii) Suppose $a\geq -N$ and, by contradiction, $u\notin \widetilde{W}%
_{loc,-}^{1,1}.$ Then, $f_{|u|}(t)\geq \ell >0$ for $t\geq T$ and large $%
T>0. $ Thus, by (\ref{15}), $\ell ^{q}\leq (f_{|u|})^{q}(t)\leq
f_{|u|^{q}}(t)$ for $t\geq T,$ so that $\int_{|x|\geq
T}|x|^{a}|u|^{q}=N\omega _{N}\int_{T}^{\infty }t^{a+N-1}f_{|u|^{q}}(t)dt\geq
N\omega _{N}\ell ^{q}\int_{T}^{\infty }t^{a+N-1}dt=\infty $ since $a\geq -N.$
This contradicts $|x|^{a}|u|^{q}\in L^{1}(\Bbb{R}^{N}).$ The case when $%
a\leq -N$ follows by Kelvin transform and$.$ the ``in particular'' part is
obvious.

(iv) If $u$ is radially symmetric, then $f_{|u|^{q}}=(f_{|u|})^{q}$ for
every $q>0,$ so that the contradiction argument in the proof of (iii) works
when $q>0,$ not just $q\geq 1.$ The ``in particular'' part is clear if we
show that $u\in \widetilde{W}_{loc}^{1,1}.$ To see this, note that $u\in 
\widetilde{W}_{\{a,b\}}^{1,(q,p)}$ implies $\partial _{\rho }u\in
L_{loc}^{p}(\Bbb{R}_{*}^{N}),$ which, by radial symmetry, implies $\nabla
u\in L_{loc}^{p}(\Bbb{R}_{*}^{N}).$ Thus, $u\in W_{loc}^{1,p}(\Bbb{R}
_{*}^{N})$ (\cite[p. 7]{Ma85}) and $W_{loc}^{1,p}(\Bbb{R}_{*}^{N})\subset 
\widetilde{W}_{loc}^{1,1}$ is obvious.
\end{proof}

If $u\in L_{loc}^{1}(\Bbb{R}_{*}^{N})$ is radially symmetric, then $%
u(x)=f_{u}(|x|).$ This justifies referring to the function $u_{S}$ in part
(ii) of Lemma \ref{lm6} as the ``radial symmetrization'' of $u.$

\begin{lemma}
\label{lm7}Let $a,b\in \Bbb{R}$ and $1\leq p,q<\infty $ be given. If $u\in 
\widetilde{W}_{\{a,b\}}^{1,(q,p)},$ then:\newline
(i) $|u|\in \widetilde{W}_{\{a,b\}}^{1,(q,p)}$ and $||\,|u|%
\,||_{a,q}=||u||_{a,q},$ $||\partial _{\rho }|u|\,||_{b,p}=||\,\partial
_{\rho }u\,||_{b,p}.$ If also $u$ is radially symmetric, this remains true
when $0<q<1.$\newline
(ii) $u_{S}\in \widetilde{W}_{\{a,b\}}^{1,(q,p)}$ and $||u_{S}||_{a,q}\leq
||u||_{a,q},||\,\partial _{\rho }u_{S}||_{b,p}\leq ||\partial _{\rho
}u||_{b,p}.$\newline
\end{lemma}

\begin{proof}
(i) This follows from $u\in \widetilde{W}_{\{a,b\}}^{1,(q,p)}\subset 
\widetilde{W}_{loc}^{1,1}$ (see Lemma \ref{lm6} (iv) if $u$ is radially
symmetric and $0<q<1$) so that $\partial _{\rho }|u|=(\limfunc{sgn}
u)\partial _{\rho }u$ by Lemma \ref{lm5}.

(ii) Since $u_{S}(x)=f_{u}(|x|)$ and $f_{u}$ in (\ref{15}) is continuous, $%
u_{S}$ is continuous and so $u_{S}\in L_{loc}^{1}(\Bbb{R}_{*}^{N}).$ By (\ref
{15}), $|u_{S}(x)|^{q}\leq (N\omega _{N})^{-1}\int_{\Bbb{S}%
^{N-1}}|u(|x|\sigma )|^{q}d\sigma $ since $q\geq 1$ and, by (\ref{16}) and
part (ii) of Lemma \ref{lm6}, $|\partial _{\rho }u_{S}(x)|^{p}\leq $ $%
(N\omega _{N})^{-1}\int_{\Bbb{S}^{N-1}}|\partial _{\rho }u(|x|\sigma
)|^{p}d\sigma $ a.e. Therefore, $||u_{S}||_{a,q}\leq ||u||_{a,q}$ and $%
||\partial _{\rho }u_{S}||_{b,p}\leq ||\partial _{\rho }u||_{b,p}.$
\end{proof}

We complete this section with an inequality (Theorem \ref{th9}) which is the
basic tool for the proof of Lemmas \ref{lm12} and \ref{lm13} in the next
section.

\begin{lemma}
\label{lm8}Let $f\in W_{loc}^{1,1}(0,\infty ),f\geq 0$ and $\gamma \in \Bbb{R%
}$ be given. \newline
(i) If $\gamma \geq 1-N$ and $\underline{\lim }_{t\rightarrow \infty
}f(t)=0, $ then 
\begin{equation}
0\leq t^{N-1+\gamma }f(t)\leq \int_{t}^{\infty }\tau ^{N-1+\gamma
}|f^{\prime }(\tau )|d\tau \leq \infty ,\qquad \forall t>0.  \label{19}
\end{equation}
\newline
(ii) If $\gamma \leq 1-N$ and $\underline{\lim }_{t\rightarrow 0^{+}}f(t)=0,$
then 
\begin{equation}
0\leq t^{N-1+\gamma }f(t)\leq \int_{0}^{t}\tau ^{N-1+\gamma }|f^{\prime
}(\tau )|d\tau \leq \infty ,\qquad \forall t>0.  \label{20}
\end{equation}
\newline
\end{lemma}

\begin{proof}
(i) Given $t>0,$ let $T>t$ and write $f(t)=f(T)-\int_{t}^{T}f^{\prime }(\tau
)d\tau .$ Since $\gamma \geq 1-N$ implies $t^{N-1+\gamma }\leq \tau
^{N-1+\gamma }$ when $t\leq \tau ,$ this yields $t^{N-1+\gamma }f(t)\leq
t^{N-1+\gamma }f(T)+\int_{t}^{T}\tau ^{N-1+\gamma }|f^{\prime }(\tau )|d\tau
\leq t^{N-1+\gamma }f(T)+\int_{t}^{\infty }\tau ^{N-1+\gamma }|f^{\prime
}(\tau )|d\tau .$ Thus, (\ref{19}) follows from $f\geq 0$ and from $%
\underline{\lim }_{T\rightarrow \infty }f(T)=0.$

(ii) Given $t>0,$ let $0<\varepsilon <t$ and write $f(t)=f(\varepsilon
)+\int_{\varepsilon }^{t}f^{\prime }(\tau )d\tau .$ Since $\gamma \leq 1-N$
implies $t^{N-1+\gamma }\leq \tau ^{N-1+\gamma }$ when $t\geq \tau ,$ this
yields $t^{N-1+\gamma }f(t)\leq t^{N-1+\gamma }f(\varepsilon
)+\int_{\varepsilon }^{t}\tau ^{N-1+\gamma }|f^{\prime }(\tau )|d\tau \leq
t^{N-1+\gamma }f(\varepsilon )+\int_{0}^{t}\tau ^{N-1+\gamma }|f^{\prime
}(\tau )|d\tau .$ Thus, (\ref{20}) follows from $f\geq 0$ and from $%
\underline{\lim }_{\varepsilon \rightarrow 0}f(\varepsilon )=0.$
\end{proof}

In Theorem \ref{th9} below, the norm notation is only used for convenience
since all the norms may actually be infinite. In practice, this simply means
that in the inequalities, the finiteness of the right-hand side implies the
finiteness of the left-hand side, which therefore need not be assumed
separately. An alternate proof can be based on the case ``$q=\infty $'' of 
\cite[Theorem 2, p.40]{Ma85} and Kelvin transform, but the direct argument
used below is more explicit and not longer.

\begin{theorem}
\label{th9}Let $\gamma \in \Bbb{R}$ and $1\leq p<\infty $ be given. There is
a constant $C>0$ such that if $u\in \widetilde{W}_{loc}^{1,1}$ is radially
symmetric and either $\gamma >1-N$ and $u\in \widetilde{W}_{loc,-}^{1,1}$ or 
$\gamma <1-N$ and $u\in \widetilde{W}_{loc,+}^{1,1},$ then $%
||\,|x|^{N-1+\gamma }u||_{\infty }\leq C||\,|x|^{\gamma +\frac{N}{p^{\prime }%
}}\partial _{\rho }u||_{p}.$\newline
Furthermore, if $p=1,$ this inequality remains true when $\gamma =1-N.$
\end{theorem}

\begin{proof}
Suppose first $p=1$ and $\gamma \geq 1-N$ and let $u\in \widetilde{W}
_{loc,-}^{1,1}.$ By part (i) of Lemma \ref{lm6} and Lemma \ref{lm5}, we may
and shall assume $u\geq 0$ with no loss of generality since $||\,|x|^{\gamma
}\partial _{\rho }u||_{1}$ and $||\,|x|^{N-1+\gamma }u||_{\infty }$ are
unchanged when $u$ is replaced by $|u|.$

By Lemma \ref{lm4}, $u(x)=f_{u}(|x|)$ with $f_{u}\in \widetilde{W}%
_{loc}^{1,1}(0,\infty ),f_{u}\geq 0$ and $\underline{\lim }_{t\rightarrow
\infty }f_{|u|}(t)=0$ by (\ref{17}). Thus, $||\,|x|^{N-1+\gamma }u||_{\infty
}=\sup_{t>0}t^{N-1+\gamma }f_{u}(t)$ and $\left\| \,\,|x|^{\gamma }\partial
_{\rho }u\right\| _{1}=$\linebreak $\int_{0}^{\infty }\tau ^{N-1+\gamma
}|f_{u}^{\prime }(\tau )|d\tau $ since $f_{u}^{\prime }(|x|)=\partial _{\rho
}u(x)$ (use $u=u_{S}$ and Lemma \ref{lm6} (ii)). Hence, it suffices to show
that $t^{N-1+\gamma }f_{u}(t)\leq \int_{0}^{\infty }\tau ^{N-1+\gamma
}|f_{u}^{\prime }(\tau )|d\tau \leq \infty $ for every $t>0,$ which follows
at once from (\ref{19}) for $f=f_{u}.$ If $\gamma \leq 1-N$ and $u\in 
\widetilde{W}_{loc,+}^{1,1},$ use (\ref{20}) instead of (\ref{19}).

Now, let $1<p<\infty .$ Once again we assume $u\geq 0$ with no loss of
generality, so that $u(x)=f_{u}(|x|)$ with $f_{u}\in W_{loc}^{1,1}(0,\infty
) $ and $f_{u}\geq 0.$ It suffices to prove 
\begin{equation}
t^{N-1+\gamma }f_{u}(t)\leq C\int_{0}^{\infty }|f_{u}^{\prime }(\tau
)|^{p}\tau ^{pN+p\gamma -1}d\tau ,  \label{21}
\end{equation}
for every $t>0.$ We merely show how the proof when $p=1$ above can be
modified to yield this inequality.

Suppose $\gamma >1-N$ and let $u\in \widetilde{W}_{loc,-}^{1,1}.$ The
inequality (\ref{19}) with $\gamma =1-N$ -which is allowed in Lemma \ref{lm8}
- and $f=f_{u}$ yields $f_{u}(t)\leq \int_{t}^{\infty }|f_{u}^{\prime }(\tau
)|d\tau $ for every $t>0.$ Write $|f_{u}^{\prime }(\tau )|=\left(
|f_{u}^{\prime }(\tau )|\tau ^{N-1+\gamma +\frac{1}{p^{\prime }}}\right)
\tau ^{1-N-\gamma -\frac{1}{p^{\prime }}}$ and, since $\gamma >1-N,$ use
H\"{o}lder's inequality to get $f_{u}(t)\leq Ct^{1-N-\gamma }\left(
\int_{t}^{\infty }|f_{u}^{\prime }(\tau )|^{p}\tau ^{pN+p\gamma -1}d\tau
\right) ^{\frac{1}{p}}$ with $C:=[p^{\prime }(\gamma +N-1)]^{-\frac{1}{%
p^{\prime }}},$ which is stronger than (\ref{21}). If $\gamma <1-N$ and $%
u\in \widetilde{W}_{loc,+}^{1,1},$ follow the same procedure, but starting
with the inequality (\ref{20}).
\end{proof}

\section{Embedding theorem for radially symmetric functions\label{radial}}

In this section, we give necessary and sufficient conditions for the
continuity of the embedding of the subspace of $\widetilde{W}
_{\{a,b\}}^{1,(q,p)}$ of radially symmetric functions into $L^{r}(\Bbb{R}
^{N};|x|^{c}dx).$ In principle, this can of course be done by reduction to
the half-line, which is reflected in the proofs, but we have found no
expository or technical advantage in doing so explicitly. Our first task
will be to make sure that the cut-off operation is continuous. As a
preamble, we need:

\begin{lemma}
\label{lm10}Let $\Omega $ denote a bounded open annulus centered at $0\notin 
\overline{\Omega }$ and let $a,b\in \Bbb{R}$ and $1\leq p<\infty ,0<q<\infty 
$ be given$.$ There is a constant $C>0$ such that $||u||_{p,\Omega }\leq
C||u||_{\{a,b\},(q,p)}$ for every radially symmetric $u\in \widetilde{W}%
_{\{a,b\}}^{1,(q,p)}.$
\end{lemma}

\begin{proof}
Let $u\in \widetilde{W}_{\{a,b\}}^{1,(q,p)}$ be radially symmetric. We
already pointed out in the Introduction that $\widetilde{W}
_{\{a,b\}}^{1,(q,p)}$ and $W_{\{a,b\}}^{1,(q,p)}(\Bbb{R}_{*}^{N})$ have the
same radially symmetric functions, with the same induced (quasi) norms.
Since $u\in W_{\{a,b\}}^{1,(q,p)}(\Bbb{R}_{*}^{N})$ implies $\nabla u\in
L_{loc}^{p}(\Bbb{R}_{*}^{N}),$ it follows that $u\in W_{loc}^{1,p}(\Bbb{R}
_{*}^{N})$ (this was already used in the proof of Lemma \ref{lm6} (iv)) and
hence that $u\in W^{1,p}(\Omega ).$ Thus, it suffices to prove that $%
||v||_{p,\Omega }\leq C(||v||_{q,\Omega }+||\nabla v||_{p,\Omega })$ for
every $v\in W^{1,p}(\Omega ).$

This is common knowledge when $q\geq 1,$ but since only $q>0$ is assumed, we
give a proof for completeness. By contradiction, assume that there is a
sequence $(v_{n})\subset W^{1,p}(\Omega )$ such that $||v_{n}||_{p,\Omega
}=1 $ and $\lim_{n\rightarrow \infty }||v_{n}||_{q,\Omega }+||\nabla
v_{n}||_{p,\Omega }=0.$ Since $(v_{n})$ is bounded in $W^{1,p}(\Omega )$ and
the embedding $W^{1,p}(\Omega )\hookrightarrow L^{p}(\Omega )$ is compact
(even when $p=1$), there is $v\in L^{p}(\Omega )$ and a subsequence, still
denoted by $(v_{n}),$ such that $v_{n}\rightarrow v$ in $L^{p}(\Omega )$ and
that $v_{n}\rightarrow v$ a.e. on $\Omega .$ Obviously, $||v||_{p}=1.$

Now, since $|v_{n}|^{q}\rightarrow 0$ in $L^{1}(\Omega ),$ there is a
subsequence $(v_{n_{k}})$ such that $|v_{n_{k}}|^{q}\rightarrow 0$ a.e. on $%
\Omega .$ Thus, $v_{n_{k}}\rightarrow 0$ a.e. on $\Omega ,$ so that $v=0,$
which contradicts $||v||_{p}=1.$
\end{proof}

With the help of Lemma \ref{lm10}, we can now prove that truncation has the
expected properties in the subspace of $\widetilde{W}_{\{a,b\}}^{1,(q,p)}$
of radially symmetric functions.

\begin{lemma}
\label{lm11}Let $a,b\in \Bbb{R}$ and $1\leq p<\infty ,0<q<\infty $ be given
and let $\varphi \in C^{\infty }(\Bbb{R}^{N})$ be radially symmetric,
constant on a neighborhood of $0$ and constant outside a ball. Then, the
multiplication by $\varphi $ is continuous on the subspace of radially
symmetric functions of $\widetilde{W}_{\{a,b\}}^{1,(q,p)}.$
\end{lemma}

\begin{proof}
If $u\in \widetilde{W}_{\{a,b\}}^{1,(q,p)},$ then $||\varphi u||_{a,q}\leq
||\varphi ||_{\infty }||u||_{a,q}$ and $\partial _{\rho }(\varphi u)=\varphi
\partial _{\rho }u+(\partial _{\rho }\varphi )u.$ Clearly, $||\varphi
\partial _{\rho }u||_{b,p}\leq ||\varphi ||_{\infty }||\partial _{\rho
}u||_{b,p}.$ To evaluate $||(\partial _{\rho }\varphi )u||_{b,p}$ when $u$
is radially symmetric, note that $\limfunc{Supp}\partial _{\rho }\varphi $
is contained in a bounded open annulus $\Omega $ centered at $0\notin 
\overline{\Omega }.$ Thus, $||(\partial _{\rho }\varphi )u||_{b,p}\leq
C||\partial _{\rho }\varphi ||_{\infty }||u||_{\{a,b\},(q,p)}$ by Lemma \ref
{lm10} since $|x|^{b}$ is bounded on $\Omega .$ Altogether, this yields $%
||\varphi u||_{\{a,b\},(q,p)}\leq C||u||_{\{a,b\},(q,p)}.$
\end{proof}

The radial symmetry is unimportant in Lemmas \ref{lm10} and \ref{lm11} if $%
q\geq p$ or if $\widetilde{W}_{\{a,b\}}^{1,(q,p)}$ is replaced by $%
W_{\{a,b\}}^{1,(q,p)}(\Bbb{R}_{*}^{N}),$ but it does matter if $q<p.$

We first address the embedding when $a$ and $b-p$ are on the same side of $%
-N.$

\begin{lemma}
\label{lm12}Let $a,b,c\in \Bbb{R}$ and $1\leq p<\infty ,0<q,r<\infty $ be
given. If $a$ and $b-p$ are on the same side of $-N$ (including $-N$), the
subspace of $\widetilde{W}_{\{a,b\}}^{1,(q,p)}$ of radially symmetric
functions is continuously embedded into $L^{r}(\Bbb{R}^{N};|x|^{c}dx)$ in
the following two cases (recall the definition of $c^{0}$ and $c^{1}$ in (%
\ref{1})): \newline
(i) $\frac{a+N}{q}\neq \frac{b-p+N}{p},r\leq q$ and $c$\thinspace is in the
open interval with\thinspace endpoints $c^{0}$\thinspace and\thinspace
\nolinebreak $\nolinebreak c^{1}.$\newline
(ii) $\frac{a+N}{q}\neq \frac{b-p+N}{p},b-p\neq -N$ if $p>1,r>q$ and $c$ is
in the semi-open interval with endpoints $c^{*}:=\left( 1-\frac{q}{r}\right)
c^{1}+\frac{q}{r}c^{0}$ (included) and $c^{1}$ (not included).\newline
\end{lemma}

\begin{proof}
By Kelvin transform (Remark \ref{rm2}), we may assume $a\geq -N$ and $%
b-p\geq -N$ and, by Lemma \ref{lm7}, $u\geq 0.$ By Lemma \ref{lm11} and with 
$\zeta $ as in subsection \ref{notation}, it suffices to show that $%
||(1-\zeta )u||_{c,r}\leq C||(1-\zeta )u||_{\{a,b\},(q,p)}$ and that $%
||\zeta u||_{c,r}\leq C||\zeta u||_{\{a,b\},(q,p)}$ for some constant $C>0$
independent of $u.$

(i) The assumption $0<r\leq q$ is retained.

\textit{Case (i-{\scriptsize 1}):} $b-p>-N$ or $p=1$ and $b-1\geq -N.$

We first prove $||v||_{c,r}\leq C||v||_{\{a,b\},(q,p)}$ when $v:=(1-\zeta )u$
($\geq 0$). Given $\xi \in \Bbb{R}$ and $c\in \Bbb{R},$ write $%
|x|^{c}v^{r}=|x|^{-\xi }\left( |x|^{c+\xi }v^{r}\right) .$ Since $\limfunc{
Supp}v\subset \Bbb{R}^{N}\backslash B(0,\frac{1}{2})$ and by H\"{o}lder's
inequality, $||v||_{c,r}^{r}=\int_{\Bbb{R}^{N}}|x|^{c}v^{r}\leq \left( \int_{%
\Bbb{R}^{N}\backslash B(0,\frac{1}{2})}|x|^{-k^{\prime }\xi }\right) ^{\frac{
1}{k^{\prime }}}\left( \int_{\Bbb{R}^{N}}|x|^{k(c+\xi )}v^{kr}\right) ^{%
\frac{1}{k}},$ where $k>1$ is arbitrary.

If $k^{\prime }\xi >N,$ then $M_{k,\xi }:=\left( \int_{\Bbb{R}^{N}\backslash
B(0,\frac{1}{2})}|x|^{-k^{\prime }\xi }\right) ^{\frac{1}{k^{\prime }}
}<\infty $ and it suffices to find a majorization of $\int_{\Bbb{R}
^{N}}|x|^{k(c+\xi )}v^{kr}.$ Split $|x|^{k(c+\xi )}v^{kr}=\left(
|x|^{k(c+\xi )-a}v^{kr-q}\right) |x|^{a}v^{q},$ so that, if $kr-q>0,$ then $%
\int_{\Bbb{R}^{N}}|x|^{k(c+\xi )}v^{kr}\leq \left\| |x|^{k(c+\xi
)-a}v^{kr-q}\right\| _{\infty }\int_{\Bbb{R}^{N}}|x|^{a}v^{q}=\left\| |x|^{%
\frac{k(c+\xi )-a}{kr-q}}v\right\| _{\infty }^{kr-q}||v||_{a,q}^{q}.$

The next task is to majorize $\left\| |x|^{\frac{k(c+\xi )-a}{kr-q}
}v\right\| _{\infty }.$ This can be done by using Theorem \ref{th9}, as we
now explain. Suppose in addition that $k$ and $\xi $ are chosen so that $%
\frac{k(c+\xi )-a}{kr-q}=\frac{b-p+N}{p}.$ By part (iii) of Lemma \ref{lm6}, 
$v\in \widetilde{W}_{loc,-}^{1,1}$ since $a\geq -N.$ Next, if $\gamma :=%
\frac{b}{p}-\frac{N}{p^{\prime }},$ then $\gamma >1-N$ if $p>1$ since $%
b-p>-N $ and $\gamma \geq 1-N$ if $p=1$ since $b-1\geq -N.$ Thus, $\left\|
|x|^{\frac{b-p+N}{p}}v\right\| _{\infty }\leq C||\partial _{\rho
}v||_{b,p}<\infty $ by Theorem \ref{th9}. To summarize, 
\begin{equation}
||v||_{c,r}\leq M_{k,\xi }^{\frac{1}{r}}C^{1-\frac{q}{kr}}||\partial _{\rho
}v||_{b,p}^{1-\frac{q}{kr}}||v||_{a,q}^{\frac{q}{kr}},  \label{22}
\end{equation}
if $k$ and $\xi \in \Bbb{R}$ can be found such that $k^{\prime }\xi
>N,kr-q>0 $ (hence $k>1$ since $r\leq q$) and $\frac{k(c+\xi )-a}{kr-q}=%
\frac{b-p+N}{p}.$ By introducing $s:=kr-q>0,$ so that $k=\frac{s+q}{r},$ it
follows that $\frac{k(c+\xi )-a}{kr-q}=\frac{b-p+N}{p}$ if and only if $\xi =%
\frac{arp+rs(b-p+N)}{p(s+q),}-c$ and then $k^{\prime }\xi >N$ if and only if 
\begin{equation}
c<\frac{arp+rs(b-p+N)-Nps-Npq+Npr}{p(s+q)}=\frac{c^{1}s+c^{0}q}{s+q}.
\label{23}
\end{equation}

Thus, this inequality for some $s>0$ ensures that (\ref{22}) holds with $k:=%
\frac{s+q}{r}>1$ and $\xi =\frac{arp+rs(b-p+N)}{p(s+q)}-c.$ The right-hand
side of (\ref{23}) is a monotone function of $s>0$ with limits $c^{0}$ and $%
c^{1}$ as $s$ tends to $0$ and $\infty ,$ respectively. Therefore, $s>0$ can
be chosen so that (\ref{23}) holds if and only if $c<\max \left\{
c^{0},c^{1}\right\} $ and then, since $v=$ $(1-\zeta )u$ in (\ref{22}), the
arithmetic-geometric inequality yields $||(1-\zeta )u||_{c,r}\leq
C||(1-\zeta )u||_{\{a,b\},(q,p)}$ with $C>0$ independent of $u.$

If now $v:=\zeta u,$ then once again $v\in \widetilde{W}_{loc,-}^{1,1}$
because $v$ has bounded support. The same procedure, but with $k^{\prime
}\xi >N$ replaced by $k^{\prime }\xi <N,$ shows that $||v||_{c,r}=||\zeta
u||_{c,r}\leq C||u||_{\{a,b\},(q,p)}$ if $c>\min \left\{ c^{0},c^{1}\right\}
.$ Hence, both $||(1-\zeta )u||_{c,r}\leq C||u||_{\{a,b\},(q,p)}$ and $%
||\zeta u||_{c,r}\leq C||u||_{\{a,b\},(q,p)}$ hold when $c$ is in the open
interval with endpoints $c^{0}$ and $c^{1}.$

\textit{Case (i-{\scriptsize 2}):} $b-p=-N$ (and\footnote{%
The argument also works when $p=1.$} $p>1$).

If so, $\frac{a+N}{q}\neq \frac{b-p+N}{p}=0$ and $a\geq -N$ imply $a>-N$ and 
$-N=c^{1}<c<c^{0}.$ If $v:=(1-\zeta )u,$ then, $||v||_{c,r}\leq
C||v||_{a,q}\leq C||v||_{\{a,b\},(q,p)}$ by H\"{o}lder's inequality (use $%
|x|^{c}|v|^{r}=|x|^{c-\frac{ar}{q}}\left( |x|^{\frac{ar}{q}}(1-\zeta
)^{r}|u|^{r}\right) ,$ $\limfunc{Supp}(1-\zeta )\subset \Bbb{R}
^{N}\backslash B(0,\frac{1}{2})$ and $\frac{cq-ar}{q-r}<-N,$ i.e., $c<c^{0},$
if $r<q,$ or $c<a=c^{0}$ if $r=q$).

Next, choose $\hat{b}>b$ (so that $\hat{b}-p>-N$) such that $\hat{c}^{1}:=%
\frac{r(\hat{b}-p+N)}{p}-N<c$ and use Case (i-{\scriptsize 1}) with $b$
replaced by $\hat{b}$ -which changes $c^{1}$ into $\hat{c}^{1}$ but does not
change $c^{0}$- and $u$ replaced by $\zeta u.$ This yields $||\zeta
u||_{c,r}\leq C||\zeta u||_{\{a,\hat{b}\},(q,p)}\leq C||\zeta
u||_{\{a,b\},(q,p)}$ where the second inequality follows from $\hat{b}>b$
and $\limfunc{ Supp}\zeta \subset \overline{B}(0,1)$ (so that $||\nabla
(\zeta u)||_{\hat{b},p}\leq ||\nabla (\zeta u)||_{b,p}$).

(ii) The assumption $0<q<r$ is retained.

By part (iv) of Lemma \ref{lm6} and Lemma \ref{lm11}, $u,\zeta u$ and $%
(1-\zeta )u$ are in $\widetilde{W}_{loc}^{1,1}$ (even $\widetilde{W}
_{loc,-}^{1,1}$ since $a\geq -N$ and $\zeta u$ has bounded support) due to
radial symmetry, even when $q<1.$ Since $b-p\geq -N$ and $b-p\neq -N$ when $%
p>1,$ it follows that $b-p>-N$ if $p>1.$

The general procedure is the same as in Case (i-{\scriptsize 1}), with the
following difference: To prove (\ref{22}) with $v:=(1-\zeta )u$ ($\geq 0$), $%
k$ and $\xi \in \Bbb{R}$ must be found so that $k^{\prime }\xi >N,k>1$ and $%
\frac{k(c+\xi )-a}{kr-q}=\frac{b-p+N}{p}.$ With the same change of variable $%
k:=\frac{s+q}{r}$ as before, $k>1$ amounts to $s>r-q,$ so that (\ref{22})
holds for some $\xi $ if and only if $c<\max \left\{ c^{*},c^{1}\right\} $
(the supremum of the right-hand-side of (\ref{23}) when $s>r-q$).

Likewise, as in Case (i-{\scriptsize 1}), (\ref{22}) holds with $v=\zeta u$
if $c>\min \left\{ c^{*},c^{1}\right\} .$ This proves (ii) when $b-p>-N,$ or 
$p=1$ and $b-1\geq -N,$ and when $c$ is in the \emph{open} interval with
endpoints $c^{*}$ and $c^{1}.$ Thus, it only remains to discuss the case $%
c=c^{*}.$

This can be done by proving the inequality (\ref{22}) for $v=u$ radially
symmetric, with $k=1$ and $\xi =0$ (no need to split $u$). Specifically,
since $r>q$ (unlike in part (i)), write $||u||_{c^{*},r}^{r}=\int_{\Bbb{R}%
^{N}}|x|^{c^{*}}|u|^{r}=\int_{\Bbb{R}^{N}}|x|^{a+(r-q)\left( \frac{b-p+N}{p}
\right) }|u|^{r}\leq \left\| |x|^{\frac{b-p+N}{p}}u\right\| _{\infty
}^{r-q}||u||_{a,q}^{q}$ and notice $\left\| |x|^{\frac{b-p+N}{p}}u\right\|
_{\infty }\leq $ $C||\partial _{\rho }u||_{b,p}$ by using, as before,
Theorem \ref{th9} with $\gamma :=\frac{b}{p}-\frac{N}{p^{\prime }}.$ This
requires $b-p>-N$ if $p>1,$ but $b-1=-N$ is allowed if $p=1.$
\end{proof}

Part (ii) of Lemma \ref{lm12} is not optimal, but before improving it (in
Lemma \ref{lm15} below) we prove a similar result when $a$ and $b-p$ are on
opposite sides of $-N.$

\begin{lemma}
\label{lm13} Let $a,b,c\in \Bbb{R}$ and $1\leq p<\infty ,0<q,r<\infty $ be
given. If $a$ and $b-p$ are strictly on opposite sides of $-N,$ the subspace
of $\widetilde{W}_{\{a,b\}}^{1,(q,p)}$ of radially symmetric functions is
continuously embedded into $L^{r}(\Bbb{R}^{N};|x|^{c}dx)$ in the following
two cases: \newline
(i) $r\leq q$ and $c$ is in the open interval with endpoints $c^{0}$ and $%
-N. $\newline
(ii) $q<r,1-\frac{q}{r}<\theta _{-N}$ and\footnote{%
Since $-N$ is between $c_{0}$ and $c_{1}$ when $a$ and $b-p$ are on opposite
sides of $-N,$ it follows that $\theta _{-N}\in (0,1).$} $c$ is in the
semi-open interval with endpoints $c^{*}:=\left( 1-\frac{q}{r}\right) c^{1}+%
\frac{q}{r}c^{0}$ (included) and $-N$ (not included).
\end{lemma}

\begin{proof}
Since $a$ and $b-p$ are strictly on opposite sides of $-N,$ we may assume
that $b-p<-N<a$ by the usual Kelvin transform argument.

(i) By (\ref{1}), $c^{1}<-N<c^{0}.$ Let $c\in \left( -N,c^{0}\right) $ be
given. As in the proof of Lemma \ref{lm12}, it suffices to show that $%
||(1-\zeta )u||_{c,r}\leq C||(1-\zeta )u||_{\{a,b\},(q,p)}$ and that $%
||\zeta u||_{c,r}\leq C||\zeta u||_{\{a,b\},(q,p)}$ when $u$ is radially
symmetric.

Since $\limfunc{Supp}(1-\zeta )\subset \Bbb{R}^{N}\backslash B(0,\frac{1}{2}%
),$ it follows that $(1-\zeta )u\in \widetilde{W}_{loc,+}^{1,1}.$ As a
result, the argument of the proof of Case (i-{\scriptsize 1}) of Lemma \ref
{lm12}, based on Theorem \ref{th9}, can be repeated verbatim with now $%
\gamma :=\frac{b}{p}-\frac{N}{p^{\prime }}<1-N.$ This shows that $||(1-\zeta
)u||_{c,r}\leq C||(1-\zeta )u||_{\{a,b\},(q,p)}$ since $c<\max \left\{
c^{0},c^{1}\right\} =c^{0}.$

The inequality $||\zeta u||_{c,r}\leq C||\zeta u||_{\{a,b\},(q,p)}$ cannot
be obtained as in Case (i-{\scriptsize 1}) of Lemma \ref{lm12} because $%
b-p<-N$ but $\zeta u\notin \widetilde{W}_{loc,+}^{1,1},$ so that Theorem \ref
{th9} is not applicable. However, it can be proved with the trick used in
Case (i-{\scriptsize 2}) of that lemma: Since $-N<c<c^{0},$ part (i) of
Lemma \ref{lm12} can be used with $b$ replaced by $p-N>b$ because $a\neq -N$
and $c^{1}$ becomes $-N$ when $b$ is replaced by $p-N$ while $c^{0}$ is
unchanged. Thus, $||\zeta u||_{c,r}\leq C||\zeta u||_{\{a,p-N\},(q,p)}$
while $||\zeta u||_{\{a,p-N\},(q,p)}\leq ||\zeta u||_{\{a,b\},(q,p)}$ since $%
p-N>b$ and $\limfunc{Supp}\zeta \subset \overline{B}(0,1).$

(ii) Observe that $c^{1}<c^{*}<c^{0}$ because $q<r$ and $c^{1}<c^{0}$
(recall $b-p<-N<a$), while $1-\frac{q}{r}<\theta _{-N}$ ensures that $%
-N<c^{*}.$

Let then $c\in (-N,c^{*})$ be given. By using once again the fact that $%
(1-\zeta )u\in \widetilde{W}_{loc,+}^{1,1}$ since $\limfunc{Supp}(1-\zeta
)\subset \Bbb{R}^{N}\backslash B(0,\frac{1}{2})$ and Theorem \ref{th9} with $%
\gamma :=\frac{b}{p}-\frac{N}{p^{\prime }}<1-N,$ the argument of the proof
of part (ii) of Lemma \ref{lm12} (with obvious modifications) yields $%
||(1-\zeta )u||_{c,r}\leq C||(1-\zeta )u||_{\{a,b\},(q,p)}$ because $%
c<c^{*}=\max \left\{ c^{*},c^{1}\right\} .$

If $c=c^{*},$ the same argument works with ``$k=1,\xi =0$'': Let $%
v:=(1-\zeta )u\in \widetilde{W}_{loc,+}^{1,1}$ and write $%
||v||_{c^{*},r}^{r}=\int_{\Bbb{R}^{N}}|x|^{c^{*}}|v|^{r}=\int_{\Bbb{R}
^{N}}|x|^{a+(r-q)\left( \frac{b-p+N}{p}\right) }|v|^{r}\leq \left\| |x|^{%
\frac{b-p+N}{p}}v\right\| _{\infty }^{r-q}||v||_{a,q}^{q}.$ Then, use
Theorem \ref{th9} with $\gamma :=\frac{b}{p}-\frac{N}{p^{\prime }}<1-N$ to
get $\left\| |x|^{\frac{b-p+N}{p}}v\right\| _{\infty }\leq C||\partial
_{\rho }v||_{b,p}.$

The proof of $||\zeta u||_{c,r}\leq C||\zeta u||_{\{a,b\},(q,p)}$ when $c\in
(-N,c^{*}]$ proceeds as in (i) above, with minor modifications. If $\hat{b}%
>b,$ then $\hat{c}^{1}:=\frac{r(\hat{b}-p+N)}{p}-N>c^{1}$ and so $\hat{c}%
^{*}:=\left( 1-\frac{q}{r}\right) \hat{c}^{1}+\frac{q}{r}c^{0}>c^{*}.$ Note
also that $\hat{c}^{1}$ is arbitrarily close to $-N$ if $\hat{b}$ is close
enough to $p-N.$ As a result, $c$ is in the open interval with endpoints $%
\hat{c}^{*}$ and $\hat{c}^{1}$ (even when $c=c^{*}$) provided that $\hat{b}%
>p-N$ is close to $p-N,$ while $a$ and $\hat{b}-p$ are both on the right of $%
-N.$ Thus, part (ii) of Lemma \ref{lm12} is applicable with $b$ replaced by $%
\hat{b}$ (unlike in (i), $\hat{b}=p-N$ cannot be chosen if $p>1$ due to the
requirement $\hat{b}-p\neq -N$ to use part (ii) of Lemma \ref{lm12}).
\end{proof}

We shall now prove optimal variants of Lemmas \ref{lm12} and \ref{lm13}. To
do this, we need a complement of part (i) of Lemma \ref{lm7} in the radially
symmetric case.

\begin{lemma}
\label{lm14}Let $a,b\in \Bbb{R}$ and $1\leq p<\infty ,0<q<\infty $ be given.
If $1\leq \xi \leq \frac{q}{p^{\prime }}+1$ and $u\in \widetilde{W}%
_{\{a,b\}}^{1,(q,p)}$ is radially symmetric, then $|u|^{\xi }\in \widetilde{W%
}_{\{a,b_{\xi }\}}^{1,(q_{\xi },p_{\xi })},$ where 
\begin{equation}
p_{\xi }:=\frac{pq}{p(\xi -1)+q}\geq 1,\text{ }q_{\xi }:=\frac{q}{\xi }>0%
\text{ and }b_{\xi }:=\left( \frac{a(\xi -1)}{q}+\frac{b}{p}\right) p_{\xi }.
\label{24}
\end{equation}
Furthermore, $|u|^{\xi }$ (is radially symmetric and) 
\begin{equation}
||\,|u|^{\xi }||_{a,q_{\xi }}=||u||_{a,q}^{\xi },\qquad ||\partial _{\rho
}(|u|^{\xi })||_{b_{\xi },p_{\xi }}\leq \xi ||u||_{a,q}^{\xi -1}||\partial
_{\rho }u||_{b,p}.  \label{25}
\end{equation}
\end{lemma}

\begin{proof}
If $\xi =1,$ then $q_{\xi }=q,p_{\xi }=p$ and $b_{\xi }=b,$ the case covered
by Lemma \ref{lm7}, which also shows that it is not restrictive to assume $%
u\geq 0.$ From now on, $\xi >1.$ The assumption $\xi \leq \frac{q}{p^{\prime
}}+1$ ensures that $p_{\xi }\geq 1$ in (\ref{24}).

That $u^{\xi }$ is radially symmetric, $u^{\xi }\in L^{q_{\xi }}(\Bbb{R}
^{N};|x|^{a}dx)$ and $||\,|u|^{\xi }||_{a,q_{\xi }}=||u||_{a,q}^{\xi }$ is
obvious. It remains to prove that $u^{\xi }\in L_{loc}^{1}(\Bbb{R}_{*}^{N}),$
that $\partial _{\rho }(u^{\xi })\in L^{p_{\xi }}(\Bbb{R}^{N};|x|^{b_{\xi
}}dx)$ and that the second inequality holds in (\ref{25}).

By part (iv) of Lemma \ref{lm6}, $u\in \widetilde{W}_{loc,\pm }^{1,1}$ $%
\subset \widetilde{W}_{loc}^{1,1}$ (depending upon whether $a\geq -N$ or $%
a\leq -N$). Thus, from Lemma \ref{lm4}, $u(x)=f_{u}(|x|)$ with $f_{u}\in
W_{loc}^{1,1}(0,\infty ),f_{u}\geq 0$ and $\partial _{\rho
}u(x)=f_{u}^{\prime }(|x|).$ Since $\xi >1,$ it is clear that $f_{u}^{\xi
}\in W_{loc}^{1,1}(0,\infty )$ and that $\left( f_{u}^{\xi }\right) ^{\prime
}=\xi f_{u}^{\xi -1}f_{u}^{\prime }.$ Hence, once again by Lemma \ref{lm4}, $%
u^{\xi }(x)=f_{u}^{\xi }(|x|)$ is in $\widetilde{W}_{loc}^{1,1}\subset
L_{loc}^{1}(\Bbb{R}_{*}^{N})$ and $\partial _{\rho }(u^{\xi })(x)=\xi
f_{u}^{\xi -1}(|x|)f_{u}^{\prime }(|x|),$ i.e., $\partial _{\rho }(u^{\xi
})=\xi u^{\xi -1}\partial _{\rho }u.$

In general, if $\mu ,\nu >0,$ the multiplication maps $L^{\mu }\times L^{\nu
}$ into $L^{\frac{\mu \nu }{\mu +\nu }}$ and $||vw||_{\frac{\mu \nu }{\mu
+\nu }}\leq ||v||_{\mu }||w||_{\nu }.$ This does not require $\mu \geq 1$ or 
$\nu \geq 1$ (just use $|v|^{\frac{\mu \nu }{\mu +\nu }}\in L^{1+\frac{\mu }{%
\nu }}$ and $|w|^{\frac{\mu \nu }{\mu +\nu }}\in L^{1+\frac{\nu }{\mu }}$
and H\"{o}lder's inequality). Now, $|x|^{\frac{a(\xi -1)}{q}}|u|^{\xi -1}\in
L^{\frac{q}{\xi -1}}(\Bbb{R}^{N})$ since $|x|^{a}|u|^{q}\in L^{1}(\Bbb{R}%
^{N})$ and $\xi >1,$ and $|x|^{\frac{b}{p}}\partial _{\rho }u\in L^{p}(\Bbb{R%
}^{N}).$ Therefore, $|x|^{\frac{b_{\xi }}{p_{\xi }}}u^{\xi -1}\partial
_{\rho }u\in L^{p_{\xi }}(\Bbb{R}^{N})$ with $p_{\xi }$ and $b_{\xi }$ given
by (\ref{24}) and 
\begin{equation*}
\left\| \,|x|^{\frac{b_{\xi }}{p_{\xi }}}u^{\xi -1}\partial _{\rho
}u\right\| _{p_{\xi }}\leq \left\| \,|x|^{\frac{a(\xi -1)}{q}}|u|^{\xi
-1}\right\| _{\frac{q}{\xi -1}}||\,|x|^{\frac{b}{p}}\partial _{\rho
}u||_{p}=||u||_{a,q}^{\xi -1}||\partial _{\rho }u||_{b,p}.
\end{equation*}
From the above, this implies $\partial _{\rho }(u^{\xi })\in L^{p_{\xi }}(%
\Bbb{R}^{N};|x|^{b_{\xi }}dx)$ with $||\partial _{\rho }(u^{\xi })||_{b_{\xi
},p_{\xi }}\leq \xi ||u||_{a,q}^{\xi -1}||\partial _{\rho }u||_{b,p}.$
\end{proof}

\begin{remark}
If $1\leq \xi \leq \min \left\{ q,\frac{q}{p^{\prime }}+1\right\} ,$ Lemma 
\ref{lm14} is true without the radial symmetry assumption. Indeed, if $u\in 
\widetilde{W}_{\{a,b\}}^{1,(q,p)},$ then $|u|^{\xi }\in L^{q_{\xi }}(\Bbb{R}%
^{N};|x|^{a}dx)\subset L_{loc}^{1}(\Bbb{R}_{*}^{N})$ and $u\in \widetilde{W}%
_{\{a,b\}}^{1,(q,p)}\subset \widetilde{W}^{1,1}$ implies that $|u|^{\xi }$
is locally absolutely continuous on almost every ray through the origin (see
Section \ref{spaces}) with $\partial _{\rho }(|u|^{\xi })=\xi |u|^{\xi -1}(%
\limfunc{sgn}u)\partial _{\rho }u\in L^{p_{\xi }}(\Bbb{R}^{N};|x|^{b_{\xi
}}dx)\subset L_{loc}^{1}(\Bbb{R}_{*}^{N}).$ This will be used elsewhere.
\end{remark}

\begin{lemma}
\label{lm15}Let $a,b,c\in \Bbb{R}$ and $1\leq p<\infty ,0<q,r<\infty $ be
given and let $\breve{\theta}:=\left( 1-\frac{q}{r}\right) \left( \frac{q}{%
p^{\prime }}+1\right) ^{-1}<1$ ($\leq 0$ if $r\leq q$). The subspace of $%
\widetilde{W}_{\{a,b\}}^{1,(q,p)}$ of radially symmetric functions is
continuously embedded into $L^{r}(\Bbb{R}^{N};|x|^{c}dx)$ in the following
two cases.\newline
(i) $a$ and $b-p$ are on the same side of $-N$ (including $-N$), $\frac{a+N}{%
q}\neq \frac{b-p+N}{p},$ $c$ is in the open interval\thinspace with
endpoints $c^{0}$\thinspace and $c^{1}$\thinspace and$\,\theta _{c}\geq 
\breve{\theta}$ (vacuous if $r\leq q$).\newline
(ii) $a$ and $b-p$ are strictly on opposite sides of $-N,c$ is in the open
interval with endpoints $c^{0}$ and $-N$ and $\theta _{c}\geq \breve{\theta}$
(empty set if $\breve{\theta}\geq \theta _{-N}$).
\end{lemma}

\begin{proof}
If $r\leq q$ (so that $\breve{\theta}\leq 0$) or if $r>q$ and $p=1$ (so that 
$\breve{\theta}=$ $1-\frac{q}{r}$), (i) follows from Lemma \ref{lm12} (where 
$b-p\neq -N$ is not required in part (ii) when $p=1$) and (ii) follows from
Lemma \ref{lm13}. From now on, $r>q$ (so that $\breve{\theta}\in (0,1)$) and 
$p>1.$ For convenience, we set $\breve{\xi}:=\frac{q}{p^{\prime }}+1>1.$ In
particular, the interval $(1,\breve{\xi}]$ is not empty, which is implicitly
used below.

(i) Let $u\in \widetilde{W}_{\{a,b\}}^{1,(q,p)}$ be radially symmetric. If $%
1\leq \xi \leq \breve{\xi},$ then, by Lemma \ref{lm14}, $|u|^{\xi }\in 
\widetilde{W}_{\{a,b_{\xi }\}}^{1,(q_{\xi },p_{\xi })}$ with $q_{\xi
}>0,p_{\xi }\geq 1$ and $b_{\xi }$ given by (\ref{24}). A routine
verification shows that $a$ and $b_{\xi }-p_{\xi }$ are on the same side of $%
-N$ (since the same thing is true of $a$ and $b-p$) and that $\frac{a+N}{%
q_{\xi }}\neq \frac{b_{\xi }-p_{\xi }+N}{p_{\xi }}$ (since $\frac{a+N}{q}%
\neq \frac{b-p+N}{p}$).

Furthermore, $b_{\xi }-p_{\xi }\neq -N$ if $\xi >1$ (which need not be true
if $\xi =1$ since $b-p\neq -N$ is \emph{not} assumed). Indeed, $b_{\xi
}-p_{\xi }=-N$ amounts to $\frac{a+N}{q}(\xi -1)+\frac{b-p+N}{p}=0.$ Since $%
a $ and $b-p$ are on the same side of $-N,$ this can only happen if $%
a+N=b-p+N=0$ when $\xi >1,$ which contradicts the assumption $\frac{a+N}{q}
\neq \frac{b-p+N}{p}.$

Accordingly, from part (ii) of Lemma \ref{lm12} with $b,p,q$ and $r$
replaced by $b_{\xi },p_{\xi },q_{\xi }$ and $s,$ respectively, $%
W_{\{a,b_{\xi }\}}^{1,\left( q_{\xi },p_{\xi }\right) }(\Bbb{R}
_{*}^{N})\hookrightarrow L^{s}(\Bbb{R}^{N};|x|^{c}dx)$ whenever $1<\xi \leq 
\breve{\xi},0<q_{\xi }<s$ and $c$ is in the semi-open interval with
endpoints $a+\frac{(s-q_{\xi })(b_{\xi }-p_{\xi }+N)}{p_{\xi }}$ (included;
this corresponds to $c^{*}$ with the parameters $b_{\xi },p_{\xi },q_{\xi
},s $) and $\frac{s(b_{\xi }-p_{\xi }+N)}{p_{\xi }}-N$ (not included; this
corresponds to $c^{1}$ with the parameters $b_{\xi },p_{\xi },q_{\xi },s$).
Since $r>q,$ and $q_{\xi }=\frac{q}{\xi },$ the condition $0<q_{\xi }<s$
holds when $s=\frac{r}{\xi }.$ If so, the embedding inequality $||\,|u|^{\xi
}||_{c,\frac{r}{\xi }}\leq C_{\xi }(||\,|u|^{\xi }||_{a,q_{\xi }}+||\partial
_{\rho }(|u|^{\xi })||_{b_{\xi },p_{\xi }})$ reads (use (\ref{25})) 
\begin{equation*}
||u||_{c,r}^{\xi }\leq C_{\xi }(||u||_{a,q}^{\xi }+||u||_{a,q}^{\xi
-1}||\partial _{\rho }u||_{b,p})\leq C_{\xi }(||u||_{a,q}+||\partial _{\rho
}u||_{b,p})^{\xi },
\end{equation*}
so that $||u||_{c,r}\leq C_{\xi }^{\xi ^{-1}}||u||_{\{a,b\},(q,p)}.$ Above, $%
c$ is in the semi-open interval $J_{\xi }$ with (distinct) endpoints $%
e_{1}(\xi ):=a+\frac{(r-q)(b_{\xi }-p_{\xi }+N)}{\xi p_{\xi }}$ (included)
and $e_{2}(\xi ):=\frac{r(b_{\xi }-p_{\xi }+N)}{\xi p_{\xi }}-N$ (not
included) and $1<\xi \leq \breve{\xi}.$ Thus, when $c\in J:=\cup _{\xi \in
(1,\breve{\xi}]}J_{\xi },$ 
\begin{equation}
||u||_{c,r}\leq C(||u||_{a,q}+||\partial _{\rho }u||_{b,p}),  \label{26}
\end{equation}
for some constant $C>0$ independent of the radially symmetric $u\in
W_{\{a,b\}}^{1,(q,p)}(\Bbb{R}_{*}^{N})$ (specifically, $C=C_{\xi }^{\xi
^{-1}}$ for any $\xi $ such that $c\in J_{\xi }$).

Since the distinct endpoints of $J_{\xi }$ depend continuously upon $\xi ,$
the lower (upper) endpoint $e_{-}(\xi )$ ($e_{+}(\xi )$) is either $%
e_{1}(\xi )$ for every $\xi $ or $e_{2}(\xi )$ for every $\xi .$ Hence, $%
e_{\pm }$ are continuous and never equal functions of $\xi .$ With that
remark, it is an easy exercise to show that $J$ contains the open interval
with endpoints $\inf e_{-}$ and $\sup e_{+}.$

If $\frac{a+N}{q}>\frac{b-p+N}{p},$ then $e_{1}>e_{2}$ and both $e_{1}$ and $%
e_{2}$ are increasing functions of $\xi ,$ so that $J$ contains $%
(e_{2}(1),e_{1}(\breve{\xi})).$ In addition, since it contains $e_{1}(\breve{
\xi})\in J_{\breve{\xi}},$ it contains -and, in fact, coincides with- $%
(e_{2}(1),e_{1}(\breve{\xi})].$

If $\frac{a+N}{q}<\frac{b-p+N}{p},$ then $e_{2}>e_{1}$ and both $e_{1}$ and $%
e_{2}$ are decreasing functions of $\xi ,$ so that $J$ contains the open
interval $(e_{1}(\breve{\xi}),e_{2}(1)).$ Once again, it also contains $%
e_{1}(\breve{\xi}).$ Therefore, in all cases, $J$ is the semi-open interval
with endpoints $e_{1}(\breve{\xi})=\breve{\theta}c^{1}+(1-\breve{\theta}%
)c^{0}$ (included) and $e_{2}(1)=c^{1}$ (not included). For every $c$ in
that interval, (\ref{26}) holds for some constant $C$ independent of the
radially symmetric $u\in \widetilde{W}_{\{a,b\}}^{1,(q,p)}.$ Clearly, $J$ is
equally characterized as the set of those $c$ in the open interval with
endpoints $c^{0}$ and $c^{1}$ such that $\theta _{c}\geq \breve{\theta}.$

(ii) First, since $a$ and $b-p$ are on opposite sides of $-N,$ it is obvious
that $\frac{a+N}{q}\neq \frac{b-p+N}{p}.$ The proof will proceed as in part
(i), but extra technicalities arise from the fact that the points $a$ and $%
b_{\xi }-p_{\xi }$ (see (\ref{24})) need not remain on opposite sides of $-N$
for all $\xi \in [1,\breve{\xi}].$

Nonetheless, since $b_{\xi }-p_{\xi }$ is a strictly monotone function of $%
\xi $ equal to $b-p$ when $\xi =1$ and since $a$ and $b-p$ are strictly on
opposite sides of $-N,$ there are only two options: Either $a$ and $b_{\xi
}-p_{\xi }$ are strictly on opposite sides of $-N$ when $\xi =\breve{\xi}$
-which amounts to $a$ and $\frac{b}{p}+\frac{a}{p^{\prime }}-1$ being
strictly on opposite sides of $-N$- and then the same thing is true for
every $\xi \in [1,\breve{\xi}],$ or $b_{\xi _{0}}-p_{\xi _{0}}=-N$ for some
unique $\xi _{0}\in (1,\breve{\xi}],$ and then $a$ and $b_{\xi }-p_{\xi }$
are on the same side of $-N$ for every $\xi \in [\xi _{0},\breve{\xi}].$

\textit{Case (ii-{\scriptsize 1}):} $a$ and $\frac{b}{p}+\frac{a}{p^{\prime }%
}-1$ are strictly on opposite sides of $-N.$

Replace $q,r,b,p$ by $q_{\breve{\xi}}=\frac{q}{\breve{\xi}},r_{\breve{\xi}}=%
\frac{r}{\breve{\xi}},b_{\breve{\xi}}=\frac{a}{p^{\prime }}+\frac{b}{p},p_{%
\breve{\xi}}=1,$ respectively, in part (ii) of Lemma \ref{lm13} and use that
result with $u$ replaced by $|u|^{\breve{\xi}}.$ This is justified by Lemma 
\ref{lm14}. However, it is crucial to notice that, due to the change of
parameters, the condition ``$1-\frac{q}{r}<\theta _{-N}$'' in Lemma \ref
{lm13} does not involve $\theta _{-N}$ but, instead, the number $\breve{%
\theta}_{-N}$ given by the same formula (\ref{2}) when $c^{0}$ and $c^{1}$
are replaced by $\breve{c}^{0}$ and $\breve{c}^{1}$ defined by (\ref{1})
with the new parameters $q_{\breve{\xi}},r_{\breve{\xi}},b_{\breve{\xi}},p_{%
\breve{\xi}}.$ Thus, $\breve{c}^{0}=c^{0}$ but $\breve{c}^{1}=$ $r\breve{\xi}%
^{-1}\left( \frac{a+N}{p^{\prime }}+\frac{b-p+N}{p}\right) -N,$ so that $%
\breve{c}^{1}-\breve{c}^{0}=r\breve{\xi}^{-1}\left( \frac{b-p+N}{p}-\frac{a+N%
}{q}\right) .$ With this remark, it is readily checked that $\breve{\theta}
_{-N}=\breve{\xi}\theta _{-N},$ so that the condition $1-\frac{q_{\breve{\xi}
}}{r_{\breve{\xi}}}<\breve{\theta}_{-N}$ is $\breve{\theta}:=\left( 1-\frac{q%
}{r}\right) \left( \frac{q}{p^{\prime }}+1\right) ^{-1}<\theta _{-N}.$

In summary, the continuity of the embedding is ensured if $\breve{\theta}
<\theta _{-N}$ and $c$ is in the semi-open interval with endpoints $\breve{c}
:=\left( 1-\frac{q}{r}\right) \breve{c}^{1}+\frac{q}{r}\breve{c}^{0}=\breve{
\theta}c^{1}+(1-\breve{\theta})c^{0}$ (included) and $-N$ (not included),
which -since $r>q$- coincides with the set of those $c$ in the open interval
with endpoints $c^{0}$ and $-N$ such that $\theta _{c}\geq \breve{\theta}.$

\textit{Case (ii-{\scriptsize 2}):} $b_{\xi _{0}}-p_{\xi _{0}}=-N$ for some $%
\xi _{0}\in (1,\breve{\xi}).$

Since $a$ and $b_{\xi }-p_{\xi }$ are on the same side of $-N$ for $\xi \in
[\xi _{0},\breve{\xi}]$ and since $b_{\xi }-p_{\xi }\neq -N$ if $\xi \in
(\xi _{0},\breve{\xi}],$ part (ii) of Lemma \ref{lm12} with $u$ replaced by $%
|u|^{\xi }$ and $q,r,b,p$ replaced by $\frac{q}{\xi },\frac{r}{\xi },b_{\xi
},p_{\xi },$ respectively, yields that the subspace of $W_{\{a,b%
\}}^{1,(q,p)}(\Bbb{R}_{*}^{N})$ of radially symmetric functions is
continuously embedded into $L^{r}(\Bbb{R}^{N};|x|^{c}dx)$ for every $c\in
J:=\cup _{\xi \in (\xi _{0},\breve{\xi}]}J_{\xi },$ where $J_{\xi }$ is the
semi-open interval with endpoints $e_{1}(\xi ):=a+\frac{(r-q)(b_{\xi
}-p_{\xi }+N)}{\xi p_{\xi }}$ (included) and $e_{2}(\xi ):=\frac{r(b_{\xi
}-p_{\xi }+N)}{\xi p_{\xi }}-N$ (not included). Both endpoints are distinct
(because $\frac{a+N}{q}\neq \frac{b-p+N}{p}$) and on the same side of $-N$
when $\xi >\xi _{0}.$ By arguing as in the proof of (i) above, $J$ is found
to be the semi-open interval with endpoints $e_{1}(\breve{\xi})=\breve{c}=%
\breve{\theta}c^{1}+(1-\breve{\theta})c^{0}$ (included) and $e_{2}(\xi
_{0})=-N$ (not included), exactly as in (ii-{\scriptsize 1}). Therefore, the
final argument is also the same.

\textit{Case (ii-{\scriptsize 3}):} $b_{\xi _{0}}-p_{\xi _{0}}=-N$ with $\xi
_{0}=\breve{\xi}.$

Since $a$ and $b_{\breve{\xi}}-p_{\breve{\xi}}$ are on the same side of $-N$ 
\emph{\ and} since $p_{\breve{\xi}}=1,$ it suffices to use part (ii) of
Lemma \ref{lm12} with $u$ replaced by $|u|^{\breve{\xi}}$ and $q,r,b,p$
replaced by $\frac{q}{\breve{\xi}},\frac{r}{\breve{\xi}},b_{\breve{\xi}}=%
\frac{b}{p}+\frac{a}{p^{\prime }},p_{\breve{\xi}}=1,$ respectively.
\end{proof}

It is informative that even if $q,r\geq 1$ in Lemma \ref{lm15}, the proof
involves part (ii) of Lemmas \ref{lm12} and \ref{lm13} when $q,r>0$ ($%
q,r\geq 1$ is not enough). The next theorem gives necessary and sufficient
conditions for the continuous embedding of the subspace of radially
symmetric functions.

\begin{theorem}
\label{th16}Let $a,b,c\in \Bbb{R},1\leq p<\infty $ and $0<q,r<\infty $ and
set $\breve{\theta}:=\left( 1-\frac{q}{r}\right) \left( \frac{q}{p^{\prime }}%
+1\right) ^{-1}.$ The subspace of $\widetilde{W}_{\{a,b\}}^{1,(q,p)}$ of
radially symmetric functions is continuously embedded into $L^{r}(\Bbb{R}%
^{N};|x|^{c}dx)$ (and hence into $\widetilde{W}_{\{c,b\}}^{1,(r,p)}$) if and
only if one of the following conditions holds:\newline
(i) $a$ and $b-p$ are on the same side of $-N$ (including $-N$), $\frac{a+N}{%
q}\neq \frac{b-p+N}{p},$ $c$ is in the open interval with endpoints $c^{0}$
and $c^{1}$ and $\theta _{c}\geq \breve{\theta}$ (vacuous if $q\geq r$).%
\newline
(ii) $a$ and $b-p$ are strictly on opposite sides of $-N,c$ is in the open
interval with endpoints $c^{0}$ and $-N$ and $\theta _{c}\geq \breve{\theta}$
(empty set if $\breve{\theta}\geq \theta _{-N}$).\newline
(iii) $r\geq p,a\leq -N$ and $b-p<-N$ or $a\geq -N$ and $b-p>-N,c=c^{1}.$
Furthermore, there is a constant $C>0$ such that 
\begin{equation}
||u||_{c,r}\leq C||\partial _{\rho }u||_{b,p},  \label{27}
\end{equation}
for every radially symmetric function $u\in \widetilde{W}_{\{a,b%
\}}^{1,(q,p)}.$ \newline
(iv) $r=q$ and $c=c^{0}$ ($=a$), or $p\neq q,\min \{p,q\}\leq r\leq \max
\{p,q\},\frac{a+N}{q}=\frac{b-p+N}{p}\neq 0$ and $c=c^{0}$ ($=c^{1}$).
Furthermore, there is a constant $C>0$ such that 
\begin{equation}
||u||_{c,r}\leq C||\partial _{\rho }u||_{b,p}^{\theta }||u||_{a,q}^{1-\theta
},  \label{28}
\end{equation}
for every radially symmetric function $u\in \widetilde{W}_{\{a,b%
\}}^{1,(q,p)},$ where $\theta =0$ if $r=q$ and $c=a$ and $\theta =\frac{%
p(r-q)}{r(p-q)}$ otherwise.\newline
(v) $a=-N,b=p-N,r>q$ and $c=c^{0}$ ($=c^{1}=-N$). Furthermore, there is a
constant $C>0$ such that 
\begin{equation}
||u||_{-N,r}\leq C||\partial _{\rho }u||_{p-N,p}^{\breve{\theta}%
}||u||_{-N,q}^{1-\breve{\theta}},  \label{29}
\end{equation}
for every radially symmetric function $u\in \widetilde{W}_{\{a,b%
\}}^{1,(q,p)}.$
\end{theorem}

\begin{proof}
The theorem is (as it should be) equivalent to Theorem \ref{th0} when $N=1$
(in particular, $p^{*}=\infty $ regardless of $p$ and $\frac{1}{p}-\frac{1}{N%
}-\frac{1}{q}=-\left( \frac{1}{p^{\prime }}+\frac{1}{q}\right) $) and $a,b,c$
replaced by $a+N-1,b+N-1$ and $c+N-1,$ respectively. Since the hypotheses of
Theorem \ref{th0} are necessary (Section \ref{necessary}), the necessity
follows.

The sufficiency of parts (i) and (ii) was already proved in Lemma \ref{lm15}%
. To complete the proof, we show that parts (iii), (iv) or (v) are also
sufficient.

(iii) By Kelvin transform (Remark \ref{rm2}), we may assume $a\leq -N$ and $%
b-p<-N.$ In particular, $u\in \widetilde{W}_{loc,+}^{1,1}$ by part (iv) of
Lemma \ref{lm6}. By part (i) of the same lemma, we may also assume $u\geq 0$
with no loss of generality. Then, $u(x)=f_{u}(|x|)$ with $f_{u}\in
W_{loc}^{1,1}(0,\infty ),f_{u}\geq 0$ and $\underline{\lim }_{t\rightarrow
0^{+}}f_{u}(t)=0,$ so that $f_{u}(t)\leq \int_{0}^{t}|f_{u}^{\prime }(\tau
)|d\tau $ by (\ref{20}) with $\gamma =1-N$ and $f=f_{u}.$

On the other hand, 
\begin{equation*}
\left( \int_{0}^{\infty }t^{\frac{r(b-p+N)}{p}-1}\left(
\int_{0}^{t}|f_{u}^{\prime }(\tau )|d\tau \right) ^{r}dt\right) ^{\frac{1}{r}
}\leq C\left( \int_{0}^{\infty }t^{b+N-1}|f_{u}^{\prime }(t)|^{p}dt\right) ^{%
\frac{1}{p}},
\end{equation*}
by a weighted Hardy inequality of Bradley for nonnegative measurable
functions on $(0,\infty )$ (\cite[Theorem 1]{Br78}, \cite[p. 40]{Ma85})
inspired by Muckenhoupt \cite{Mu72} when $r=p.$ This yields (\ref{27}) since 
$c=c^{1}=\frac{r(b-p+N)}{p}-N$ and $\partial _{\rho }u(x)=f_{u}^{\prime
}(|x|)$ (Lemma \ref{lm4}).

(iv) This is trivial if $r=q$ and $c=a.$ From now on, $p\neq q$ and $r$ is
between $p$ and $q$ (both included), so that $r=\mu p+(1-\mu )q$ where $\mu =%
\frac{r-q}{p-q}\in [0,1],$ whence $\mu (b-p)+(1-\mu )a=c^{0}=c$ (use $b-p=%
\frac{p(a+N)}{q}-N$). Thus, if $u$ is measurable, $\int_{\Bbb{R}
^{N}}|x|^{c}|u|^{r}=\int_{\Bbb{R}^{N}}(|x|^{b-p}|u|^{p})^{\mu
}(|x|^{a}|u|^{q})^{1-\mu }$ and, by H\"{o}lder's inequality, 
\begin{equation}
||u||_{c,r}^{r}\leq ||u||_{b-p,p}^{\mu p}||u||_{a,q}^{(1-\mu )q}.  \label{30}
\end{equation}
Since $\frac{a+N}{q}=\frac{b-p+N}{p}\neq 0,$ both $a$ and $b-p$ are on the
same side of $-N$ and neither equals $-N.$ Therefore, when $u\in \widetilde{W%
}_{\{a,b\}}^{1,(q,p)}$ is radially symmetric, $||u||_{b-p,p}\leq C||\partial
_{\rho }u||_{b,p}$ by (iii) with $r=p$ (hence $c=b-p$). By substitution into
(\ref{30}), $||u||_{c,r}\leq C||\partial _{\rho }u||_{b,p}^{\mu \frac{p}{r}
}||u||_{a,q}^{(1-\mu )\frac{q}{r}}=C||\partial _{\rho }u||_{b,p}^{\theta
}||u||_{a,q}^{1-\theta }$ with $\theta =\mu \frac{p}{r}=\frac{p(r-q)}{r(p-q)}
.$ This proves (\ref{28}) and hence the embedding property as well.

(v) If $r>q,N=1$ and $a=b=c=0,$ it follows from part (i) of the theorem if $%
p=1$ and from its part (ii) if $p>1,$ that the subspace of even functions in
the unweighted space $W^{1,(q,p)}(\Bbb{R}_{*})$ is continuously embedded
into $L^{r}(\Bbb{R}).$ In this one-dimensional setting, this readily implies
the same result without the evenness assumption, i.e., $W^{1,(q,p)}(\Bbb{R}
_{*})\hookrightarrow L^{r}(\Bbb{R}),$ and then 
\begin{equation}
||g||_{r}\leq C||g^{\prime }||_{p}^{\breve{\theta}}||g||_{q}^{1-\breve{\theta%
}},  \label{31}
\end{equation}
for $g\in W^{1,(q,p)}(\Bbb{R}_{*})$ by the usual rescaling argument. In
particular, (\ref{31}) holds with $g\in W^{1,(q,p)}(\Bbb{R})$ (if $g\in
C_{0}^{\infty }(\Bbb{R})$ and $q\geq 1,$ this also follows from \cite
{CaKoNi84}).

Now, as in (iii), if $u\in \widetilde{W}_{\{-N,p-N\}}^{1,(q,p)}$ is radially
symmetric, then $u(x)=f_{u}(|x|)$ with $f_{u}\in W_{loc}^{1,1}(0,\infty )$
and $\partial _{\rho }u(x)=f_{u}^{\prime }(|x|),$ so that $%
||u||_{-N,q}^{q}=N\omega _{N}\int_{0}^{\infty }t^{-1}|f_{u}(t)|^{q}dt<\infty 
$ and $||\partial _{\rho }u||_{p-N,p}^{p}=N\omega _{N}\int_{0}^{\infty
}t^{p-1}|f_{u}^{\prime }(t)|^{p}dt<\infty .$

On the other hand, with $g(s):=f_{u}(e^{s}),$ it is readily checked that $%
g\in W^{1,(q,p)}(\Bbb{R})$ with $||g||_{q}^{q}=\int_{0}^{\infty
}t^{-1}|f_{u}(t)|^{q}dt$ and $||g^{\prime }||_{p}^{p}=\int_{0}^{\infty
}t^{p-1}|f_{u}^{\prime }(t)|^{p}dt.$ Therefore, (\ref{31}) may be rewritten
as $||u||_{-N,r}\leq C||\partial _{\rho }u||_{p-N,p}^{\breve{\theta}
}||u||_{-N,q}^{1-\breve{\theta}}.$ This completes the proof.
\end{proof}

\begin{remark}
Since $\widetilde{W}_{\{a,b\}}^{1,(q,p)}$ and $W_{\{a,b\}}^{1,(q,p)}(\Bbb{R}%
_{*}^{N})$ contain the same radially symmetric functions and the induced
norms are the same, Theorem \ref{th16} is also true when $\widetilde{W}%
_{\{a,b\}}^{1,(q,p)}$ is replaced by $W_{\{a,b\}}^{1,(q,p)}(\Bbb{R}%
_{*}^{N}). $
\end{remark}

\begin{remark}
\label{rm4}In part (ii) of Theorem \ref{th16}, the admissible interval is
empty if $\breve{\theta}\geq \theta _{-N},$ which can only happen if $r>q.$
However, a careful examination of the proofs reveals that the subspace of
(radially symmetric) functions with support in a ball $\overline{B}$
centered at $0$ is continuously embedded into $L^{r}(\Bbb{R}^{N};|x|^{c}dx)$
if $c>-N$ when $b-p<-N<a$ and if $c\geq \breve{c}$ (even if $\breve{c}=-N$)
when $a<-N<b-p.$ For functions with support in $\Bbb{R}^{N}\backslash B$ the
conditions ($c\leq \breve{c}$ if $b-p<-N<a$ and $c<-N$ if $a<-N<b-p$) follow
by Kelvin transform. Details are left to the reader.
\end{remark}

\section{Embedding theorem when $1\leq r\leq \min \{p,q\}$\label{embedding1}}

We now extend Theorem \ref{th16} to the non-symmetric case when $1\leq r\leq
\min \{p,q\}.$ To do this, we need the following refinement of part (ii) of
Lemma \ref{lm7}:

\begin{lemma}
\label{lm17}Let $a,b\in \Bbb{R}$ and $1\leq r\leq p,q<\infty $ be given. If $%
u\in \widetilde{W}_{\{a,b\}}^{1,(q,p)},$ then $v:=[(|u|^{r})_{S}]^{\frac{1}{r%
}}\in \widetilde{W}_{\{a,b\}}^{1,(q,p)}.$ Furthermore, $||v||_{a,q}\leq
||u||_{a,q}$ and $||\,\partial _{\rho }v||_{b,p}\leq ||\partial _{\rho
}u||_{b,p},$ so that $||v||_{\{a,b\},(q,p)}\leq ||u||_{\{a,b\},(q,p)}.$
\end{lemma}

\begin{proof}
By part (i) of Lemma \ref{lm7}, $|u|\in \widetilde{W}_{\{a,b\}}^{1,(q,p)},$
so that it is not restrictive to assume $u\geq 0.$ Since $%
v(x)=[f_{u^{r}}(|x|)]^{\frac{1}{r}}$ with $f_{u^{r}}(t):=(N\omega
_{N})^{-1}\int_{\Bbb{S}^{N-1}}u^{r}(t\sigma )d\sigma ,$ it follows from $%
r\leq q$ and H\"{o}lder's inequality that $(v(x))^{q}\leq (N\omega
_{N})^{-1}\int_{\Bbb{S}^{N-1}}u^{q}(|x|\sigma )d\sigma .$ Thus, $%
||v||_{a,q}\leq ||u||_{a,q}$ is clear.

We now show that $\partial _{\rho }v\in L^{p}(\Bbb{R}^{N};|x|^{b}dx)$ and
prove the desired estimate. Formally, if $h:=(f_{u^{r}})^{\frac{1}{r}},$
then $h^{\prime }=\frac{1}{r}(f_{u^{r}})^{-\frac{1}{r^{\prime }}
}f_{u^{r}}^{\prime }$ but, by the de la Vall\'{e}e Poussin criterion (\cite
{Va15}, \cite[Lemma 1.2]{MaMi72}, \cite[Corollary 8]{SeVa69}), this formula
holds and $h\in W_{loc}^{1,1}(0,\infty )$ if and only if $f_{u^{r}}\in
W_{loc}^{1,1}(0,\infty )$ and $(f_{u^{r}})^{-\frac{1}{r^{\prime }}%
}f_{u^{r}}^{\prime }\in L_{loc}^{1}(0,\infty ),$ with the understanding that 
$(f_{u^{r}})^{-\frac{1}{r^{\prime }}}f_{u^{r}}^{\prime }=0$ when $%
f_{u^{r}}^{\prime }=0,$ irrespective of whether $(f_{u^{r}})^{-\frac{1}{
r^{\prime }}}$ is defined. Since $f_{u^{r}}^{\prime }=0$ a.e. on $%
(f_{u^{r}})^{-1}(0),$ this amounts to defining $(f_{u^{r}})^{-\frac{1}{
r^{\prime }}}f_{u^{r}}^{\prime }=0$ on $(f_{u^{r}})^{-1}(0).$ That $%
(f_{u^{r}})^{-\frac{1}{r^{\prime }}}f_{u^{r}}^{\prime }\in
L_{loc}^{1}(0,\infty )$ is verified below.

First, $u\in \widetilde{W}_{loc}^{1,r}$ since $r\leq p,q.$ By Lemma \ref{lm5}%
, $u^{r}\in \widetilde{W}_{loc}^{1,1}$ (so that $f_{u^{r}}\in
W_{loc}^{1,1}(0,\infty )$) and $\partial _{\rho }(u^{r})=ru^{r-1}\partial
_{\rho }u.$ Upon replacing $u$ by $u^{r}$ in (\ref{16}) and by H\"{o}lder's
inequality, it follows that $|f_{u^{r}}^{\prime }|\leq r(f_{u^{r}})^{\frac{1%
}{r^{\prime }}}\left( (N\omega _{N})^{-1}\int_{\Bbb{S}^{N-1}}|\partial
_{\rho }u|^{r}d\sigma \right) ^{\frac{1}{r}}\leq r(f_{u^{r}})^{\frac{1}{%
r^{\prime }}}\left( (N\omega _{N})^{-1}\int_{\Bbb{S}^{N-1}}|\partial _{\rho
}u|^{p}d\sigma \right) ^{\frac{1}{p}}.$ Since $(f_{u^{r}})^{-\frac{1}{
r^{\prime }}}f_{u^{r}}^{\prime }=0$ on $(f_{u^{r}})^{-1}(0),$ this yields $%
(f_{u^{r}})^{-\frac{1}{r^{\prime }}}|f_{u^{r}}^{\prime }|\leq r\left(
(N\omega _{N})^{-1}\int_{\Bbb{S}^{N-1}}|\partial _{\rho }u|^{p}d\sigma
\right) ^{\frac{1}{p}}\in L_{loc}^{p}(0,\infty )\subset L_{loc}^{1}(0,\infty
).$

From the above, $h\in W_{loc}^{1,1}(0,\infty ),h^{\prime }=\frac{1}{r}
(f_{u^{r}})^{-\frac{1}{r^{\prime }}}f_{u^{r}}^{\prime }$ and, in addition, $%
|h^{\prime }(t)|\leq \left( (N\omega _{N})^{-1}\int_{\Bbb{S}^{N-1}}|\partial
_{\rho }u(t\sigma )|^{p}d\sigma \right) ^{\frac{1}{p}}.$ Since $\partial
_{\rho }v(x)=h^{\prime }(|x|)$ by Lemma \ref{lm4}, $|\partial _{\rho
}v(x)|^{p}\leq (N\omega _{N})^{-1}\int_{\Bbb{\ \ S}^{N-1}}|\partial _{\rho
}u(|x|\sigma )|^{p}d\sigma ,$ so that $||\,\partial _{\rho }v||_{b,p}\leq
||\partial _{\rho }u||_{b,p}.$
\end{proof}

\begin{theorem}
\label{th18} Let $a,b,c\in \Bbb{R}$ and $1\leq r\leq p,q<\infty $ be given.
Then, $\widetilde{W}_{\{a,b\}}^{1,(q,p)}\hookrightarrow L^{r}(\Bbb{R}%
^{N};|x|^{c}dx)$ (and hence $\widetilde{W}_{\{a,b\}}^{1,(q,p)}%
\hookrightarrow \widetilde{W}_{\{c,b\}}^{1,(r,p)}$) in the following cases: 
\newline
(i) $a$ and $b-p$ are on the same side of $-N$ (including $-N$), $\frac{a+N}{%
q}\neq \frac{b-p+N}{p}$ and $c$ is in the open interval with endpoints $%
c^{0} $ and $c^{1}.$\newline
(ii) $a$ and $b-p$ are strictly on opposite sides of $-N$ (hence $\frac{a+N}{%
q}\neq \frac{b-p+N}{p}$) and $c$ is in the open interval with endpoints $%
c^{0}$ and $-N.$\newline
(iii) $r=q$ $\leq p$ and $c=a.$\newline
(iv) $r=p$ $\leq q,$ $a\leq -N$ and $b-p<-N$ or $a\geq -N$ and $b-p>-N,$ and 
$c=b-p.$
\end{theorem}

\begin{proof}
(i)-(ii) Set $v:=[(|u|^{r})_{S}]^{\frac{1}{r}}.$ By Lemma \ref{lm17}, $v\in 
\widetilde{W}_{\{a,b\}}^{1,(q,p)}$ and $||v||_{\{a,b\},(q,p)}\leq
||u||_{\{a,b\},(q,p)}.$ Thus, since $v$ is radially symmetric, it follows
from parts (i) and (ii) of Theorem \ref{th16} (where $\theta _{c}\geq \breve{
\theta}$ holds since $\breve{\theta}\leq 0$) that $||v||_{c,r}\leq
C||u||_{\{a,b\},(q,p)},$ where $C>0$ is independent of $u.$ The conclusion
follows from the remark that $||v||_{c,r}=||u||_{c,r}.$

(iii) is trivial.

(iv) Argue as in (i)-(ii) above, now using part (iii) of Theorem \ref{th16}
with $r=p.$
\end{proof}

\smallskip When $\widetilde{W}_{\{a,b\}}^{1,(q,p)}$ is replaced by the
smaller space $W_{\{a,b\}}^{1,(q,p)}(\Bbb{R}_{*}^{N}),$ Theorem \ref{th18}
coincides with Theorem \ref{th0} when $1\leq r\leq p,q<\infty .$ Indeed, $%
r\leq \min \{p,q\}$ implies $r\leq \min \{p^{*},q\},$ so that the inequality 
$\theta _{c}\left( \frac{1}{p}-\frac{1}{N}-\frac{1}{q}\right) \leq \frac{1}{r%
}-\frac{1}{q}$ holds for every $c$ in the closed interval with endpoints $%
c^{0}$ and $c^{1}$ (Remark \ref{rm1}).

\section{The Caffarelli-Kohn-Nirenberg Lemma and application\label{CKNlemma}}

The reduction to the radially symmetric case in the previous section cannot
be used when $r>\min \{p,q\}.$ Consistent with the strategy outlined in the
Introduction, this section is devoted to the formulation and proof of an
embedding property for a direct complement of the subspace of radially
symmetric functions.

It will be necessary to confine attention to the space $W_{\{a,b%
\}}^{1,(q,p)}(\Bbb{R}_{*}^{N})$ (as opposed to $\widetilde{W}
_{\{a,b\}}^{1,(q,p)}$), because integrability conditions about all the first
order partial derivatives are implicitly required. While phrased differently
and under less general conditions, Lemma \ref{lm19} below is already
contained in \cite{CaKoNi84}.

\begin{lemma}[CKN lemma]
\label{lm19}Let $a,b,c\in \Bbb{R}$ and $1\leq p,q,r<\infty $ be given and
suppose that there are $\delta \leq \frac{b}{p}$ and $\theta \in [0,1]$ such
that:\newline
(i) $\frac{c}{r}=\theta \delta +(1-\theta )\frac{a}{q}.$\newline
(ii) $\frac{c+N}{r}=\theta \frac{b-p+N}{p}+(1-\theta )\frac{a+N}{q}.$\newline
(iii) $\frac{\theta r}{p}+\frac{(1-\theta )r}{q}\geq 1.$\newline
Then, 
\begin{equation}
W_{0}:=\left\{ u\in W_{\{a,b\}}^{1,(q,p)}(\Bbb{R}_{*}^{N}):u_{S}=0\right\}
\hookrightarrow L^{r}(\Bbb{R}^{N};|x|^{c}dx)  \label{33}
\end{equation}
and there is a constant $C>0$ such that 
\begin{equation}
||u||_{c,r}\leq C||\nabla u||_{b,p}^{\theta }||u||_{a,q}^{(1-\theta
)},\qquad \forall u\in W_{0}.  \label{34}
\end{equation}
\end{lemma}

\begin{proof}
Of course, it suffices to prove (\ref{34}). For $\tau >0,$ let $\Omega
_{\tau }$ denote the annulus $\{x\in \Bbb{R}^{N}:\tau <|x|<2\tau \}.$ Under
the conditions (i) and (ii) of the lemma\footnote{%
None of the other assumptions in \cite{CaKoNi84} is involved.}, it is shown
in \cite[pp. 262-263]{CaKoNi84} that the \emph{unweighted} inequality 
\begin{equation}
\int_{\Omega _{1}}|u|^{r}\leq C\left( \int_{\Omega _{1}}|\nabla
u|^{p}\right) ^{\frac{\theta r}{p}}\left( \int_{\Omega _{1}}|u|^{q}\right) ^{%
\frac{(1-\theta )r}{q}},  \label{35}
\end{equation}
holds for some constant $C$ and every $u\in C_{0}^{\infty }(\Bbb{R}^{N})$
such that $\int_{\Omega _{1}}u=0.$ The proof relies on the
Gagliardo-Nirenberg and Sobolev inequalities. (Since $a,b,c$ are not
involved in (\ref{35}), what matters is the relation $\frac{1}{r}=\theta
\left( \frac{1}{p}-\frac{1}{N}+\gamma \right) +(1-\theta )\frac{1}{q}$ with $%
\gamma \geq 0;$ that $\gamma =\frac{1}{N}\left( \frac{b}{p}-\delta \right) $
from (i) and (ii) combined, is not relevant at this stage.)

From the geometric properties of $\Omega _{1},$ the denseness of $%
C_{0}^{\infty }(\Bbb{R}^{N})$ in the unweighted space $W^{1,(q,p)}(\Omega
_{1}):=\{u\in L^{q}(\Omega _{1}):\nabla u\in L^{p}(\Omega _{1})\}$ is
routine (see \cite{BeIlNi79}, \cite{Ra79} for more general results) and it
is trivial that denseness remains true if, in both spaces, attention is
confined to functions with mean $0$ on $\Omega _{1}.$ Thus, (\ref{35})
continues to hold for $u\in W^{1,(q,p)}(\Omega _{1})$ such that $%
\int_{\Omega _{1}}u=0$ and hence for $u\in W_{\{a,b\}}^{1,(q,p)}(\Bbb{R}
_{*}^{N})$ such that $\int_{\Omega _{1}}u=0$ since, irrespective of $a$ and $%
b,$ the restrictions to $\Omega _{1}$ of functions in $W_{\{a,b\}}^{1,(q,p)}(%
\Bbb{R}_{*}^{N})$ are obviously in $W^{1,(q,p)}(\Omega _{1}).$

If $x\in \Omega _{1},$ then $|x|^{a},|x|^{b}$ are bounded below and $|x|^{c}$
is bounded above. Thus, after changing $C,$ (\ref{35}) yields $\int_{\Omega
_{1}}|x|^{c}|u|^{r}\leq C\left( \int_{\Omega _{1}}|x|^{b}|\nabla
u|^{p}\right) ^{\frac{\theta r}{p}}\left( \int_{\Omega
_{1}}|x|^{a}|u|^{q}\right) ^{\frac{(1-\theta )r}{q}}$ for every $u\in
W_{\{a,b\}}^{1,(q,p)}(\Bbb{R}_{*}^{N})$ such that $\int_{\Omega _{1}}u=0.$
By rescaling and using (ii), this implies, with the same $C$ independent of $%
\tau ,$ 
\begin{equation}
\int_{\Omega _{\tau }}|x|^{c}|u|^{r}\leq C\left( \int_{\Omega _{\tau
}}|x|^{b}|\nabla u|^{p}\right) ^{\frac{\theta r}{p}}\left( \int_{\Omega
_{\tau }}|x|^{a}|u|^{q}\right) ^{\frac{(1-\theta )r}{q}},  \label{36}
\end{equation}
for every $u\in W_{\{a,b\}}^{1,(q,p)}(\Bbb{R}_{*}^{N})$ such that $%
\int_{\Omega _{\tau }}u=0.$ In particular, (\ref{36}) holds for every $\tau
>0$ and every $u\in W_{0}$ defined in (\ref{33}).

It is also observed in \cite[p. 268]{CaKoNi84} that if $k\in \Bbb{Z}$ and $%
A_{k},B_{k}\geq 0$ and if $\alpha ,\beta \geq 0$ satisfy $\alpha +\beta \geq
1,$ then 
\begin{equation}
\sum_{k\in \Bbb{Z}}A_{k}^{\alpha }B_{k}^{\beta }\leq \left( \sum_{k\in \Bbb{Z%
}}A_{k}\right) ^{\alpha }\left( \sum_{k\in \Bbb{Z}}B_{k}\right) ^{\beta },
\label{37}
\end{equation}
where the first (second) factor on the right is $1$ when $\alpha =0$ ($\beta
=0$). Thus, when condition (iii) holds, (\ref{34}) follows from (\ref{37})
and (\ref{36}) with $\tau =2^{k},k\in \Bbb{Z}.$
\end{proof}

There is a clearer and more convenient formulation of Lemma \ref{lm19}:

\begin{corollary}
\label{cor20}Let $a,b,c\in \Bbb{R}$ and $1\leq p,q,r<\infty $ be given.%
\newline
(i) If $\frac{a+N}{q}\neq \frac{b-p+N}{p}$ and $c$ is in the closed interval
with endpoints $c^{0}$ and $c^{1},$ then $W_{0}\hookrightarrow L^{r}(\Bbb{R}%
^{N};|x|^{c}dx)$ if the following conditions hold (with $\theta _{c}$ given
by (\ref{2})):\newline
(i-{\scriptsize 1}) Either $r=q$ and $c=c^{0}$ ($=a$) or $c\neq c^{0}$ and $%
\theta _{c}\left( \frac{1}{p}-\frac{1}{N}-\frac{1}{q}\right) \leq \frac{1}{r}%
-\frac{1}{q}.$\newline
(i-{\scriptsize 2}) $\frac{\theta _{c}r}{p}+\frac{(1-\theta _{c})r}{q}\geq
1. $\newline
(ii) If $\frac{a+N}{q}=\frac{b-p+N}{p}$ and $c=c^{0}$ ($=c^{1}$), then $%
W_{0}\hookrightarrow L^{r}(\Bbb{R}^{N};|x|^{c}dx)$ if $\min \{p,q\}\leq $ $%
r\leq \max \{p^{*},q\}.$ Furthermore, there is a constant $C>0$ such that,
for every $u\in W_{0},$ 
\begin{equation}
||u||_{c,r}\leq C||\nabla u||_{b,p}\quad \text{ if }p\leq r\leq p^{*},
\label{38}
\end{equation}
\begin{multline}
||u||_{c,r}\leq C||\nabla u||_{b,p}^{\theta }||u||_{a,q}^{1-\theta }\text{
if }r=q\text{ or if }p\neq q\text{ and }  \label{39} \\
\min \{p,q\}\leq r\leq \max \{p,q\},
\end{multline}
where $\theta :=\frac{p(r-q)}{r(p-q)}$ if $p\neq q$ and $\theta =0$ if $%
p=q=r.$
\end{corollary}

\begin{proof}
(i) Suppose $\frac{a+N}{q}\neq \frac{b-p+N}{p}.$ By (\ref{4}), condition
(ii) of Lemma \ref{lm19} holds if and only if $\theta =\theta _{c}.$ If $r=q$
and $c=c^{0}=a,$ so that $\theta _{c^{0}}=0,$ condition (i) of Lemma \ref
{lm19} holds with any $\delta .$ On the other hand, if $c\neq c^{0},$ then $%
\theta _{c}\in (0,1]$ and condition (i) of Lemma \ref{lm19} holds with $%
\delta =\frac{b-p+N}{p}+\frac{1-\theta _{c}}{\theta _{c}}\frac{N}{q}-\frac{1%
}{\theta _{c}}\frac{N}{r}.$ Hence, $\delta \leq \frac{b}{p}$ -as required in
Lemma \ref{lm19}- if and only if $\theta _{c}\left( \frac{1}{p}-\frac{1}{N}-%
\frac{1}{q}\right) \leq \frac{1}{r}-\frac{1}{q}.$ Thus, $W_{0}%
\hookrightarrow L^{r}(\Bbb{R}^{N};|x|^{c}dx)$ if also $\frac{\theta _{c}r}{p}
+\frac{(1-\theta _{c})r}{q}\geq 1.$

(ii) Suppose $\frac{a+N}{q}=\frac{b-p+N}{p}$ and let $c=c^{0}=c^{1}.$ Then,
condition (ii) of Lemma \ref{lm19} holds with any $\theta \in [0,1].$ Thus,
it only remains to show that if $\min \{p,q\}\leq r\leq \max \{p^{*},q\},$
then $\delta \leq \frac{b}{p}$ and $\theta \in [0,1]$ can be chosen such
that $\frac{c}{r}=\theta \delta +(1-\theta )\frac{a}{q}$ and that $\frac{
\theta r}{p}+\frac{(1-\theta )r}{q}\geq 1.$ If so, all the requirements of
Lemma \ref{lm19} are satisfied, whence $W_{0}\hookrightarrow L^{r}(\Bbb{R}
^{N};|x|^{c}dx).$

Observe that $\min \{p,q\}\leq r\leq \max \{p^{*},q\}$ if and only if either 
$p\leq r\leq p^{*}$ or $p\neq q$ and $\min \{p,q\}\leq r\leq \max \{p,q\}$
(possibly both). If $p\leq $ $r\leq p^{*},$ we may choose $\delta =\frac{c}{r%
}=\frac{b}{p}+\frac{N}{p}-\frac{N}{r}-1\leq \frac{b}{p}$ (since $r\leq p^{*}$%
) and $\theta =1,$ so that $\frac{\theta r}{p}+\frac{(1-\theta )r}{q}=\frac{r%
}{p}\geq 1.$ Then, (\ref{38}) follows from (\ref{34}).

If now $p\neq q$ and $\min \{p,q\}\leq r\leq \max \{p,q\},$ let $\theta $ be
defined by $\frac{1}{r}=\frac{\theta }{p}+\frac{1-\theta }{q},$ i.e., $%
\theta =\frac{p(q-r)}{r(q-p)}.$ Obviously, $\frac{\theta r}{p}+\frac{
(1-\theta )r}{q}=1,$ but it must be checked that $\frac{c}{r}=\theta \delta
+(1-\theta )\frac{a}{q}$ for some $\delta \leq \frac{b}{p}.$ Since $c=c^{0}$
and $\frac{a+N}{q}=\frac{b-p+N}{p},$ a straightforward verification shows
that $\frac{c}{r}=\theta \delta +(1-\theta )\frac{a}{q}$ with $\delta =\frac{%
b}{p}-1.$ Thus, (\ref{39}) follows from (\ref{34}). Of course, (\ref{39})
remains true with $\theta =0$ if $p=q=r.$
\end{proof}

\smallskip While Corollary \ref{cor20} gives sufficient conditions for $%
W_{0}\hookrightarrow L^{r}(\Bbb{R}^{N};|x|^{c}dx),$ necessary and sufficient
ones for $W_{rad}\hookrightarrow L^{r}(\Bbb{R}^{N};|x|^{c}dx)$ are listed in
Theorem \ref{th16}, where $W_{rad}$ is the subspace of radially symmetric
functions in $W_{\{a,b\}}^{1,(q,p)}(\Bbb{R}_{*}^{N}).$ Thus, $%
W_{\{a,b\}}^{1,(q,p)}(\Bbb{R}_{*}^{N})\hookrightarrow L^{r}(\Bbb{R}
^{N};|x|^{c}dx)$ can be inferred from the remark that $W_{\{a,b\}}^{1,(q,p)}(%
\Bbb{R}_{*}^{N})=W_{rad}\oplus W_{0}$ together with the following obvious
lemma:

\begin{lemma}
\label{lm21} Let $X$ and $Y$ be normed spaces and let $X_{1}$ and $X_{2}$ be
two subspaces of $X$ such that $X=X_{1}\oplus X_{2}$ (topological direct
sum). Then, $X\hookrightarrow Y$ if and only if $X_{i}\hookrightarrow
Y,i=1,2.$
\end{lemma}

The relation $W_{\{a,b\}}^{1,(q,p)}(\Bbb{R}_{*}^{N})=W_{rad}\oplus W_{0}$
reflects the equality $u=u_{S}+(u-u_{S})$ with $u_{S}$ the radial
symmetrization of $u,$ that is, $u_{S}(x)=f_{u}(|x|)$ with $f_{u}$ given by (%
\ref{15}). Then, $u_{S}\in W_{rad}$ and $||u_{S}||_{\{a,b\},(q,p)}\leq
||u||_{\{a,b\},(q,p)}\leq ||u||_{a,q}+||\nabla u||_{b,p}$ by part (ii) of
Lemma \ref{lm7}, which proves the continuity of $u\mapsto u_{S}$ ($%
W_{\{a,b\}}^{1,(q,p)}(\Bbb{R}_{*}^{N})$ and $\widetilde{W}
_{\{a,b\}}^{1,(q,p)}$ contain the same radially symmetric functions and the
induced norms are the same). That $u-u_{S}\in W_{0}$ and $W_{rad}\cap
W_{0}=\{0\}$ is trivial.

The principle outlined above is simple, but it cannot always be implemented
in a straightforward way, primarily because the condition (i-{\scriptsize 2}%
) in Corollary \ref{cor20} is far from being necessary. The case when $%
r<\min \{p,q\}$ (Section \ref{embedding1}) is one, but not the only,
example. In practice, this means that Corollary \ref{cor20} alone does not
always suffice to prove that $W_{0}\hookrightarrow L^{r}(\Bbb{R}%
^{N};|x|^{c}dx)$ under optimal conditions about $c.$ Other arguments will be
needed, most notably Theorem \ref{th18} (but with other parameters); see the
proofs of Lemma \ref{lm23} and of Theorem \ref{th27}.

\section{Embedding theorem when $p<r\leq q$\label{embedding2}}

In this section, we discuss the embedding $W_{\{a,b\}}^{1,(q,p)}(\Bbb{R}
_{*}^{N})\hookrightarrow L^{r}(\Bbb{R}^{N};|x|^{c}dx)$ when $p<r\leq q.$
Together with Theorem \ref{th18} (when $1\leq r\leq \min \{p,q\}$), this
will settle the issue when $1\leq r\leq q.$

\begin{theorem}
\label{th22}Let $a,b,c\in \Bbb{R}$ and $1\leq p<r\leq q<\infty $ be given.
Then, $W_{\{a,b\}}^{1,(q,p)}(\Bbb{R}_{*}^{N})\hookrightarrow L^{r}(\Bbb{R}%
^{N};|x|^{c}dx)$ (and hence $W_{\{a,b\}}^{1,(q,p)}(\Bbb{R}%
_{*}^{N})\hookrightarrow W_{\{c,b\}}^{1,(r,p)}(\Bbb{R}_{*}^{N})$) in the
following cases: \newline
(i) $a$ and $b-p$ are on the same side of $-N$ (including $-N$), $\frac{a+N}{%
q}\neq \frac{b-p+N}{p},c$ is in the open interval with endpoints $c^{0}$ and 
$c^{1}$ and $\theta _{c}\left( \frac{1}{p}-\frac{1}{N}-\frac{1}{q}\right)
\leq \frac{1}{r}-\frac{1}{q}.$ \newline
(ii) $a$ and $b-p$ are strictly on opposite sides of $-N$ (hence $\frac{a+N}{%
q}\neq \frac{b-p+N}{p}$), $c$ is in the open interval with endpoints $c^{0}$
and $-N$ and $\theta _{c}\left( \frac{1}{p}-\frac{1}{N}-\frac{1}{q}\right)
\leq \frac{1}{r}-\frac{1}{q}.$ \newline
(iii) $r=q$ and $c=a.$ \newline
(iv) $r\leq p^{*},a\leq -N$ and $b-p<-N$ or $a\geq -N$ and $b-p>-N,c=c^{1}.$%
\newline
(v) $\frac{a+N}{q}=\frac{b-p+N}{p}\neq 0$ and $c=c^{1}$ ($=c^{0}$).
\end{theorem}

\subsection{Proof of parts (i) and (ii)}

In this subsection, we assume $\frac{a+N}{q}\neq \frac{b-p+N}{p}.$ Let $%
0\leq \bar{\theta}\leq 1$ denote the largest value of $\theta $ such that $%
\theta \left( \frac{1}{p}-\frac{1}{N}-\frac{1}{q}\right) \leq \frac{1}{r}-%
\frac{1}{q},$ that is, since $r\leq q$ is assumed, 
\begin{equation}
\bar{\theta}=\left\{ 
\begin{array}{l}
1\text{ if }r\leq p^{*},\text{ } \\ 
\left( \frac{1}{r}-\frac{1}{q}\right) \left( \frac{1}{p}-\frac{1}{N}-\frac{1%
}{q}\right) ^{-1}<1\text{ if }p<N\text{ and }r>p^{*}
\end{array}
\right. \text{ }  \label{40}
\end{equation}
and call $\bar{c}$ the corresponding value of $c,$ namely, 
\begin{equation}
\bar{c}:=\bar{\theta}c^{1}+(1-\bar{\theta})c^{0}  \label{41}
\end{equation}
(so that $\bar{\theta}=\theta _{\bar{c}};$ see (\ref{2})). Since $\frac{a+N}{
q}\neq \frac{b-p+N}{p},$ the points $c^{0}$ and $\bar{c}$ coincide if and
only if $\bar{\theta}=0,$ i.e., $r=q>p^{*},$ and then $\bar{c}=c^{0}=a$ by (%
\ref{1}).

\begin{lemma}
\label{lm23}If $\frac{a+N}{q}\neq \frac{b-p+N}{p}$ and $\bar{c}$ is given by
(\ref{40}) and (\ref{41}), the subspace $W_{0}$ of $W_{\{a,b\}}^{1,(q,p)}(%
\Bbb{R}_{*}^{N})$ in (\ref{33}) is continuously embedded into $L^{r}(\Bbb{R}%
^{N};|x|^{c}dx)$ for every $c$ in the interval $J$ with endpoints $\bar{c}$
(included) and $c^{0}$ (not included, unless $r=q$).
\end{lemma}

\begin{proof}
If $r=q,$ the embedding $W_{0}\hookrightarrow L^{r}(\Bbb{R}^{N};|x|^{c}dx)$
for $c\in J$ follows from part (i) of Corollary \ref{cor20} since $\theta
_{c}\left( \frac{1}{p}-\frac{1}{N}-\frac{1}{q}\right) \leq \frac{1}{r}-\frac{
1}{q}=0$ by definition of $J$ and $\frac{\theta _{c}r}{p}+(1-\theta
_{c})\geq 1$ irrespective of $\theta _{c}\in [0,1]$ since $r>p$ is assumed.

From now on, $r<q$ and $c^{0}\notin J.$ Observe that the set $\{c\in \Bbb{R}
:W_{0}\hookrightarrow L^{r}(\Bbb{R}^{N};|x|^{c}dx)\}$ is always an interval
(in this statement, $W_{0}$ may be replaced by any normed space of
measurable functions on $\Bbb{R}^{N}$). Thus, to prove that this interval
contains $J,$ it suffices to show that $W_{0}\hookrightarrow L^{r}(\Bbb{R}%
^{N};|x|^{c}dx)$ when $c=$ $\bar{c}$ and when $c\in J$ is arbitrarily close
to $c^{0}.$

The embedding $W_{0}\hookrightarrow L^{r}(\Bbb{R}^{N};|x|^{\bar{c}}dx)$
follows once again from part (i) of Corollary \ref{cor20} since $\bar{\theta}%
\left( \frac{1}{p}-\frac{1}{N}-\frac{1}{q}\right) \leq \frac{1}{r}-\frac{1}{q%
}$ by definition of $\bar{\theta}$ and $\frac{\bar{\theta}r}{p}+\frac{(1-%
\bar{\theta})r}{q}\geq 1$ by a simple calculation (obvious if $\bar{\theta}
=1;$ otherwise, use $p<N$ and $q>r>p^{*}$).

To complete the proof, assume that $c\in J$ is close to $c^{0},$ so that $%
\theta _{c}>0$ is small. If so, condition (i-{\scriptsize 2}) of Corollary 
\ref{cor20} fails when $r<q$ and this corollary cannot be used. Nonetheless,
we shall prove by another argument that $W_{\{a,b\}}^{1,(q,p)}(\Bbb{R}
_{*}^{N})\hookrightarrow L^{r}(\Bbb{R}^{N};|x|^{c}dx)$ in this case, a
stronger result than actually needed.

Define $\tilde{c}:=\frac{(b-p)(q-r)+a(r-p)}{q-p}$ and note that, by (\ref{4}%
), $\theta _{\tilde{c}}=\frac{p(q-r)}{r(q-p)}\in (0,1)$ (recall $p<r<q$), so
that $\tilde{c}\neq c^{0}.$ Both the open intervals with endpoints $c^{0}$
and $\tilde{c}\neq c^{0}$ or $\bar{c}\neq c^{0}$ consist of convex
combinations of $c^{0}$ and $c^{1}.$ Thus, they intersect along a nontrivial
open interval having $c^{0}$ as an endpoint. As a result, it suffices to
show that $W_{\{a,b\}}^{1,(q,p)}(\Bbb{R}_{*}^{N})\hookrightarrow L^{r}(\Bbb{R%
}^{N};|x|^{c}dx)$ for $c$ close enough to $c^{0}$ in the open interval $%
\widetilde{J}$ with endpoints $c^{0}$ and $\tilde{c}.$

Given any such $c,$ set $\sigma :=c-a\frac{r-p}{q-p}$ and $\gamma :=\frac{q-r%
}{q-p}\in (0,1).$ If $u\in W_{\{a,b\}}^{1,(q,p)}(\Bbb{R}_{*}^{N}),$ write $%
|x|^{c}|u|^{r}=|x|^{\sigma }|u|^{p\gamma }|x|^{c-\sigma }|u|^{r-p\gamma }$
and use H\"{o}lder's inequality to get 
\begin{equation}
\int_{\Bbb{R}^{N}}|x|^{c}|u|^{r}dx\leq \left( \int_{\Bbb{R}^{N}}|x|^{\frac{
\sigma }{\gamma }}|u|^{p}dx\right) ^{\gamma }\left( \int_{\Bbb{R}
^{N}}|x|^{a}|u|^{q}dx\right) ^{1-\gamma }.  \label{42}
\end{equation}

By parts (i) and (ii) of Theorem \ref{th18} with $c$ replaced by $d$ and $r$
replaced by $p$ (since $p=\min \{p,q\}$), there is a nonempty open interval $%
I$ with endpoint $d^{0}:=\frac{p(a+N)}{q}-N$ and second endpoint between $%
d^{0}$ and $d^{1}:=b-p$ (specifically, $b-p$ or $-N$), such that $%
W_{\{a,b\}}^{1,(q,p)}(\Bbb{R}_{*}^{N})\hookrightarrow L^{p}(\Bbb{R}%
^{N};|x|^{d}dx)$ when $d\in I.$

When $c$ is moved from $c^{0}$ to $\tilde{c},$ the point $d:=\frac{\sigma }{
\gamma }=\frac{c(q-p)-a(r-p)}{q-r}$ moves from $d^{0}$ to $b-p.$ Therefore, $%
d\in I$ for $c$ in some nonempty open subinterval $\widetilde{I}$ of $%
\widetilde{J}$ having $c^{0}$ as an endpoint. From the above, $%
W_{\{a,b\}}^{1,(q,p)}(\Bbb{R}_{*}^{N})\hookrightarrow L^{p}(\Bbb{R}%
^{N};|x|^{d}dx)$ when $c\in \widetilde{I}.$ By Corollary \ref{cor2}, this
embedding is accounted for by a multiplicative inequality of the type (\ref
{11}) (with $c$ replaced by $d$ and $r$ replaced by $p$), namely $%
||u||_{d,p}\leq C||\nabla u||_{b,p}^{\theta _{d}}||u||_{a,q}^{1-\theta _{d}}$
with $\theta _{d}:=\frac{d-d^{0}}{d^{1}-d^{0}}.$ Since $d=\frac{\sigma }{%
\gamma },$ the substitution into (\ref{42}) yields, when $c\in \widetilde{I}%
, $ the inequality $||u||_{c,r}\leq C||\nabla u||_{b,p}^{\nu
}||u||_{a,q}^{1-\nu }$ for $u\in W_{\{a,b\}}^{1,(q,p)}(\Bbb{R}_{*}^{N}),$
where $\nu =\frac{p\gamma \theta _{d}}{r}\in (0,1).$ In turn, this implies a
corresponding additive (i.e., embedding) inequality.
\end{proof}

\textit{Proof of part (i):} If $\bar{\theta}=0$ in (\ref{40}) (so that $%
r=q>p^{*}$), no $c$ in the open interval with endpoints $c^{0}=a$ and $c^{1}=%
\frac{q(b-p+N)}{p}-N$ satisfies $\theta _{c}\left( \frac{1}{p}-\frac{1}{N}-%
\frac{1}{q}\right) \leq \frac{1}{r}-\frac{1}{q}$ since $\theta _{c}>0$ and $%
\theta _{c}\leq \bar{\theta}=0$ are contradictory. Thus, there is nothing to
prove.

If $0<\bar{\theta}\leq 1,$ so that $\bar{c}\neq c^{0},$ Lemma \ref{lm23}
ensures that $W_{0}\hookrightarrow L^{r}(\Bbb{R}^{N};|x|^{c}dx)$ for $c$ in
the semi-open interval $J$ with endpoints $\bar{c}$ (included) and $c^{0}$
(not included, unless $r=q$). Meanwhile, by part (i) of Theorem \ref{th16}, $%
W_{rad}\hookrightarrow L^{r}(\Bbb{R}^{N};|x|^{c}dx)$ for $c$ in the open
interval with endpoints $c^{0}$ and $c^{1}$ (since $\breve{\theta}\leq 0$
when $r\leq q$). Thus, by Lemma \ref{lm21}, $W_{\{a,b\}}^{1,(q,p)}(\Bbb{R}
_{*}^{N})\hookrightarrow L^{r}(\Bbb{R}^{N};|x|^{c}dx)$ for $c$ in the
intersection of these two intervals. By definition of $\bar{\theta},$ this
intersection is the set of those $c$ in the open interval with endpoints $%
c^{0}$ and $c^{1}$ such that $\theta _{c}\left( \frac{1}{p}-\frac{1}{N}-%
\frac{1}{q}\right) \leq \frac{1}{r}-\frac{1}{q}.$

\textit{Proof of part (ii): }Once again, it is not restrictive to assume $0<%
\bar{\theta}\leq 1.$ By part (ii) of Theorem \ref{th16}, $%
W_{rad}\hookrightarrow L^{r}(\Bbb{R}^{N};|x|^{c}dx)$ for $c$ in the open
interval with endpoints $c^{0}$ and $-N$ (since $\breve{\theta}\leq 0$) and,
by Lemma \ref{lm23}, $W_{0}\hookrightarrow L^{r}(\Bbb{R}^{N};|x|^{c}dx)$ for 
$c$ in the semi-open interval with endpoints $c^{0}$ and $\bar{c}$ ($\neq
c^{0}$ since $\bar{\theta}>0$), including $\bar{c}$ but not $c^{0}.$ Hence,
by Lemma \ref{lm21}, $W_{\{a,b\}}^{1,(q,p)}(\Bbb{R}_{*}^{N})\hookrightarrow
L^{r}(\Bbb{R}^{N};|x|^{c}dx)$ for $c$ in the intersection of these two
intervals, which is the set of those $c$ in the open interval with endpoints 
$c^{0}$ and $-N$ such that $\theta _{c}\left( \frac{1}{p}-\frac{1}{N}-\frac{1%
}{q}\right) \leq \frac{1}{r}-\frac{1}{q}.$

\subsection{Proof of parts (iii), (iv) and (v)}

Since part (iii) is obvious, it only remains to prove (iv) and (v). If $%
\frac{a+N}{q}\neq \frac{b-p+N}{p}$ (so that $c^{1}\neq c^{0}$), the proof of
(iv) follows from Lemma \ref{lm21}, from part (iii) of Theorem \ref{th16}
and from part (i) of Corollary \ref{cor20} (recall $\theta _{c^{1}}=1$ and $%
p<r\leq p^{*}$). The use of part (ii) of Corollary \ref{cor20} instead of
part (i) yields (v), which in turn implies (iv) when $\frac{a+N}{q}=\frac{%
b-p+N}{p}.$

\section{Embedding theorem when $r>q\geq 1$ and $r\geq p$\label{embedding3}}

Throughout this section, we assume $r>q\geq 1$ and $r\geq p.$ If also ($p<N$
and) and $r>p^{*},$ it follows from Theorem \ref{th3} and part (i) of
Theorem \ref{th1} that $W_{\{a,b\}}^{1,(q,p)}(\Bbb{R}_{*}^{N})$ is not
continuously embedded into any $L^{r}(\Bbb{R}^{N};|x|^{c}dx).$ Thus, it is
not restrictive to confine attention to the case when $r\leq p^{*}.$

If $\frac{a+N}{q}\neq \frac{b-p+N}{p},$ the combination $r>q$ and $q<p^{*}$
(i.e., $\frac{1}{q}+\frac{1}{N}-\frac{1}{p}>0$) shows that the necessary
condition for the embedding $W_{\{a,b\}}^{1,(q,p)}(\Bbb{R}
_{*}^{N})\hookrightarrow L^{r}(\Bbb{R}^{N};|x|^{c}dx)$ given in part (i) of
Theorem \ref{th3} is $\theta _{c}\geq \bar{\theta}>0$ where 
\begin{equation}
\bar{\theta}=\left( \frac{1}{r}-\frac{1}{q}\right) \left( \frac{1}{p}-\frac{1%
}{N}-\frac{1}{q}\right) ^{-1}.  \label{43}
\end{equation}
This formula is the same as in (\ref{40}), but now $\bar{\theta}$ is the 
\emph{smallest} value of $\theta \in [0,1]$ such that $\theta \left( \frac{1%
}{p}-\frac{1}{N}-\frac{1}{q}\right) \leq \frac{1}{r}-\frac{1}{q}.$ Note that
indeed $\bar{\theta}\leq 1$ because $r\leq p^{*}$ (and $\bar{\theta}=1$ if
and only if $r=p^{*}$). Equivalently, $c$ must belong to the closed interval
with endpoints $\bar{c}:=\bar{\theta}c^{1}+(1-\bar{\theta})c^{0}$ (as in (%
\ref{41})) and $c^{1}.$

In addition, $p\leq r<\infty $ ensures that the subspace $W_{0}$ in (\ref{33}%
) is continuously embedded into $L^{r}(\Bbb{R}^{N};|x|^{c}dx)$ for $c$ in
the closed interval with endpoints $\bar{c}$ and $c^{1}.$ This follows from
part (i) of Corollary \ref{cor20} since $\frac{r\theta _{c}}{p}+\frac{%
r(1-\theta _{c})}{q}\geq 1$ irrespective of $\theta _{c}\in [0,1].$ We
record this result for future reference.

\begin{lemma}
\label{lm24}Let $a,b,c\in \Bbb{R}$ and $1\leq p,q,r<\infty ,$ be such that $%
r>q\geq 1,r\geq p$ and $\frac{a+N}{q}\neq \frac{b-p+N}{p}.$ If $\bar{c}$ is
given by (\ref{43}) and (\ref{41}), then $W_{0}\hookrightarrow L^{r}(\Bbb{R}%
^{N};|x|^{c}dx)$ for $c$ in the closed interval with endpoints $\bar{c}$ and 
$c^{1}.$
\end{lemma}

\begin{lemma}
\label{lm25}Let $a,b\in \Bbb{R}$ and $1\leq p,r<\infty ,1\leq q<r\leq p^{*}$
be such that $\frac{a+N}{q}\neq \frac{b-p+N}{p}.$ If $\breve{\theta}:=\left(
1-\frac{q}{r}\right) \left( \frac{q}{p^{\prime }}+1\right) ^{-1}$ and $\bar{%
\theta}$ is given by (\ref{43}), then $0<\breve{\theta}\leq \bar{\theta}.$
\end{lemma}

\begin{proof}
An explicit calculation (using $q<p^{*}$).
\end{proof}

\begin{theorem}
\label{th26} Let $a,b,c\in \Bbb{R}$ and $1\leq p,q,r<\infty $ be such that $%
1\leq q<r$ and $r\geq p.$ Then, $W_{\{a,b\}}^{1,(q,p)}(\Bbb{R}%
_{*}^{N})\hookrightarrow L^{r}(\Bbb{R}^{N};|x|^{c}dx)$ (and hence $%
W_{\{a,b\}}^{1,(q,p)}(\Bbb{R}_{*}^{N})\hookrightarrow W_{\{c,b\}}^{1,(r,p)}(%
\Bbb{R}_{*}^{N})$) in the following cases: \newline
(i) $a$ and $b-p$ are on the same side of $-N$ (including $-N$), $\frac{a+N}{%
q}\neq \frac{b-p+N}{p},c$ is in the open interval with endpoints $c^{0}$ and 
$c^{1}$ and $\theta _{c}\left( \frac{1}{p}-\frac{1}{N}-\frac{1}{q}\right)
\leq \frac{1}{r}-\frac{1}{q}.$ \newline
(ii) $a$ and $b-p$ are strictly on opposite sides of $-N$ (hence $\frac{a+N}{%
q}\neq \frac{b-p+N}{p}$), $c$ is in the open interval with endpoints $c^{0}$
and $-N$ and $\theta _{c}\left( \frac{1}{p}-\frac{1}{N}-\frac{1}{q}\right)
\leq \frac{1}{r}-\frac{1}{q}.$ \newline
(iii) $r\leq p^{*}$ and either $a\leq -N$ and $b-p<-N,$ or $a\geq -N$ and $%
b-p>-N,c=c^{1}.$\newline
(iv) $a=-N,b=p-N,q<r\leq p^{*}$ and $c=c^{0}$ ($=c^{1}=-N$).
\end{theorem}

\begin{proof}
(i) If $r>\max \{q,p^{*}\},$ the condition $\theta _{c}\left( \frac{1}{p}-%
\frac{1}{N}-\frac{1}{q}\right) \leq \frac{1}{r}-\frac{1}{q}$ never holds and
there is nothing to prove. Accordingly, assume $r\leq p^{*},$ so that $\bar{
\theta}\leq 1.$ Since $\breve{\theta}\leq \bar{\theta},$ it follows from
part (i) of Theorem \ref{th16} that $W_{rad}\hookrightarrow L^{r}(\Bbb{R}
^{N};|x|^{c}dx)$ for every $c$ in the semi-open interval $J$ with endpoints $%
\bar{c}$ (included) and $c^{1}$ (not included). Therefore, by Lemma \ref
{lm24} and Lemma \ref{lm21}, $W_{\{a,b\}}^{1,(q,p)}(\Bbb{R}
_{*}^{N})\hookrightarrow L^{r}(\Bbb{R}^{N};|x|^{c}dx)$ for $c\in J.$ By
definition of $\bar{\theta},$ it is plain that $J$ consists of those $c$ in
the open interval with endpoints $c^{0}$ and $c^{1}$ such that $\theta
_{c}\left( \frac{1}{p}-\frac{1}{N}-\frac{1}{q}\right) \leq \frac{1}{r}-\frac{%
1}{q}.$ \newline

(ii) As in (i), it is not restrictive to assume $r\leq p^{*}.$ Then, $\breve{
\theta}\leq \bar{\theta}$ by Lemma \ref{lm25} while $\theta _{c}\geq \bar{
\theta}$ for every $c$ satisfying the specified conditions. Thus, the result
follows from part (ii) of Theorem \ref{th16}, Lemma \ref{lm24} and Lemma \ref
{lm21}.

(iii) Use part (ii) of Corollary \ref{cor20}, part (iii) of Theorem \ref
{th16} and Lemma \ref{lm21}.

(iv) Use part (ii) of Corollary \ref{cor20}, part (v) of Theorem \ref{th16}
and Lemma \ref{lm21}.
\end{proof}

\section{Embedding theorem when $1\leq q<r<p$\label{embedding4}}

If $q<r<p,$ then $r<p^{*}.$ Thus, as in the previous section, $\bar{\theta}$
in (\ref{43}) is the smallest $\theta \in [0,1]$ such that $\theta \left( 
\frac{1}{p}-\frac{1}{N}-\frac{1}{q}\right) \leq \frac{1}{r}-\frac{1}{q}.$
Clearly, $\bar{\theta}\in (0,1).$

\begin{theorem}
\label{th27}Let $a,b,c\in \Bbb{R}$ and $1\leq q<r<p<\infty $ be given. Then, 
$W_{\{a,b\}}^{1,(q,p)}(\Bbb{R}_{*}^{N})\hookrightarrow L^{r}(\Bbb{R}%
^{N};|x|^{c}dx)$ (and hence $W_{\{a,b\}}^{1,(q,p)}(\Bbb{R}%
_{*}^{N})\hookrightarrow W_{\{c,b\}}^{1,(r,p)}(\Bbb{R}_{*}^{N})$) in the
following cases: \newline
(i) $a$ and $b-p$ are on the same side of $-N$ (including $-N$), $\frac{a+N}{%
q}\neq \frac{b-p+N}{p},c$ is in the open interval with endpoints $c^{0}$ and 
$c^{1}$ and $\theta _{c}\left( \frac{1}{p}-\frac{1}{N}-\frac{1}{q}\right)
\leq \frac{1}{r}-\frac{1}{q}.$ \newline
(ii) $a$ and $b-p$ are strictly on opposite sides of $-N$ (hence $\frac{a+N}{%
q}\neq \frac{b-p+N}{p}$), $c$ is in the open interval with endpoints $c^{0}$
and $-N$ and $\theta _{c}\left( \frac{1}{p}-\frac{1}{N}-\frac{1}{q}\right)
\leq \frac{1}{r}-\frac{1}{q}.$ \newline
(iii) $\frac{a+N}{q}=\frac{b-p+N}{p}\neq 0$ and $c=c^{0}=c^{1}.$\newline
(iv) $a=-N,b=p-N$ and $c=c^{0}$ ($=c^{1}=-N$).
\end{theorem}

\begin{proof}
(i) As in the proof of Lemma \ref{lm23}, let $\tilde{c}:=\frac{
(b-p)(q-r)+a(r-p)}{q-p},$ so that, by (\ref{2}), $\tilde{\theta}:=\theta _{%
\tilde{c}}=$ $\frac{p(q-r)}{r(q-p)}\in (0,1).$ If $c$ is in the semi-open
interval with endpoints $c^{0}$ (not included) and $\tilde{c}$ (included),
then $0<\theta _{c}\leq \tilde{\theta}.$ A routine verification reveals that
condition (i-{\scriptsize 2}) of Corollary \ref{cor20} holds but condition
(i-{\scriptsize 1}) holds if and only if $\theta _{c}\geq \bar{\theta}$ and $%
\tilde{\theta}>\bar{\theta}$ by another simple verification. Thus, by
Corollary \ref{cor20}, $W_{0}\hookrightarrow L^{r}(\Bbb{R}^{N};|x|^{c}dx)$
if $c$ is in the closed interval $K$ with endpoints $\bar{c}$ in (\ref{41})
and $\tilde{c}.$

By part (i) of Theorem \ref{th16}, $W_{rad}\hookrightarrow L^{r}(\Bbb{R}
^{N};|x|^{c}dx)$ if $c$ is in the semi-open interval with endpoints $\breve{c%
}:=\breve{\theta}c^{1}+(1-\breve{\theta})c^{0}$ (included) and $c^{1}$ (not
included) and, by Lemma \ref{lm25}, this interval contains $K.$ Thus, by
Lemma \ref{lm21}, it follows that $W_{\{a,b\}}^{1,(q,p)}(\Bbb{R}
_{*}^{N})\hookrightarrow L^{r}(\Bbb{R}^{N};|x|^{c}dx)$ when $c\in K.$

This is not yet the desired result, but $W_{\{a,b\}}^{1,(q,p)}(\Bbb{R}%
_{*}^{N})\hookrightarrow L^{r}(\Bbb{R}^{N};|x|^{\tilde{c}}dx)$ since $\tilde{%
c}\in K,$ so that $W_{\{a,b\}}^{1,(q,p)}(\Bbb{R}_{*}^{N})\hookrightarrow
W_{\{\tilde{c},b\}}^{1,(r,p)}(\Bbb{R}_{*}^{N}).$ Now, $\frac{\tilde{c}+N}{r}%
\neq \frac{b-p+N}{p}$ (because $\tilde{\theta}<1$) and $\tilde{c}$ and $b-p$
are on the same side of $-N$ (because the same thing is true of $a$ and $b-p$%
). Therefore, by part (i) of Theorem \ref{th18} with $a$ and $q$ replaced by 
$\tilde{c}$ and $r,$ respectively (use $r=\min \{p,r\}$), $W_{\{\tilde{c}
,b\}}^{1,(r,p)}(\Bbb{R}_{*}^{N})\hookrightarrow L^{r}(\Bbb{R}^{N};|x|^{c}dx)$
for $c$ in the open interval with endpoints $\tilde{c}$ and $c^{1}.$

Altogether, $W_{\{a,b\}}^{1,(q,p)}(\Bbb{R}_{*}^{N})\hookrightarrow L^{r}(%
\Bbb{R}^{N};|x|^{c}dx)$ for $c$ in the union of $K$ with the open interval
with endpoints $\tilde{c}$ and $c^{1},$ that is, the semi-open interval with
endpoints $\bar{c}$ and $c^{1}.$ By definition of $\bar{\theta},$ this
interval is the set of those $c$ in the open interval with endpoints $c^{0}$
and $c^{1}$ such that $\theta _{c}\left( \frac{1}{p}-\frac{1}{N}-\frac{1}{q}
\right) \leq \frac{1}{r}-\frac{1}{q}.$ \newline

(ii) If $\theta _{-N}\leq \bar{\theta},$ there is nothing to prove since no $%
c$ satisfies the required conditions. Suppose then $\theta _{-N}>\bar{\theta}
.$ As in the proof of (i) above, $W_{0}\hookrightarrow L^{r}(\Bbb{R}
^{N};|x|^{c}dx)$ if $c$ is in the (nontrivial) closed interval $K$ with
endpoints $\bar{c}$ and $\tilde{c}.$ On the other hand, since $\breve{\theta}%
\leq \bar{\theta}$ by Lemma \ref{lm25} and $\bar{\theta}<\theta _{-N},$ it
follows from part (ii) of Theorem \ref{th16} that $W_{rad}\hookrightarrow
L^{r}(\Bbb{R}^{N};|x|^{c}dx)$ if $c$ is in the semi-open interval $\breve{J}$
with endpoints $\breve{c}:=\breve{\theta}c^{1}+(1-\breve{\theta})c^{0}$
(included) and $-N$ (not included). Therefore, by Lemma \ref{lm21}, $%
W_{\{a,b\}}^{1,(q,p)}(\Bbb{R}_{*}^{N})\hookrightarrow L^{r}(\Bbb{R}
^{N};|x|^{c}dx)$ when $c\in K\cap \breve{J}.$

Since $\breve{\theta}\leq \bar{\theta}$ $<\theta _{-N},$ it follows that $%
\bar{c}\in K\cap \breve{J}$ is an endpoint of $K\cap \breve{J}.$ Since also $%
\bar{\theta}<\tilde{\theta},$ the second endpoint can only be $-N$ or $%
\tilde{c}.$ If $\theta _{-N}\leq \tilde{\theta},$ then $K\cap \breve{J}$ is
the semi-open interval with endpoints $\bar{c}$ (included) and $-N$ (not
included). If $\theta _{-N}>\tilde{\theta},$ then $K\cap \breve{J}$ is the
closed interval with endpoints $\bar{c}$ and $\tilde{c}.$ Yet, once again, $%
W_{\{a,b\}}^{1,(q,p)}(\Bbb{R}_{*}^{N})\hookrightarrow L^{r}(\Bbb{R}
^{N};|x|^{c}dx)$ when $c$ is in the semi-open interval with endpoints $\bar{c%
}$ (included) and $-N$ (not included), as shown below. This proves the
desired result since, by definition of $\bar{\theta},$ this interval
consists of those $c$ in the open interval with endpoints $c^{0}$ and $-N$
such that $\theta _{c}\left( \frac{1}{p}-\frac{1}{N}-\frac{1}{q}\right) \leq 
\frac{1}{r}-\frac{1}{q}.$

To complete the proof, note that, by (\ref{4}), $\theta _{-N}>\tilde{\theta}$
implies that $\frac{\tilde{c}+N}{r}$ and $\frac{a+N}{q},$ and hence also $%
\tilde{c}+N$ and $a+N,$ have the same (nonzero) sign, so that $\tilde{c}$
and $b-p$ are strictly on opposite sides of $-N.$ As in the proof of (i)
above, but now by part (ii) of Theorem \ref{th18} with $a$ and $q$ replaced
by $\tilde{c}$ and $r,$ respectively, it follows that $W_{\{a,b\}}^{1,(q,p)}(%
\Bbb{R}_{*}^{N})\hookrightarrow L^{r}(\Bbb{R}^{N};|x|^{c}dx)$ when $c$ is in
the union of the closed interval with endpoints $\bar{c}$ and $\tilde{c}$
with the open interval with endpoints $\tilde{c}$ and $-N,$ that is, the
semi-open interval with endpoints $\bar{c}$ (included) and $-N$ (not
included), as claimed.

(iii) Use part (iv) of Theorem \ref{th16}, part (ii) of Corollary \ref{cor20}
and Lemma \ref{lm21}.

(iv) The argument is the same as in the proof of part (iv) of Theorem \ref
{th26}.
\end{proof}

\section{Generalized CKN inequalities\label{CKNinequalities}}

If $W_{\{a,b\}}^{1,(q,p)}(\Bbb{R}_{*}^{N})\hookrightarrow L^{r}(\Bbb{R}%
^{N};|x|^{c}dx),$ then $r\leq \max \{p^{*},q\}$ by Theorem \ref{th3} and $c$
is in the closed interval with endpoints $c^{0}$ and $c^{1}$ by part (i) of
Theorem \ref{th1}. If, in addition $\frac{a+N}{q}\neq \frac{b-p+N}{p},$ it
was shown in Corollary \ref{cor2} that the embedding is accounted for by the
multiplicative inequality 
\begin{equation}
||u||_{c,r}\leq C||\nabla u||_{b,p}^{\theta _{c}}||u||_{a,q}^{1-\theta _{c}},
\label{46}
\end{equation}
with $\theta _{c}$ given by (\ref{2}). When $a,b,c>-N$ and $u\in
C_{0}^{\infty }(\Bbb{R}^{N}),$ such inequalities coincide with some of the
CKN inequalities proved in \cite{CaKoNi84}.

With no \textit{a priori} limitation about $a,b$ and $c,$ but when $%
p=q=r=2,c=\frac{a+b}{2}-1$ -so that $\theta _{c}=\frac{1}{2}$- and $u\in
C_{0}^{\infty }(\Bbb{R}_{*}^{N}),$ (\ref{46}) was recently obtained, by
variational methods, by Catrina and Costa \cite{CaCo09} (see also \cite{Co08}%
), with best constant $C.$ This does not imply (\ref{46}) for $u\in
W_{\{a,b\}}^{1,(2,2)}(\Bbb{R}_{*}^{N}),$ or that the best constant is the
same; see Subsection \ref{ex3}.

The CKN inequalities also incorporate the limiting case $\frac{a+N}{q}=\frac{%
b-p+N}{p}$ (when $\theta _{c}$ in (\ref{2}) is not defined). It is therefore
natural to ask whether the embedding $W_{\{a,b\}}^{1,(q,p)}(\Bbb{R}%
_{*}^{N})\hookrightarrow L^{r}(\Bbb{R}^{N};|x|^{c}dx)$ can be characterized
by similar multiplicative inequalities when $\frac{a+N}{q}=\frac{b-p+N}{p},$
so that $c=c^{0}$ ($=c^{1}$) is the only possible value.

The next lemma is, roughly speaking, a ``multiplicative'' analog of Lemma 
\ref{lm21}. Recall the definition (\ref{33}) of the subspace $W_{0}$ as well
as the shorthand $W_{rad}$ for the subspace of radially symmetric functions
of $W_{\{a,b\}}^{1,(q,p)}(\Bbb{R}_{*}^{N}).$

\begin{lemma}
\label{lm28}Let $a,b,c\in \Bbb{R}$ and $1\leq p,q,r<\infty $ be given. If
there is $\theta \in [0,1]$ such that $||u||_{c,r}\leq C||\nabla
u||_{b,p}^{\theta }||u||_{a,q}^{1-\theta }$ for every $u\in W_{rad}$ and
every $u\in W_{0},$ then $||u||_{c,r}\leq C||\nabla u||_{b,p}^{\theta
}||u||_{a,q}^{1-\theta }$ for every $u\in W_{\{a,b\}}^{1,(q,p)}(\Bbb{R}
_{*}^{N})$ after modifying $C.$
\end{lemma}

\begin{proof}
Let $u\in W_{\{a,b\}}^{1,(q,p)}(\Bbb{R}_{*}^{N})$ be given. Then $%
u=u_{S}+(u-u_{S}),$ where $u_{S}\in W_{rad}$ and $u-u_{S}\in W_{0}.$ By part
(ii) of Lemma \ref{lm7}, $||u_{S}||_{a,q}\leq ||u||_{a,q}$ and $||\partial
_{\rho }u_{S}||_{b,p}\leq ||\partial _{\rho }u||_{b,p}\leq ||\nabla
u||_{b,p}.$ Since $||\partial _{\rho }u_{S}||_{b,p}=||\nabla u_{S}||_{b,p},$
the inequality $||u_{S}||_{c,r}\leq $\linebreak $C||\nabla
u_{S}||_{b,p}^{\theta }||u_{S}||_{a,q}^{1-\theta }$ yields $%
||u_{S}||_{c,r}\leq C||\nabla u||_{b,p}^{\theta }||u||_{a,q}^{1-\theta }$
after modifying $C.$

Also, $||u-u_{S}||_{a,q}\leq ||u||_{a,q}+||u_{S}||_{a,q}\leq 2||u||_{a,q}$
and $||\nabla (u-u_{S})||_{b,p}\leq ||\nabla u||_{b,p}+||\nabla
u_{S}||_{b,p}\leq M||\nabla u||_{b,p}$ for some $M>0$ independent of $u.$
Thus, $||u-u_{S}||_{c,r}\leq C||\nabla u||_{b,p}^{\theta
}||u||_{a,q}^{1-\theta }$ after another modification of $C.$ As a result, $%
||u||_{c,r}\leq ||u_{S}||_{c,r}+||u-u_{S}||_{c,r}\leq 2C||\nabla
u||_{b,p}^{\theta }||u||_{a,q}^{1-\theta }.$
\end{proof}

\begin{theorem}
\label{th29} Let $a,b\in \Bbb{R}$ and $1\leq p,q,r<\infty $ be such that $%
\frac{a+N}{q}=\frac{b-p+N}{p}\neq 0$ and let $c=c^{0}=c^{1}.$\newline
(i) If $p\leq r\leq p^{*},$ there is a constant $C>0$ such that 
\begin{equation}
||u||_{c^{1},r}\leq C||\nabla u||_{b,p},\qquad \forall u\in
W_{\{a,b\}}^{1,(q,p)}(\Bbb{R}_{*}^{N}).  \label{47}
\end{equation}
(ii) If $r=p=q$ or if $p\neq q$ and $\min \{p,q\}\leq r\leq \max \{p,q\},$
there is a constant $C>0$ such that 
\begin{equation}
||u||_{c^{1},r}\leq C||\nabla u||_{b,p}^{\theta }||u||_{a,q}^{1-\theta
},\qquad \forall u\in W_{\{a,b\}}^{1,(q,p)}(\Bbb{R}_{*}^{N}),  \label{48}
\end{equation}
where $\theta =0$ if $r=p=q$ and $\theta =\frac{p(r-q)}{r(p-q)}$ if $p\neq
q. $
\end{theorem}

\begin{proof}
(i) Use Lemma \ref{lm28} together with the inequality (\ref{27}) in part
(iii) of Theorem \ref{th16} and the inequality (\ref{38}) in part (ii) of
Corollary \ref{cor20}.

(ii) Use Lemma \ref{lm28} together with the inequality (\ref{28}) in part
(iv) of Theorem \ref{th16} and the inequality (\ref{39}) in part (ii) of
Corollary \ref{cor20}.
\end{proof}

\begin{remark}
\label{rm5}\smallskip If $\frac{a+N}{q}=\frac{b-p+N}{p}\neq 0$ and $p\leq
r\leq \min \{p^{*},\max \{p,q\}\},$ (\ref{47}) and (\ref{48}) show that $%
||u||_{c^{1},r}\leq C||\nabla u||_{b,p}^{\theta }||u||_{a,q}^{1-\theta }$
for $u\in W_{\{a,b\}}^{1,(q,p)}(\Bbb{R}_{*}^{N})$ with $\theta =1$ and with $%
\theta =\underline{\theta },$ where $\underline{\theta }=\frac{p(r-q)}{r(p-q)%
}$ if $p\neq q$and $\underline{\theta }=0$ if $p=r=q.$ Hence, the inequality
holds with $\theta \in \left[ \underline{\theta },1\right] $ and so $\theta $
is not unique if $r>p\neq q$ or if $r=p=q.$ This is actually trivial if $%
r=q\geq p$ (because (\ref{48}) is trivial), but not in the other cases: $%
p<r<q\leq p^{*}$ or $p<N$ and $p<r\leq p^{*}<q.$
\end{remark}

Clearly, (\ref{47}) is an $N$-dimensional weighted Hardy-type inequality,
apparently new when $q\neq p.$ It is proved in \cite[p. 309]{OpKu90} when $%
q=p,$ so that $a=b-p\neq -N.$ When $u\in C_{0}^{\infty }(\Bbb{R}_{*}^{N}),$
it was obtained earlier by Gatto, Guti\'{e}rrez and Wheeden \cite{GaGuWh85},
who showed that $p\leq r\leq p^{*}$ is already necessary in that setting. A
number of special cases of (\ref{47}) for various classes of smooth
functions with compact support can be found in both the older and the recent
literature (\cite{GlMaGrTh76}, \cite{Li83}, \cite{SuWaWi07}, among others).
The inequality (\ref{48}), meaningless when $q=p,$ seems to be known only if 
$a,b,c>-N$ and $u\in C_{0}^{\infty }(\Bbb{R}^{N}),$ when it is one of the
CKN inequalities.

By Corollary \ref{cor2}, the inequality (sharper than (\ref{46})) 
\begin{equation*}
||u||_{c,r}\leq C||\partial _{\rho }u||_{b,p}^{\theta
_{c}}||u||_{a,q}^{1-\theta _{c}},\qquad \forall u\in \widetilde{W}%
_{\{a,b\}}^{(q,p)},
\end{equation*}
holds if $\frac{a+N}{q}\neq \frac{b-p+N}{p},c$ is in the closed interval
with endpoints $c^{0}$ and $c^{1}$ and $\widetilde{W}_{\{a,b\}}^{(q,p)}%
\hookrightarrow L^{r}(\Bbb{R}^{N};|x|^{c}dx).$ Necessary and sufficient
conditions for this embedding were given in Theorem \ref{th18} when $r\leq
\min \{p,q\},$ where it is also shown that $\widetilde{W}_{\{a,b\}}^{(q,p)}%
\hookrightarrow L^{r}(\Bbb{R}^{N};|x|^{c}dx)$ if $\frac{a+N}{q}=\frac{b-p+N}{
p}\neq 0,r=p$ ($\leq q$) and $c=c^{0}=c^{1}.$ If so, it follows from part
(iii) of Theorem \ref{th16} and from Lemma \ref{lm17} that 
\begin{equation*}
||u||_{b-p,p}\leq C||\partial _{\rho }u||_{b,p},\qquad \forall u\in 
\widetilde{W}_{\{a,b\}}^{(q,p)}.
\end{equation*}

The only case when the embedding $W_{\{a,b\}}^{1,(q,p)}(\Bbb{R}
_{*}^{N})\hookrightarrow L^{r}(\Bbb{R}^{N};|x|^{c}dx)$ is true but \emph{not}
equivalent to a multiplicative inequality arises in part (vi) of Theorem \ref
{th0} when $N\geq 2$ (if $u$ is radially symmetric, or $N=1,$ see (\ref{29}):

\begin{theorem}
\label{th30}If $q<r\leq p^{*},$ then $W_{\{-N,p-N\}}^{1,(q,p)}(\Bbb{R}%
_{*}^{N})\hookrightarrow L^{r}(\Bbb{R}^{N};|x|^{-N}dx)$ but when $N\geq 2,$
the inequality $||u||_{-N,r}\leq C||\nabla u||_{p-N,p}^{\theta
}||u||_{-N,q}^{1-\theta }$ fails to hold for every $C>0$ and every $\theta
\in [0,1].$
\end{theorem}

\begin{proof}
The embedding is part (vi) of Theorem \ref{th0}. Also, the inequality can
only hold if $\theta =\breve{\theta}:=\left( 1-\frac{q}{r}\right) \left( 
\frac{q}{p^{\prime }}+1\right) ^{-1}.$ This follows by choosing $u(x)=g(\ln
|x|)$ with $g\in C_{0}^{\infty }(\Bbb{R})$ and by reversing the steps of the
proof of part (v) of Theorem \ref{th16} (by \cite{CaKoNi84}, (\ref{31})
cannot hold with $\theta \neq \breve{\theta}$ when $g\in C_{0}^{\infty }(%
\Bbb{R})$ is arbitrary).

Next, if $||u||_{-N,r}\leq C||\nabla u||_{p-N,p}^{\theta
}||u||_{-N,q}^{1-\theta },$ the method of proof of Theorem \ref{th3} with $%
a=c=-N$ and $b=p-N$ shows that $\theta \left( \frac{1}{p}-\frac{1}{N}-\frac{1%
}{q}\right) \leq \frac{1}{r}-\frac{1}{q}.$ Upon substituting the only
possible value $\theta =\breve{\theta},$ a short calculation yields $%
q(N-1)\geq r(N-1).$ If $N\geq 2,$ this implies $q\geq r,$ which contradicts $%
q<r\leq p^{*}.$

Consistent with Theorem \ref{th30} and its proof, it is easily verified that
when $a=b-p=c=-N,\theta =\breve{\theta}$ and $N\geq 2,$ Lemma \ref{lm28} is
not applicable because condition (i) of Lemma \ref{lm19} fails, so that (\ref
{34}) cannot be used.
\end{proof}

\section{Examples\label{examples}}

\subsection{Embedding of unweighted spaces into $L^{r}(\Bbb{R}
^{N};|x|^{c}dx) $\label{ex1}}

We spell out the special case of Theorem \ref{th0} when $a=b=0.$ It is
noteworthy that $W^{1,(q,p)}(\Bbb{R}_{*}^{N})=W^{1,(q,p)}(\Bbb{R}%
^{N})=\{u\in L^{q}(\Bbb{R}^{N}):\nabla u\in (L^{p}(\Bbb{R}^{N}))^{N}\}$ if $%
N\geq 2,$ with the same norm (see Remark \ref{rm6} later). At any rate, if $%
a=b=0,$ then $\theta _{c}$ in (\ref{2}) is defined if and only if $\frac{1}{p%
}-\frac{1}{N}-\frac{1}{q}\neq 0,$ i.e., $q\neq p^{*}$ and then $\theta
_{c}=\left( \frac{c+N}{rN}-\frac{1}{q}\right) \left( \frac{1}{p}-\frac{1}{N}-%
\frac{1}{q}\right) ^{-1}.$ Therefore, the condition $\theta _{c}\left( \frac{
1}{p}-\frac{1}{N}-\frac{1}{q}\right) \leq \frac{1}{r}-\frac{1}{q}$ in parts
(i) and (ii) of Theorem \ref{th0} is just $c\leq 0.$ It follows that $%
W^{1,(q,p)}(\Bbb{R}_{*}^{N})\hookrightarrow L^{r}(\Bbb{R}^{N};|x|^{c}dx)$ if
and only if $r\leq \max \{p^{*},q\}$ and one of the following conditions
holds:

(i) $p\leq N,q\neq p^{*}$ and $c\leq 0$ is in the open interval with
endpoints $\frac{rN}{q}-N$ and $r\left( \frac{N-p}{p}\right) -N$ (a nonempty
set if $r<\max \{p^{*},q\}$).

(ii) $p>N$ and either $r\leq q$ and $-N<c<\frac{rN}{q}-N$ ($\leq 0$) or $r>q$
and $-N<c\leq 0.$

(iii) $r=q$ and $c=0.$

(iv) $p<N,p\leq r\leq p^{*}$ and $c=r\left( \frac{N-p}{p}\right) -N$ ($\leq
0 $ since $r\leq p^{*}$).

Since $\frac{N}{q}=\frac{N}{p}-1$ implies $p<N$ and $q>p,$ part (v) of
Theorem \ref{th0} coincides with (iv) above. Part (vi) of Theorem \ref{th0}
is not applicable.

If $r\leq \min \{p,q\},$ the conditions (i)-(iv) are necessary and
sufficient for $\widetilde{W}^{1,(q,p)}\hookrightarrow L^{r}(\Bbb{R}%
^{N};|x|^{c}dx),$ where $\widetilde{W}^{1,(q,p)}:=\widetilde{W}
_{\{0,0\}}^{1,(q,p)}$ is unweighted (Theorem \ref{th18}) an they take the
simpler form (i) $p\leq N,q\neq p^{*}$ and $c$ is in the open interval with
endpoints $\frac{rN}{q}-N$ and $r\left( \frac{N-p}{p}\right) -N$ (hence $c<0$%
). (ii) $p>N$ and $-N<c<\frac{rN}{q}-N$ ($\leq 0$). (iii) $q\leq p,r=q$ and $%
c=0.$ (iv) $p\leq q,p<N,r=p$ and $c=-p.$

When $c=0,$ the conditions become (i) $p<N$ and $r$ is in the closed
interval with endpoints $p^{*}$ and $q$ or (ii) $p\geq N$ and $r\geq q.$
This is of course well-known, especially when $p=q.$

\begin{remark}
\label{rm6}That $W^{1,(q,p)}(\Bbb{R}_{*}^{N})=W^{1,(q,p)}(\Bbb{R}^{N})$ with
the same norm if $N>1$ can be seen as follows: First, it suffices to show
that if $u\in W^{1,(q,p)}(\Bbb{R}_{*}^{N})$ has bounded support, then $u\in
W^{1,(q,p)}(\Bbb{R}^{N})$ with the same norm. Now, if $u\in W^{1,(q,p)}(\Bbb{%
R}_{*}^{N})$ has bounded support, then $u\in W^{1,\min \{p,q\}}(\Bbb{R}%
_{*}^{N})=W^{1,\min \{p,q\}}(\Bbb{R}^{N}),$ for example by \cite[p. 52]
{HeKiMa93}. Thus, as a distribution on $\Bbb{R}^{N},\nabla u$ is a function,
so that its restriction to $\Bbb{R}_{*}^{N}$ coincides with $\nabla u$ as a
distribution on $\Bbb{R}_{*}^{N}.$ Since the latter is in $(L^{q}(\Bbb{R}%
^{N}))^{N},$ the same thing is true of the former, which proves the claim.
\end{remark}

\subsection{Embedding of weighted spaces into $L^{r}(\Bbb{R}^{N})$\label{ex2}
}

The necessary and sufficient conditions for $W_{\{a,b\}}^{1,(q,p)}(\Bbb{R}
_{*}^{N})\hookrightarrow L^{r}(\Bbb{R}^{N})$ are given by Theorem \ref{th0}
with $c=0.$ If so, $\theta _{0}=\left( \frac{N}{r}-\frac{a+N}{q}\right)
\left( \frac{b-p+N}{p}-\frac{a+N}{q}\right) ^{-1}$ in (\ref{2}) when $\frac{%
a+N}{q}\neq \frac{b-p+N}{p}$ and these conditions become (after some work) $%
r\leq \max \{p^{*},q\}$ and

(i)-(ii) Either $-N\leq a<N\left( \frac{q}{r}-1\right) ,b>p+N\left( \frac{p}{%
r}-1\right) $and $a\left( \frac{p}{r}-1+\frac{p}{N}\right) \leq b\left( 
\frac{q}{r}-1\right) ,$ or $a>N\left( \frac{q}{r}-1\right) ,b<p+N\left( 
\frac{p}{r}-1\right) $ and $a\left( \frac{p}{r}-1+\frac{p}{N}\right) \geq
b\left( \frac{q}{r}-1\right) .$

(iii) $r=q$ and $a=0.$

(iv) $p\leq r\leq p^{*},b=p+N\left( \frac{p}{r}-1\right) $ ($\leq p$) and $%
a\geq -N.$

(v) $r\geq \min \{p,q\},a=N\left( \frac{q}{r}-1\right) $ and $b=p+N\left( 
\frac{p}{r}-1\right) .$

In (i)-(ii) above, the condition $\theta _{0}\left( \frac{1}{p}-\frac{1}{N}-%
\frac{1}{q}\right) \leq \frac{1}{r}-\frac{1}{q}$ is accounted for by $%
a\left( \frac{p}{r}-1+\frac{p}{N}\right) \leq b\left( \frac{q}{r}-1\right) $
or its reverse, as the case may be. By Remark \ref{rm1}, this condition
holds if $r\leq \min \{p^{*},q\},$ which of course is corroborated by a
direct verification.

\subsection{Embedding when $p=q$\label{ex3}}

If $p=q,$ then $r\leq \max \{p^{*},q\}$ is simply $r\leq p^{*}$ and $\frac{
a+N}{q}\neq \frac{b-p+N}{p}$ if and only if $a\neq b-p.$ The condition $%
\theta _{c}\left( \frac{1}{p}-\frac{1}{N}-\frac{1}{q}\right) \leq \frac{1}{r}
-\frac{1}{q}$ in parts (i) and (ii) of Theorem \ref{th0} becomes $\theta
_{c}\geq \frac{N}{p}-\frac{N}{r},$ which is not a restriction when $r\leq p.$
Also, part (v) is now a special case of part (iv).

If, in addition, $p=q=r,$ Theorem \ref{th18} is applicable to the larger
space $\widetilde{W}_{\{a,b\}}^{1,(p,p)}.$ Furthermore, $c^{0}=a$ and $%
c^{1}=b-p$ and so $\widetilde{W}_{\{a,b\}}^{1,(p,p)}\hookrightarrow L^{p}(%
\Bbb{R}^{N};|x|^{c}dx)$ if and only if either (i) $a$ and $b-p$ are on the
same side of $-N,$ not both equal to $-N,$ and $c$ is in the closed interval
with endpoints $a$ and $b-p$ or (ii) $a$ and $b-p$ are strictly on opposite
sides of $-N$ and $c$ is in the semi-open interval with endpoints $a$
(included) and $-N$ (not included). These are also necessary and sufficient
conditions for $W_{\{a,b\}}^{1,(p,p)}(\Bbb{R}_{*}^{N})\hookrightarrow L^{p}(%
\Bbb{R}^{N};|x|^{c}dx).$

When $p=q=r=2$ and $c=\frac{a+b}{2}-1,$ it follows from \cite{CaCo09} that $%
C_{0}^{\infty }(\Bbb{R}_{*}^{N})\hookrightarrow L^{2}(\Bbb{R}
^{N};|x|^{c}dx), $ unless $a=b-2=-N.$ If (for example) $a<-N$ and $%
b>-a+2-2N, $ then $b-2>-N,$ so that $a$ and $b-2$ are on opposite sides of $%
-N$ but, since $c=\frac{a+b}{2}-1>-N,$ condition (ii) above does not hold
and so $W_{\{a,b\}}^{1,(2,2)}(\Bbb{R}_{*}^{N})$ is not continuously embedded
into $L^{2}(\Bbb{R}^{N};|x|^{c}dx).$ This shows that $C_{0}^{\infty }(\Bbb{R}
_{*}^{N})$ is not dense in $W_{\{a,b\}}^{1,(2,2)}(\Bbb{R}_{*}^{N}).$
Accordingly, in general, embedding (or other) inequalities for $%
W_{\{a,b\}}^{1,(p,q)}(\Bbb{R}_{*}^{N})$ cannot be proved by confining attention to $%
C_{0}^{\infty }(\Bbb{R}_{*}^{N}).$

\subsection{A generalization\label{ex4}}

Let $B\subset \Bbb{R}^{N}$ be an open ball centered at the origin. If the
space $\{u\in W_{\{a,b\}}^{1,(q,p)}(\Bbb{R}_{*}^{N}):\limfunc{Supp}u\subset 
\overline{B}\}$ is continuously embedded into $L^{r}(\Bbb{R}^{N};|x|^{c}dx),$
it is also continuously embedded into $L^{r}(\Bbb{R}^{N};|x|^{d}dx)$ when $%
d\geq c.$ Likewise, if $\{u\in W_{\{a,b\}}^{1,(q,p)}(\Bbb{R}_{*}^{N}):%
\limfunc{Supp}u\subset \Bbb{R}^{N}\backslash B\}$ is continuously embedded
into $L^{r}(\Bbb{R}^{N};|x|^{c}dx),$ it is also continuously embedded into $%
L^{r}(\Bbb{R}^{N};|x|^{d}dx)$ when $d\leq c.$

With this remark and a cut-off argument, Theorem \ref{th0} can be extended
to more general weighted spaces. Let $x_{1},...,x_{k}\in \Bbb{R}^{N}$ be
distinct points and let $a_{1},...,a_{k},a_{\infty
},b_{1},...,b_{k},b_{\infty }\in \Bbb{R}$ and $1\leq r\leq p,q<\infty $ be
given. For $a,b\in \Bbb{R},$ call $J(a,b):=\{c\in \Bbb{R}:W_{\{a,b%
\}}^{1,(q,p)}(\Bbb{R}_{*}^{N})\hookrightarrow L^{r}(\Bbb{R}^{N};|x|^{c}dx)\}$
the interval of admissible $c$ characterized in Theorem \ref{th0}, with
endpoints $c_{-}(a,b)\leq c_{+}(a,b)$ and let $c_{1},...,c_{k},c_{\infty }$
be such that $c_{i}>c_{-}(a_{i},b_{i}),1\leq i\leq k$ and $c_{\infty
}<c_{+}(a_{\infty },b_{\infty })$ (the endpoints may be included if they are
in the admissible interval). If $w_{a},w_{b}$ and\footnote{%
Here, $a,b$ and $c$ are just indices.} $w_{c}$ are positive weights on $\Bbb{%
R}^{N}\backslash \{x_{1},...,x_{k}\}$ such that $%
w_{a}(x)=|x-x_{i}|^{a_{i}},w_{b}(x)=|x-x_{i}|^{b_{i}},w_{c}(x)=|x-x_{i}|^{c_{i}} 
$ for $x$ near $x_{i},i=1,...,k$ and $w_{a}(x)=|x|^{a_{\infty
}},w_{b}(x)=|x|^{b_{\infty }},w_{c}(x)=|x|^{c_{\infty }}$ for large $|x|,$
then the space $W_{\{w_{a},w_{b}\}}^{1,(q,p)}(\Bbb{R}^{N}\backslash
\{x_{1},...,x_{k}\}):=$\linebreak $\{u\in L_{loc}^{1}(\Bbb{R}^{N}\backslash
\{x_{1},...,x_{k}\}):u\in L^{q}(\Bbb{R}^{N};w_{a}(x)dx),\nabla u\in (L^{q}(%
\Bbb{R}^{N};w_{b}(x)dx))^{N}\}$ is continuously embedded into $L^{r}(\Bbb{R}%
^{N};w_{c}(x)dx).$

A somewhat heuristic but compelling reason why such conditions should be
optimal is simple: As pointed out above, the membership to $L^{r}(\Bbb{R}%
^{N};|x|^{c}dx)$ of functions with support in a closed ball $\overline{B}$
about the origin is unaffected by increasing $c.$ Thus, the value of the
upper end $c_{+}(a,b)$ can only be dictated by the behavior of functions
with support bounded away from $0.$ The optimality of the lower end $%
c_{-}(a,b)$ is justified by a similar argument. However, this rationale is
meaningless when $J(a,b)=\emptyset .$ If so, the simplest way around the
difficulty is to rely on the related fact that for functions with support in 
$\overline{B},$ membership to $W_{\{a,b\}}^{1,(q,p)}(\Bbb{R}_{*}^{N})$ is
unaffected by increasing $a$ or $b,$ so that doing so until $J(a,b)$ becomes
nonempty can be used to define $c_{-}(a,b).$ Likewise, $a$ or $b$ can be
decreased to define $c_{+}(a,b).$ This may or may not produce the best
possible conditions. Due to space limitations, a more detailed investigation
of the optimality issue by more sophisticated procedures (elaboration on
Remark \ref{rm4}) will not be attempted here.

Naturally, the weights need only to ``look like'' (not coincide with) power
weights in the vicinity of the points $x_{i}$ (or infinity). This remark
clarifies two things. First, $w_{a},w_{b}$ and $w_{c}$ need actually not
have power-like singularities at the same points: This case is reduced to
the previous one by adding points as needed and setting the corresponding $%
a_{i},b_{i}$ or $c_{i}$ equal to $0.$ Next, the cut-off argument is
technically simplified, and nothing is changed, if it is assumed that $%
w_{a}(x)=|x-x_{1}|^{a_{\infty }},w_{b}(x)=|x-x_{1}|^{b_{\infty
}},w_{c}(x)=|x-x_{1}|^{c_{\infty }}$ for large $|x|$ (otherwise, the origin
plays a technical role even when it is not one of the points $x_{i}$).
Theorem \ref{th0} is recovered when $k=1,x_{1}=0$ and $a_{1}=a_{\infty
},b_{1}=b_{\infty },c_{1}=c_{\infty }.$

If only $k=1$ and $x_{1}=0,$ Theorem \ref{th16} too can be generalized to
obtain the embedding of the subspace $W_{rad}$ of radially symmetric
functions in $W_{\{w_{a},w_{b}\}}^{1,(q,p)}(\Bbb{R}_{*}^{N})$ into $L^{r}(%
\Bbb{R}^{N};w_{c}(x)dx)$ under the conditions $%
c_{1}>c_{-}^{rad}(a_{1},b_{1}) $ and $c_{\infty }<c_{+}^{rad}(a_{\infty
},b_{\infty }),$ where $c_{\pm }^{rad}(a,b)$ denote the endpoints of the
admissible interval in Theorem \ref{th16}. Once again, $%
c_{-}^{rad}(a_{1},b_{1})$ and $c_{+}^{rad}(a_{\infty },b_{\infty })$ may be
included if they are in the admissible interval and they can also be defined
when the admissible interval is empty.

When $1<p=q<N$ and $w_{b}=1$ (so that $b_{1}=b_{\infty }=0$), the embedding
into $L^{r}(\Bbb{R}^{N};w_{c}(x)dx)$ of the closure $C_{rad}$ of the space
of radially symmetric functions in $C_{0}^{\infty }(\Bbb{R}^{N})\cap L^{p}(%
\Bbb{R}^{N};w_{a}(x)dx)$ equipped with the $W_{\{w_{a},1\}}^{1,(p,p)}(\Bbb{R}
_{*}^{N})$ norm, has recently been investigated by Su \textit{et al.} 
\cite[Theorems 1 and 2]{SuWaWi07}. They assume that $a_{1},c_{1},a_{\infty
},c_{\infty }$ are given and find the admissible values of $r$ under the
implicit assumption $r\geq p.$ The reformulation in terms of lower (upper)
bounds about $c_{1}$ ($c_{\infty }$) given $a_{1},a_{\infty }$ and $r$ is
conceptually trivial, but quite messy and technical in practice.
Accordingly, we shall not elaborate beyond the remark that, because $C_{rad}$
is usually \emph{smaller} than $W_{rad},$ the embedding may be true under
conditions more general than $c_{1}\geq c_{-}^{rad}(a_{1},1)$ and $c_{\infty
}\leq c_{+}^{rad}(a_{\infty },1).$ On the other hand, the case $0<r<p$ and
all others ($p=1,p\geq N,q\neq p,b_{1}\neq 0,b_{\infty }\neq 0$) can be
handled by the method outlined above.

\end{document}